\documentclass[11pt]{amsart} 
\usepackage{amscd,amssymb,amsxtra}
\usepackage{mathrsfs}  % \mathrsfs{A}
\DeclareSymbolFontAlphabet{\mathrsfs}{rsfs}
\usepackage[mathscr]{eucal} 
\usepackage{mathabx}
\usepackage{comment}
\usepackage{color}
\usepackage{enumitem} 
\usepackage{marginnote}

\makeatletter
\let\@secnumfont\bfseries
\def\section{\@startsection{section}{1}%
  \z@{4\linespacing\@plus\linespacing}{\linespacing}%
  {\bfseries\centering}}
\def\introsection{\@startsection{section}{1}%
  \z@{3\linespacing\@plus\linespacing}{\linespacing}%
  {\bfseries\centering}}
\def\subsection{\@startsection{subsection}{2}%
   \z@{1.25\linespacing\@plus.7\linespacing}{.5\linespacing}%
   {\normalfont\bfseries}}
\def\subsectionsinline{\def\subsection{\@startsection{subsection}{2}%
  \z@{1\linespacing\@plus.7\linespacing}{-.5em}%
  {\normalfont\bfseries}}}

\makeatother

\theoremstyle{definition}
\newtheorem{definition}[equation]{Definition}
\newtheorem{example}[equation]{Example}

\newtheorem*{definition*}{Definition}
\newtheorem*{example*}{Example}
\newtheorem*{problem*}{\color{blue}Problem}
\newtheorem*{exercise*}{Exercise}
\newtheorem*{question*}{\color{blue}Question}
\newtheorem*{project*}{\color{blue}Project}
\newtheorem*{construction*}{Construction}

\theoremstyle{remark}

\newtheorem{remark}[equation]{Remark}

\newtheorem*{note*}{Note}
\newtheorem*{notation*}{Notation}
\newtheorem*{remark*}{Remark}
\newtheorem*{data*}{Data}

\theoremstyle{plain}
\newtheorem{theorem}[equation]{Theorem}

\newtheorem{lemma}[equation]{Lemma}
\newtheorem{proposition}[equation]{Proposition}
\newtheorem{conjecture}[equation]{Conjecture}

\newtheorem*{theorem*}{Theorem}
\newtheorem*{corollary*}{Corollary}
\newtheorem*{lemma*}{Lemma}
\newtheorem*{proposition*}{Proposition}
\newtheorem*{conjecture*}{Conjecture}
\newtheorem*{claim*}{Claim}
\newtheorem*{proposal*}{Proposal}
\newtheorem*{conclusion*}{Conclusion}
\newtheorem*{hypothesis*}{Hypothesis}
\newtheorem*{assumption*}{Assumption}

\newenvironment{proof*}[1][\proofname]{
  \begin{proof}[#1]}{  
\end{proof}}

\numberwithin{equation}{section}

\definecolor{refkey}{rgb}{0,.6,.4}

\renewcommand{\:}{\colon}
\renewcommand{\AA}{{\mathbb A}}

\DeclareMathOperator{\Aut}{Aut}
\newcommand{\CC}{{\mathbb C}}

\DeclareMathOperator{\End}{End}

\DeclareMathOperator{\Hom}{Hom}
\DeclareMathOperator{\id}{id}
\DeclareMathOperator{\Ker}{Ker}
\DeclareMathOperator{\Map}{Map}

\DeclareMathOperator{\pt}{pt}
\newcommand{\QQ}{{\mathbb Q}}

\newcommand{\RR}{{\mathbb R}}
\newcommand{\TT}{\mathbb T}

\newcommand{\ZZ}{{\mathbb Z}}

\newcommand{\chiup}{\raise.5ex\hbox{$\chi$}}
\newcommand{\cir}{S^1}

\newcommand{\inv}{^{-1}}
\DeclareRobustCommand{\mstrut}{^{\vphantom{1*\prime y\vee M}}}

\newcommand{\Res}[1]{\negmedspace\Bigm|\mstrut_{#1}}
\newcommand{\temsquare}{\raise3.5pt\hbox{\boxed{ }}}

\newcommand{\zmod}[1]{\ZZ/#1\ZZ}

\usepackage[all,2cell]{xy}\renewcommand{\cir}{\ensuremath{S^1}}
\UseAllTwocells
\usepackage{graphicx}
\usepackage[colorlinks,backref=page,citecolor=refkey]{hyperref}

\definecolor{refkey}{rgb}{0,.8,.2}\definecolor{labelkey}{rgb}{1,0,0} 

\DeclareFontFamily{OT1}{pzc}{}
\DeclareFontShape{OT1}{pzc}{m}{it}{<-> s * [1.100] pzcmi7t}{}
\DeclareMathAlphabet{\mathpzc}{OT1}{pzc}{m}{it}

\DeclareMathOperator{\Bord}{Bord}
\DeclareMathOperator{\Bun}{Bun}
\DeclareMathOperator{\Edges}{Edge}
\DeclareMathOperator{\Euler}{Euler}
\DeclareMathOperator{\Fun}{Fun}
\DeclareMathOperator{\Image}{Im}
\DeclareMathOperator{\OEdges}{OEdge}
\DeclareMathOperator{\Obj}{Obj}
\DeclareMathOperator{\Vertices}{Vert}
\DeclareMathOperator{\coev}{coev}
\DeclareMathOperator{\ev}{ev}
\DeclareMathOperator{\hol}{hol}

\newcommand{\Amod}{\sA\textnormal{-mod}}
\newcommand{\BA}{\rB\mstrut _{\sA}}
\newcommand{\BH}{\rB\mstrut _{\!H}}
\newcommand{\BYf}{B^Y\!/\varphi}
\newcommand{\Cat}{\mathscr{C}at}
\newcommand{\Cx}{\CC^{\times}}
\newcommand{\Dx}{\rD\mstrut _{\phantom{W}\!\!\!\!\!\!\!x}}
\newcommand{\EA}[1]{\End\mstrut _{\sA}(#1)}
\newcommand{\EC}[1]{\End\mstrut _C(#1)}
\newcommand{\FAA}{\rF\mstrut _{\!\sA}}
\newcommand{\FA}{\sF\mstrut _{\!\pA}}
\newcommand{\FB}{\sF\mstrut _{\!\pB}}
\newcommand{\FG}{\rG\mstrut _{\!G}}

\newcommand{\GA}{(\Lambda ;A)}
\newcommand{\HA}{\Hom\mstrut _{\sA}}

\newcommand{\OA}{\mathcal{O}(\sA)}
\newcommand{\OC}{\mathcal{O}(C)}
\newcommand{\PYv}{\Pi \mstrut _{Y,v}}
\newcommand{\PY}{\Pi \mstrut _Y}
\newcommand{\Rep}{\mathscr{R}ep}
\newcommand{\TA}{\rT\mstrut _{\sA}}
\newcommand{\TC}{\rT\mstrut _{\!C}}
\newcommand{\TensCat}{\mathscr{T}ens\mathscr{C}at}
\newcommand{\VAA}{V\mstrut _{\sA}}
\newcommand{\VGG}{\Vect_G(G)}
\newcommand{\Vect}{\mathscr{V}ect}
\newcommand{\bA}{\BA}
\newcommand{\bC}{\mathbb{C}}
\newcommand{\bR}{\mathbb{R}}
\newcommand{\bX}{\partial X}
\newcommand{\be}{\bar e}
\newcommand{\bmut}{\bmu 2}
\newcommand{\bmu}[1]{\bmuu _{#1}}
\newcommand{\bone}{\boldsymbol{1}}
\newcommand{\bord}[1]{\Bord_{#1}}
\newcommand{\bung}[1]{\Bun_G(#1)}
\newcommand{\bunh}[1]{\Bun_{H}(#1)}

\newcommand{\cD}{\mathrsfs{D}}
\newcommand{\cE}{\mathrsfs{E}}
\newcommand{\cF}{\mathrsfs{F}}
\newcommand{\cG}{\mathrsfs{G}} 
\newcommand{\cH}{\mathcal{H}}
\newcommand{\cQ}{\mathrsfs{Q}}
\newcommand{\cRep}{\Rep}
\newcommand{\cR}{\mathrsfs{R}}

\newcommand{\cVect}{\Vect}
\newcommand{\conf}[2]{#1^{\Vertices(#2)}}

\newcommand{\dual}{^\vee}

\newcommand{\gpd}{/\!/} 
\newcommand{\hFG}{\rR\mstrut _{\!G}}
\newcommand{\hW}{\widehat{W}}
\newcommand{\pA}{M}
\newcommand{\pB}{N}
\newcommand{\pL}{L}

\newcommand{\pmo}{\{\pm1\}}
\newcommand{\ptt}{\textnormal{point}}
\newcommand{\rB}{\mathrsfs{B}}
\newcommand{\rD}{\mathrsfs{D}}

\newcommand{\rF}{\mathrsfs{F}}
\newcommand{\rG}{\mathrsfs{G}}
\newcommand{\rI}{\mathrsfs{I}}
\newcommand{\rL}{\mathrsfs{L}}
\newcommand{\rR}{\mathrsfs{R}}
\newcommand{\rT}{\mathrsfs{T}}
\newcommand{\sA}{\mathscr{T}}

\newcommand{\sC}{\mathscr{C}}
\newcommand{\sE}{\mathscr{E}}
\newcommand{\sF}{\mathfrak{F}}
\newcommand{\sG}{\mathscr{G}}
\newcommand{\sL}{\mathscr{L}}
\newcommand{\sM}{\mathscr{M}}
\newcommand{\sTd}{\sT\dual}
\newcommand{\sT}{\boldsymbol{T}}

\newcommand{\sZ}{\mathscr{Z}}

\newcommand{\tG}{\widetilde{\rG}}
\newcommand{\tH}{\widetilde{H}}
\newcommand{\tK}{\widetilde{K}}
\newcommand{\tbord}{\Bord_{\langle n-1,n  \rangle}}
\newcommand{\tcat}{\TensCat}
\newcommand{\tg}{\tilde{g}}

\newcommand{\tj}{\tilde\jmath }
\newcommand{\tpL}{\widetilde{\pL}}

\newcommand{\tpp}{\tilde p }
\newcommand{\triv}[1]{\underline{#1}}
\newcommand{\tx}{\tilde{x}}
\newcommand{\vy}{\vec{y}}
\newcommand{\zo}{[0,1]}

\begin{document}

\abovedisplayskip18pt plus4.5pt minus9pt
\belowdisplayskip \abovedisplayskip
\abovedisplayshortskip0pt plus4.5pt
\belowdisplayshortskip10.5pt plus4.5pt minus6pt
\baselineskip=15 truept
\marginparwidth=55pt

\makeatletter
\renewcommand{\tocsection}[3]{%
  \indentlabel{\@ifempty{#2}{\hskip1.5em}{\ignorespaces#1 #2.\;\;}}#3}
\renewcommand{\tocsubsection}[3]{%
  \indentlabel{\@ifempty{#2}{\hskip 2.5em}{\hskip 2.5em\ignorespaces#1%
    #2.\;\;}}#3} 
%  \indentlabel{\hskip 4em#3}}
\def\csubsection{\@ifstar{\@startsection{subsection}{2}%
  \z@{1.5\linespacing\@plus\linespacing}{+\fontdimen2\font\;}%
  {\normalfont\bfseries}*}% 
  {\setcounter{subsection}{\value{equation}}
   \stepcounter{equation}\@startsection{subsection}{3}%
  \z@{0\linespacing\@plus.3\linespacing}{-\fontdimen2\font\;}%
  {\normalfont\itshape}}}
\def\@sect#1#2#3#4#5#6[#7]#8{%
  \edef\@toclevel{\ifnum#2=\@m 0\else\number#2\fi}%
  \ifnum #2>\c@secnumdepth \let\@secnumber\@empty
  \else \@xp\let\@xp\@secnumber\csname the#1\endcsname\fi
  \@tempskipa #5\relax
  \ifnum #2>\c@secnumdepth
    \let\@svsec\@empty
  \else
    \refstepcounter{#1}%
    \edef\@secnumpunct{%
      \ifdim\@tempskipa>\z@ % not a run-in section heading
        \@ifnotempty{#8}{.\@nx\enspace}%
      \else
        \@ifempty{#8}{}{\@nx\enspace}%  CHANGED BY DSF
      \fi
    }%
    \@ifempty{#8}{%
      \ifnum #2=\tw@ \def\@secnumfont{\bfseries}\fi}{}%
    \protected@edef\@svsec{%
      \ifnum#2<\@m
        \@ifundefined{#1name}{}{%
          \ignorespaces\csname #1name\endcsname\space
        }%
      \fi
      \@seccntformat{#1}%
    }%
  \fi
  \ifdim \@tempskipa>\z@ % then this is not a run-in section heading
     \begingroup #6\relax
     \@hangfrom{\hskip #3\relax\@svsec}{\interlinepenalty\@M #8\par}
     \endgroup
    \ifnum#2>\@m \else \@tocwrite{#1}{#8}\fi
  \else
  \def\@svsechd{#6\hskip #3\@svsec
    \@ifnotempty{#8}{\ignorespaces#8\unskip\@addpunct.}
%   \@addpunct.}%
%    \ifnum#2=\@m \else \@tocwrite{#1}{#8}\fi
  }%
  \fi
  \global\@nobreaktrue
  \@xsect{#5}}

\makeatother

\newcommand{\bmuu}{\mbox{$\raisebox{-.07em}{\rotatebox{9}
  {\tiny
  /}}\hspace{-0.26em}\mu\hspace{-0.88em}\raisebox{-0.98ex}{\scalebox{2} 
  {$\color{white}.$}}\hspace{-0.416em}\raisebox{+0.88ex}
  {$\color{white}.$}\hspace{0.46em}$}} 
 
\setcounter{tocdepth}{2}

%**end of header

 \title[Topological dualities in the Ising model]{Topological Dualities in the Ising Model} %% replace
 \author[D. S. Freed]{Daniel S.~Freed}
 \thanks{This material is based upon work supported by the National Science
Foundation under Grant Numbers DMS-1406056 and DMS-1611957.  Any opinions,
findings, and conclusions or recommendations expressed in this material are
those of the authors and do not necessarily reflect the views of the National
Science Foundation.  This work was initiated at the Aspen Center for Physics,
which is supported by National Science Foundation grant PHY-1066293.}
 \address{Department of Mathematics \\ University of Texas \\ Austin, TX
78712} 
 \email{dafr@math.utexas.edu}

 \author[C. Teleman]{Constantin Teleman} 
 \address{Department of Mathematics \\ University of California \\ 970 Evans
Hall \#3840 \\ Berkeley, CA 94720-3840}  
 \email{teleman@math.berkeley.edu}

 \date{February 14, 2021}
 \begin{abstract} 
 We relate two classical dualities in low-dimensional quantum field theory:
\emph {Kramers-Wannier} duality of the Ising and related lattice models in
$2$ dimensions, with \emph{electromagnetic} duality for finite gauge theories
in $3$~dimensions. The relation is mediated by the notion of \emph {boundary
field theory}: Ising models are boundary theories for pure gauge theory in
one dimension higher. Thus the Ising \emph{order/disorder operators} are
endpoints of \emph{Wilson/'t~Hooft defects} of gauge theory.  Symmetry
breaking on low-energy states reflects the multiplicity of topological
boundary states. In the process we describe lattice theories as (extended)
topological field theories with boundaries and domain walls. This allows us
to generalize the duality to non-abelian groups; finite, semi-simple Hopf
algebras; and, in a different direction, to finite homotopy theories in
arbitrary dimension.
 \end{abstract}
\maketitle

In quantum field theory and statistical mechanics, the $2$-dimensional
\emph{Ising model} has earned the double distinction of being the first
discrete model to exhibit, against expectations, phase transitions in the
large volume limit \cite{P}, and the first non-trivial one to be solved
explicitly \cite{O}.  It is the simplest lattice sigma-model (in
apocryphal terminology) with only nearest-neighbor interactions, and depends
on a single parameter, physically interpreted as the \emph{temperature} $T$,
and encoded as the reciprocal $\beta= 1/kT$ (with Boltzmann's $k$). Nowadays,
detailed treatments can be found in graduate textbooks \cite{ID,
C}. The present paper is our mathematical attempt to understand some
features of the story and locate them within the algebraic structures of
topological quantum field theory (TQFT).

The Ising model assigns two possible states (\emph{spins} valued in $\pm1$)
to each node of a $2$-dimensional lattice. Equal spins for nearby nodes are
probabilistically favored, strongly or weakly, according to $\beta$. For a
very large lattice, the system exhibits two \emph{phases}: a
\emph{ferromagnetic} phase at low temperature, where the spins are mostly
aligned, with one sign dominating; and a \emph{paramagnetic} phase at high
temperature, where regions of spins of both signs co-exist. The model also
has a $2$D Euclidean (lattice) quantum field theory interpretation, with a
space of states $\cH$ assigned to any ``latticed'' circle (subdivided into
edges and nodes). Specifically, $\cH$ is the space of functions on the set of
possible independent spin assignments to the nodes, and is acted upon by the
\emph{transfer matrix}, the analogue of the exponentiated negative-signed
Hamiltonian on a cylindrical space-time. At low temperature, its top
eigenvalue is achieved on the two aligned spin states, the two
$\delta$-functions on the constant maps to $\{\pm1\}$. At high temperature,
the single top eigenvector is the constant function on the set of all spin
configurations, matching the statistical fact that no particular spin
configuration is favored. The model has a global $\bmut=\{\pm1\}$ symmetry,
acting simultaneously on all spins, and the cold phase exhibits
\emph{symmetry breaking}: choosing to live near one or the other of the
distinguished eigenvectors leads to inequivalent spectra of the Hamiltonian on
the Hilbert spaces of states.

\emph{Kramers-Wannier} (KW) duality relates computed quantities in Ising
models at temperatures $T, T^\vee$ related by the formula $\sinh (2\beta)
\sinh(2{\beta}^\vee)=1$.  On a general surface, we must dualize the lattice
as well, so this is only a self-duality of the model on a square planar
lattice. The value $\beta = \beta^\vee = \frac{1}{2}\,\mathrm{arcsinh}(1)$,
fixed under duality, is a candidate for a \emph{critical value}, the
\emph{phase transition} between the high and low temperature phases on the
square lattice. It \emph{must} be the critical value, should there be a
unique such, as was later confirmed by the explicit solution of the
Ising model \cite{O}. A similar line of reasoning applies to the $n$-state
Potts model \cite[\S4.1]{ID}.  At the risk of irritating the expert reader,
we first describe the duality in the traditional way, as a Fourier
transform~(\S{\ref{fourier}}).

There are problems with this naive formulation~(\S\ref{failure}).  Our
resolution of these problems proceeds by coupling the Ising model to purely
topological $3$-dimensional gauge theory for the finite group~$\bmut$.
Previously, gauge theory appeared in the dual of the ungauged abelian
2d~Ising model~\cite[\S10]{KS}, and this can be understood within our
framework.  This approach---the Ising model as a boundary theory for
3-dimensional pure gauge theory---is a strong manifestation of the
$\bmut$-symmetry of the model, and it is the springboard for all that
follows.  The KW duality of the Ising model is now the mapping of boundary
theories under \emph{electromagnetic duality} of finite 3D~gauge theory.
This entire story generalizes to any finite abelian group in place
of~$\bmut$.

In addition to these new insights, our point of view leads to several new
results:

 \begin{enumerate}%[label=\textnormal{(\roman*)}]

 \item We construct a dual to the non-abelian Ising model~(\S\ref{sec:6}).
Here, the Ising side is written in the usual way, albeit with a non-abelian
finite group $G$; whereas the dual side is a state-sum construction of the
partition function from the category of representations of $G$, based on
Turaev-Viro theory.  This duality appears to be new; its most general version
features finite-dimensional semi-simple Hopf algebras.

 \item We give an abstract reformulation of the Ising model in terms of fully
extended topological field theories with a \emph{polarization}, a
complementary pair of boundary theories~(\S\ref{sec:8}).  This places lattice
theories in the context of \emph{fully extended topological field theories},
for which there is a well-developed mathematical theory.  We use it to prove
the Duality Theorem~\ref{thm:8.13}.

 \item We predict the classification of gapped phases of Ising-like models,
which ends up conforming to the \emph{Landau symmetry breaking} paradigm
~(\S\ref{sec:5}). Since Ising theory is defined relative to $3$D topological
gauge theory, so will be any of its gapped topological sectors.  Using the
cobordism hypothesis and a theorem about tensor categories, we
prove that simple, fully extended $2$D topological theories relative to gauge
theory are classified by subgroups of $G$ equipped with a central extension.
As low energy approximations to the Ising model, central extensions can be
ruled out by a positivity assumption on the (exponentiated) action, and this
strongly supports a conjectural classification of the gapped phases of the
theory in terms of subgroups of $G$, to wit, the unbroken symmetry subgroups
of Landau.

 \item We construct generalizations of the Ising model to higher dimensions
and to exotic homology theories~(\S\ref{sec:7}).  In contrast to higher
dimensional Ising models based on ordinary homology theory~\cite[\S6.1]{ID},
the models based on exotic homology theories do not have a natural lattice
formulation.

 \end{enumerate}

Here is a road map to the paper.  We offer the reader an extended executive
summary of our results in~\S\ref{sec:1}.  In~\S\ref{sec:2} we review basic
notions of extended topological field theory, including boundary theories and
domain walls.  Three-dimensional finite gauge theory is the subject
of~\S\ref{sec:3}, with an emphasis on electromagnetic duality and its
nonabelian generalization.  The Ising model as a boundary theory is developed
in~\S\ref{sec:4}, where we derive Kramers-Wannier duality from
electromagnetic duality.  Constraints on low energy effective topological
theories are described in the heuristic~\S\ref{sec:5}.  The remaining parts
of the paper lie squarely in extended topological field theory.
In~\S\ref{sec:9} we illustrate computations in three dimensions, emphasizing
the utility of the {\it regular boundary theory\/} attached to a tensor
category.  The dual to nonabelian Ising is presented in~\S\ref{sec:6}.
Section~\ref{sec:8} provides a more general setting and generalized lattice
models based on Hopf algebras.  We conclude in~\S\ref{sec:7} with a
discussion of higher dimensional theories and electromagnetic duality, with
the main tool another construction in extended field theory: the finite path
integral.

After posting our paper, we learned from P. Severa that our description
in~\S\ref{sec:8} of Kramers-Wannier duality as a bicolored TQFT reproduces
many of his ideas in~\cite{S}.  We have kept the exposition unchanged---even
though there is some repetition of~\cite{S}---both for the reader's
convenience and because the setting of \emph{fully extended} TQFTs, on which
our results rely, requires a different setup.

We thank David Ben-Zvi, John Cardy, Paul Fendley, Anton Kapustin, Subir
Sachdev, Nathan Seiberg, and Senthil Todadri for informative discussions.

{\small\tableofcontents}

   \section{Summary of the paper}\label{sec:1}
% lastsubsec@000

This section offers an executive summary of the paper, with the main
definitions and results, in the hope that it will assist the reader in
locating the material of greatest interest.

\csubsection{The Ising model on a latticed surface.}\par\label{isingref}
Choose a finite abelian group $A$ and an \emph{even} function $\theta:A\to
\bR$. The standard case has $A=\bmut=\{\pm1\}$ and $\theta(\pm1) =
e^{\pm\beta}$. Evenness makes the Fourier transform $\theta^\vee$ real-valued
on $A^\vee$. A statistical interpretation is only sensible for positive
$\theta$ (and, dually, $\theta^\vee$), but our discussion does not rely on
this.

We shall view our theories as topological, albeit in a uncommonly broad
sense: the Ising ones involve \emph{latticed} surfaces --- subdivided by a
lattice (embedded graph) $\Lambda$ into faces, each required to be
diffeomorphic to a convex closed planar polygon with at least two
edges.\footnote{From a different vantage point, a lattice is a discrete
analogue of a Riemannian metric.  However, one of our contributions is to
translate the lattice into purely topological data, allowing us to use the
rich structure of TQFT.}  The dual lattice $\Lambda^\vee$ is then defined up
to a contractible space of isotopies, and has the same properties.

Given an oriented and latticed surface $(Y,\Lambda)$ with vertices, edges and
faces indexed by sets $V,E,F$ respectively, we define a measure on the space
of \emph{classical fields} (generalized spins), the maps $s: V\to A$.  The
weight of a field $s$ is the product over all edges $e\in E$ of the
$\theta$-value of the ratio of adjacent spins:

 \begin{equation}\label{measure} 
  Z(s) := \prod\nolimits_{e\in E}
  \theta\bigl(s(\partial_+e)s(\partial_-e)^{-1}\bigr) =
	\exp\left(\beta\sum\nolimits_{e\in
	E}s(\partial_+e)s(\partial_-e)^{-1}\right).
\end{equation} 
 The orientation of the edge $e$, implicit in labeling the $+$ and $-$
endpoints, is irrelevant when $\theta$ is even. The ``partition function"
$Z(A,\theta,\Lambda):=\sum_s Z(s)$ is the sum over all fields. We can insert
functions $f(s)$ within the sum; the resulting numbers $\langle f\rangle
:=\sum_s Z(s)f(s)$ can be interpreted as (un-normalized) correlations in a
statistical mechanical system or in a lattice QFT. A function $f(s)$ is a sum
of monomials $\prod_v f_v(s(v))$ in functions $f_v$ of the values $s(v)\in A$
at specified vertices $v$, and we may restrict the $f_v$ to range over the
non-trivial characters of $A$. The latter are the \emph{order operators},
each labeled by a vertex and a non-trivial character.  There is a unique
order operator at each $v$ when $A=\bmut$.

\csubsection{Kramers-Wannier duality as a Fourier transform.}\par\label{fourier}  To express
$\langle f \rangle$, consider the Pontrjagin dual complexes of $A$-valued
co-chains and $A^\vee$-valued chains for the latticed surface $Y$,
respectively 
 \begin{equation}\label{chains}
	A^V \xrightarrow{\delta^0} A^E \xrightarrow{\delta^1} A^F, \qquad
					(A^\vee)^V \xleftarrow{\partial_1}
					(A^\vee)^E \xleftarrow{\partial_2}
					(A^\vee)^F,
 \end{equation}
placed in cohomological, respectively homological degrees $0,1,2$. Note, on
$A^E$, the two functions $\Theta := \bigotimes_E \theta$, and the
delta-function $\Delta_B$ on the subgroup $B^1(A)
=\mathrm{Im}(\delta^0)$. Then,
 \[ \sum\nolimits_s Z(s) = \#H^0(Y;A)\cdot
\langle \Theta |\Delta_B \rangle, \quad\text{ and }\quad
	\langle f \rangle = \sum\nolimits_s f(s) Z(s)
		= \langle \Theta |\delta^0_*f\rangle, \] 
 where $\delta^0_*f$ is the fiber-wise sum of $f$ along $\delta^0$. Observe
that summation over fields $s\in A^V$ has been replaced by a summation over
edges, implicit in the inner product on functions on $A^E$.

Dually, we have the Fourier transforms $\Theta^\vee = \bigotimes_{e\in E}
{\theta}^\vee$ and the delta-function $\Delta_Z = {\Delta}_B^\vee$ on the
$1$-cycles $Z_1(A^\vee) = \ker\partial_1$.  When $A=\{\pm1\}$,
${\theta}^\vee$ corresponds to the dual value $\beta^\vee$, except for an
overall scaling which rescales the expectation values $\langle
f\rangle$. Parseval tells us that $\langle \Theta |\Delta_B \rangle = \langle
\Theta^\vee | \Delta_Z \rangle$. Interpreting now the second complex in
\eqref{chains} as the cochain complex for the dual lattice,\footnote {An
orientation of $Y$ is needed, if $A\neq\bmut$.} we are close to
equating $Z(A,\theta, \Lambda)$ with $Z(A^\vee,\theta^\vee, \Lambda^\vee)$,
except for the vexing difference between $Z_1(A^\vee) = \ker\partial_1$ and
its subgroup $B_1(A^\vee) = \mathrm{Im}\,\partial_2$, over which the dual
Ising partition function would have liked to sum instead
\cite[\S6.1]{ID}.  We revisit Kramers-Wannier duality from Fourier transforms
in~\S\ref{subsec:4.4}.

\csubsection{Failure of duality.}\par\label{failure} The dual partition functions
fail to agree, with the original side missing a summation over
$H^1(Y;A)$. Duality for $\langle f\rangle $ is worse, as the Fourier
transform of $\delta^0_*f$ is $\partial_1^*({f}^\vee)$, with ${f}^\vee$
computed on $A^V$: we now sum over chains with boundary in the support of
$f^\vee$, not cycles. For example, when $f$ is a ratio of order operators
defined by the character $\chi\in A^\vee$ at the endpoints $\partial_\pm e$
of an edge $e$, $f(s) = \chi(s(\partial_+e)) \cdot
\chi(s(\partial_-e))^{-1}$, we sum $\Theta$ over the translate of $Z_1$ by
$\chi\otimes e$.

The usual escape from this second difficulty uses the language of
\emph{disorder operators}. In the $(A^\vee,\theta^\vee, \Lambda^\vee)$-Ising
model, one interprets summation over $(Z_1+\chi\otimes e)$ as a
\emph{frustrated partition function} with line of frustration $e$. In the
dual lattice, this line joins the centers of the faces $\partial_\pm e$. The
frustration line modifies the weight in \eqref{measure} by a judicious
$\chi$-insertion, as follows: the factor corresponding to the edge $e'$ of
$\Lambda^\vee$ crossing $e$ becomes $\theta^\vee\left(s^\vee
(\partial_+e')\cdot s^\vee(\partial_-e')^{-1} \cdot \chi\right)$. More
generally, for a ratio of order operators at the endpoints of a longer path
$\pi$ in the lattice $\Lambda$, this modification applies to all the edges
$e'$ of $\Lambda^\vee$ which cross $\pi$. A standard calculation
\cite[\S2.2.7]{ID} shows that only the homotopy\footnote{Only homology
matters here, but the homotopy class will be needed in the non-abelian
generalization.} class of $\pi$, relative to $\partial\pi$, affects the dual
computation of $\langle f \rangle$. The entire story is crying out for help
from elementary topology.

\csubsection{Abelian gauge theory in $3$D}\par Our way to clear these faults
consists in viewing the Ising model not as a standalone lattice theory, but
as a boundary theory for $3$-dimensional \emph{pure topological gauge
theory}: the theory which counts principal bundles.  We are coupling the
model to a background gauge field, in physics language.  The relevant
structure groups are $A$ and $A^\vee$, and we call those theories $\cG_A$ and
$\cG_{A^\vee}$.  As we will recall in~\S\ref{sec:3}, \emph{electromagnetic
duality} identifies $\cG_A$ and $\cG_{A^\vee}$ as \emph{fully extended}
TQFTs. Here, let us note that the vector space $\cG_A(Y)$ which $A$-gauge
theory assigns to a closed surface $Y$ comprises the complex functions on the
set $H^1(Y;A)$, the moduli space of (necessarily flat) $A$-bundles on
$Y$. When $Y$ is oriented, Poincar\'e duality places the groups $H^1(Y;A)$
and $H^1(Y;A^\vee)$ in Pontrjagin duality, and the Fourier transform
identifies $\cG_A(Y)$ with $\cG_{A^\vee}(Y)$. The full equivalence of gauge
theories is in fact a ``higher categorical'' version of the Fourier
transform; see~\S\ref{subsec:3.5} and \S\ref{subsec:3.2}.

On a closed latticed surface $Y$, the Ising partition function $Z(A,\theta,
\Lambda)$ can be promoted to a genuine function on the set $H^1(Y;A)$, giving
a vector $\mathbf{Z}(A,\theta, \Lambda)\in \cG_A(Y)$.  Indeed, given a
principal $A$-bundle $P\to Y$, we re-define spins to be sections of $P$ over
the vertex set $V$, rather than maps to $A$; the factors in the measure
\eqref{measure} are still meaningful, thanks to the flat structure of
$P$. (In fact, we only need $P$ to live on the $1$-skeleton of $Y$, an
observation which will come in shortly.)  A typical picture illustrating the
boundary theory has a compact three-manifold $X$ with latticed boundary
$(Y,\Lambda)$. This determines a number, namely the sum of $Z(A,\theta,
\Lambda)$-values on all principal bundles over $Y$ equipped with an extension
to $X$, weighted down by the order of their automorphism groups;
see~\eqref{eq:48}. In the formalism of relative field theory, this is the
pairing of $\mathbf{Z}(A,\theta, \Lambda)$ with the vector $\cG_A(M)\in
\cG_A(Y)$ defined by the TQFT $\cG_A$.

\begin{remark}\label{orderpair} Gauging the Ising model destroys our original
order operators: the spins at a vertex $v$ take values in an $A$-torsor, and
we cannot evaluate characters thereon.  This is part of the medicine, though:
the dual disorder operators come in opposite pairs, joined by a path (up to
homotopy). This setup also works for the order operators: parallel transport
along a lattice path identifies the torsors at different vertices $v,v'$, and
the ratio $s(v)s(v')^{-1}$ of two spins becomes a well-defined element
of $A$, on which we \emph{can} evaluate characters. This phenomena are
illuminated and resolved by the notion of \emph{defects}, below.
\end{remark} 

\csubsection{Defects.}\par There are two distinguished types of defects in finite
gauge theory \cite{Wi, tH}; in dimension~$3$, they are both
$1$-dimensional. \emph{Wilson loops} are labeled by characters $a^\vee$ of
the gauge group $A$, and change the count of principal bundles on a closed
manifold, re-scaling each by the value of~ $a^\vee$ on the holonomy along the
loop. The other distinguished defect, a \emph{'t Hooft loop},\footnote{Often
the term `vortex loop' is used instead.  't Hooft lines may be more familiar
to field theorists in $4$-dimensional gauge theory with connected gauge
group~$G$, where they are labeled by elements of $\pi_1G$. The codimension~
$2$ loops here are labeled by $\pi_0G$; in either case, they form a
Poincar\'e dual representative of a gerbe, obstructing the extension of a
principal $G$-bundle defined on the complement of the loop.}  is labeled by
an element $a\in A$. Instead of changing the measure, it modifies the space
of classical fields, from principal $A$-bundles to principal $A$-coverings
ramified along the loop, with normal monodromy $a$.  We describe these
defects in~\S\ref{subsec:3.3}.  We will see later (\S\ref{sec:8}) that these
two types of defects are naturally associated to the \emph{Dirichlet}
(gauge-fixing), respectively \emph{Neumann} (free) boundary conditions of
gauge theory. A beautiful feature of topological electromagnetic duality in
$3D$ is the interchange of these boundary conditions, hence of the Wilson and
't~Hooft loops.

On a manifold with boundary, Wilson and 't Hooft lines need not close up:
they can end in Wilson and 't~Hooft \emph{point defects} on the boundary. (If
the manifold has corners, defects must be interior points of the boundary
surface.) A surface $Y$ with (colored) Wilson and 't~Hooft defect divisors
$W,t$ leads to a modified space of states $\cG_A(Y;W,t)$. We refer to
\S\ref{subsec:3.3} for details, but note here that $\cG_A(Y;W)$ comprises
gauge invariant functions on $A$-bundles valued in the tensor product
$W_\otimes$ of the representations coloring the Wilson divisor $W$; whereas
in building $\cG_A(Y;t)$, we replace principal $A$-bundles on $Y$ with
principal covers ramified at $t$, as specified by the 't~Hooft labels.

\csubsection{Defect cancellation.}\par\label{defectcancel} If $W$ consists of
\emph{pairs} of dually colored points joined by paths, we can trivialize
$W_\otimes$ (using parallel transport) and identify $\cG_A(Y)$ with
$\cG_A(Y;W)$. Dually, ramified $A$-covers are classified by a torsor over
$H^1(Y;A)$, built from $A$-co-chains with co-boundary Poincar\'e dual to
$t$. Writing $t$ as a boundary on $Y$ supplies a base-point in the torsor and
identifies $\cG_A(Y;t)$ with $\cG_A(Y)$. The two maps relating the original
and defective spaces are induced by a uniform picture: a cylinder bordism
$Y\times[0,1]$, with defect on the top face only, closed up by buried defect
cables under the connecting paths on $Y$. We see from here that the requisite
structure is a null-bordism of the defect, and the system of paths must be
provided with over/under-crossing data.  \\ \emph{Caution.} In the
non-abelian generalization, a null-bordism gives a pair of adjoint maps
between the defective and neat spaces, but they are not isomorphisms.

\csubsection{Order and disorder defects.}\par\label{orderdef} Place, on latticed
surface, Wilson defects at certain vertices and 't~Hooft ones inside certain
faces. Enhance the Wilson defects to order operators by supplying them with a
vector in the representation at each point.\footnote{This step can be
concealed on a closed surface, because $\cG_A(Y;W)=0$ unless $W_\otimes\cong
\bC$, which carries a preferred vector; it is also invisible for 't~Hooft
defects, where we can use a canonical vector $1\in\bC$.} We can now build a
distinguished Ising partition function $\mathbf{Z}(A,\theta,\Lambda,W,t)$ in
the defective space $\cG_A(Y;W,t)$ by a state-sum recipe adapted from
\S\ref{isingref}: the product $f_W$ of Wilson order operators is valued in
$W_\otimes$ and the measure \eqref{measure} is sensibly defined, as the
ramification avoids nodes and edges; see~\S\ref{subsec:4.3}. A complete $3$D
picture with boundary will have a bulk $3$-manifold $M$, with Wilson lines
ending in $W$ and 't~Hooft ones in $t$; we obtain a number by pairing
$\cG_A(M)$ with $\mathbf{Z}(A,\theta,\Lambda,W,t)$, or by counting ramified
covers and sections over the nodes, with Ising and Wilson weights.

\csubsection{General Ising correlators.}\par  A null-bordism of the defect
(\S\ref{defectcancel}) identifies $\cG_A(Y;W,t)$ with $\cG_A(Y)$, but the two
Ising partition functions do not match. Instead,
$\mathbf{Z}(A,\theta,\Lambda,W,t)$ becomes, in the original space $\cG_A(Y)$,
the frustrated expectation value of $f_W$, converted to a function
(Remark~\ref{orderpair}) and with lines of frustration the connecting paths
for $t$.  This is our relative field theory reading of Ising correlators, in
purely topological setting.

\csubsection{Duality restored.}\par  The first failure of Kramers-Wannier duality
in \S\ref{failure} is repaired by saying that the gauged Ising partition
functions $\mathbf{Z}(\Lambda,A,\theta)$ and
$\mathbf{Z}(\Lambda^\vee,A^\vee,\theta^\vee)$, now functions on $H^1(Y;A)$
and $H^1(Y;A^\vee)$, are related by the Fourier transform.  The most general
duality involves the 't~Hooft and Wilson defects: the Fourier transform
identifies the spaces $\cG_A(Y;W,t)$ and $\cG_{A^\vee}(Y; t^\vee, W^\vee)$
and the Ising partition functions within. We summarize it in our first
theorem.

\begin{theorem}\label{maintheorem} Electromagnetic duality for $3$D finite
abelian gauge theory extends to the Ising boundary theories with Fourier dual
actions $\theta, \theta^\vee$, where it becomes Kramers-Wannier duality.
Order operators of the Ising model are based at Wilson defects, and disorder
operators at 't~Hooft defects. They get interchanged under duality.
\end{theorem}
 
\noindent
 We give the proof in~\S\ref{subsec:4.4} and~\S\ref{subsec:8.1}.

\csubsection{Symmetry breaking on low-energy states.}\par We can now resolve
another puzzle of the KW duality, the mismatched symmetry breaking. The low
temperature regime has \emph{two} vacua, or lowest-energy states,
interchanged by the global spin-change symmetry.  In the large lattice limit,
there will be two distinct Hilbert spaces of states near the two vacua: the
$\bmut$-symmetry is thus broken, no longer acting on the
separate spaces of states. The high-temperature phase has a unique vacuum,
the constant function on all spin states, invariant under the
$\bmut$ global symmetry, which continues to act on the Hilbert
space.

This mismatch appears to contradict KW duality, but again the problem is
cleared by coupling to a background principal bundle. There are two principal
$\bmut$-bundles over the circle, up to isomorphism, and we have
just examined the trivial one. The twisted sector shows a different mismatch
between the phases: at low temperature, there is no contribution to the
topological sector --- the lowest energy is greater than the (untwisted)
vacuum energy, as there is no constant-spin state.  At high temperature
however, we still get the constant function on all spin configurations as the
unique vacuum. The two vacua at high temperature split over the two sectors
$\{\pm1\}$, instead of the two representations of $\bmut$.  See the
discussion at the end of~\S\ref{sec:5}.

This is consistent with $3$D electro-magnetic duality. Specifically,
$\cG_{\bmut}$ assigns to the circle the category of
$\bmut$-equivariant vector bundles over $\bmut$,
for the trivial action.\footnote{This becomes the conjugation action, when
$\bmut$ is replaced by a non-abelian structure group.} An object
breaks up into four components, each labeled by one of the two
representations of $\bmut$ and one of the two twisted
sectors. Duality swaps the nature of the labels. The topological sector of
Ising theory is an object in this category: at low temperature, it is the
regular representation of $\bmut$, spanned by the two vacua
living over the trivial sector, while at high temperature it is the dual
object, with a copy of the trivial representation in each sector.

\csubsection{Non-abelian Ising model.}\par It is straightforward to generalize
half of this story to a non-abelian finite group $G$, equipped with an even
function $\theta:G\to \bR$, and indeed we treat arbitrary finite groups from
the beginning in~\S\ref{sec:4}. The measure is defined by the same formula
\eqref{measure}, summed over all fields; moreover, we can couple this to the
pure $G$-gauge theory $\cG_G$. This theory carries Wilson and 't~Hooft loops,
as in the Abelian case. The former are labeled by representations, and they
weight the measure on the space of fields (flat $G$-bundles) by the trace of
the holonomy around the loop. The latter modify the space of principal
$G$-bundles into a space of ramified principal covers. Defect lines may end
in defect points on a boundary surface. If we place a lattice theory on the
boundary, Wilson defects may be sited at nodes and 't~Hooft ones inside faces
of the lattice, where they appear as order/disorder
operators.\footnote{Recall from \S\ref{orderdef} that an order operator is
labeled by the representation \emph{and} a choice of vector therein.}
However, formulating the dual side without a Pontrjagin dual group requires a
step into abstraction.

\csubsection{Nonabelian Electromagnetic duality.}\par  The abelian gauge theory
$\cG_A$ is generated by the tensor category $\mathscr{V}ect[A]$ of $A$-graded
vector bundles \cite{FHLT}, and $\cG_{A^\vee}$ by $\mathscr{V}ect[ A^\vee]$.
Now, $A^\vee$-graded vector bundles are precisely the (spectrally decomposed)
representations of $A$, with their tensor product structure. For any finite
$G$, the tensor category $\mathscr{R}ep(G)$ defines a fully extended TQFT
$\cR_G$: it can be constructed by a state-sum recipe due to Turaev and Viro
\cite{TV}, applicable to any \emph{fusion category}\footnote {See \cite{EGNO}
for an account of fusion categories.} $\mathscr{T}$;
see~\S\ref{subsec:6.1}. For $\mathscr{T}=\mathscr {V}ect[G]$, the recipe
yields $\cG_G$ in familiar bundle-counting form. Applied to $\mathscr{T}=
\mathscr{R}ep(G)$, the recipe looks different; yet it produces a canonically
equivalent theory on oriented manifolds. This is \emph{non-abelian
electromagnetic duality}; abstractly, it expresses the Morita equivalence of
tensor categories $\mathscr{V}ect[G]\simeq\mathscr{R}ep(G)$ \cite{EGNO}, and
is a non-abelian version of the (doubly categorified) Fourier transform.  We
give a proof in~\S\ref{subsec:3.5} based on the cobordism hypothesis.

\csubsection{Duality for defect lines.}\par  Wilson defects can be defined in any
Turaev-Viro theory. Let us spell this out algebraically, postponing a
conceptual account until \S\ref{abstractkoszul}. We will place a mild
constraint on $\mathscr{T}$ (a \emph{pivotal structure}) to secure the
orientability of our theory on circles and surfaces; without that, the need
for Spin structures adds a layer of complexity to the story.  For $3$D
theories of oriented manifolds, line defects are determined by objects in the
category associated to the circle, which is the \emph{Drinfeld center
$Z(\mathscr{T})$} of $\mathscr{T}$.  For both $\mathscr{T}=\mathscr{V}ect[G]$
and $\mathscr{T}=\mathscr{R}ep(G)$, $Z(\mathscr{T})$ comprises the
conjugation-equivariant vector bundles over $G$. There is a trace
functor\footnote {The natural target of the trace is the \emph{co-center} of
$\mathscr{T}$; the pivotal structure of $\mathscr{T}$ identifies that with
$Z(\mathscr{T})$. In TQFT language, the trace appears as an
\emph{open-closed} map \cite{MS}.}  $\mathscr{T}\to Z(\mathscr{T})$,
and Wilson defects are components of the trace of the tensor unit of
$\mathscr{T}$.  The unit in $\mathscr{V}ect[G]$ is the line supported at
$1\in G$, and its image in $Z$ is the regular representation supported at
$1$, the push-forward from a point to $BG$.  For
$\mathscr{T}=\mathscr{R}ep(G)$, the trace pulls back $G$-representations to
$G$-equivariant bundles over $G$, and takes the unit representation to the
trivial line bundle. Its decomposition into conjugacy classes gives the
former 't~Hooft defects of $\cG_G$ as Wilson defects of $\cR_G$.

Defining 't~Hooft defects in Turaev-Viro theory requires the additional
structure of a \emph{fiber functor $\phi$} on the tensor category
$\mathscr{T}$. This determines another object in $Z(\mathscr{T})$, the trace
of the identity in the endofunctor category of $\phi$. Its components are the
't~Hooft defects we seek. Each of $\mathscr{V}ect[G]$ and $\mathscr{R}ep(G)$
has a natural fiber functor, the global sections of a bundle and the
underlying space of a representation, respectively. This time, starting with
$\mathscr{T}=\mathscr{V}ect[G]$, the endofunctor category of $\phi$ is
$\mathscr{R}ep(G)$, and we discover in $Z(\mathscr{T})$ the earlier-described
't~Hooft defects of gauge theory; whereas $\mathscr{T}=\mathscr{R}ep(G)$ with
its obvious fiber functor leads instead to the original Wilson lines.  We
develop these ideas in~\S\ref{subsec:8.2a}.

\csubsection{Topological phases of Ising theories.}\par  If the action $\theta$
is such that the theory is \emph{gapped}, we expect Ising theory to converge
to a topological field theory in the thermodynamic (large lattice) limit, as
we explore in~\S\ref{sec:5}. The structure we uncovered forces the limiting
theory to be a boundary theory for pure $3$D gauge theory. Now, in the
setting of fully extended TQFTs, the boundary theories for gauge theory can
be classified as (sums of) simple ones, each defined by symmetry breaking
down to a subgroup $H$ of $G$ together with a central extension of $H$. The
central extension contributes a ``discrete torsion'' term; any such will
involve signs or complex numbers, which cannot appear for positive actions
$\theta$. This strongly suggests that the topological phases of Ising-style
theory with group $G$ are classified by their unbroken symmetry subgroups. We
get a topological theory on the nose if $\theta$ is the characteristic
function on $H$: the transfer matrix then becomes a (scaled) projector onto
the space of vacua, which gets identified with the functions on $G/H$.

\csubsection{Ising vector in the Turaev-Viro space of states.}\par  Let us
demystify the nonabelian Ising model by sketching here the construction of
the space of states $\cF_\mathscr{T}(Y)$ for a closed latticed surface
$(Y,\Lambda)$ in Turaev-Viro theory, to be enriched by the construction of a
distinguished vector, the Ising partition function, once we supply a fiber
functor and an Ising action $\theta$. The space $\cF_\mathscr{T}(Y)$ depends
on $Y$ alone, but its construction steps through a larger,
$\Lambda$-dependent space, $\VAA(Y;\Lambda)$ where the Ising
partition function resides.  Details are found in \S\ref{sec:6}.

Choose once and for all a basis $\{x_i\}$ of simple objects in the fusion
category $\mathscr{T}$.  Orient the edges of $\Lambda$ and label them with
simple objects. For each face $f$ with bounding edges labeled by $x_j$,
build the space $H_f:=\mathrm{Hom}_\mathscr{T}(\mathbf{1},\otimes x_j^
{(\vee)})$, with tensor factors cyclically ordered along the boundary
$\partial f$, and dualized whenever our edge orientation disagrees with the
$\partial f$ orientation. We sum the tensor products $\bigotimes_f H_f$ over
all labelings to produce $\VAA(Y;\Lambda)$. This is a version of
$\cF_\mathscr{T}(Y)$ with `gauge-fixing' at the $\Lambda$-vertices. To remove
the gauge-fixing, we will recall in \S\ref{subsec:6.1} the construction of a
commuting family of projectors --- one for each vertex of $\Lambda$ --- which
enforce gauge invariance, in the sense that their common image is
$\cF_\mathscr{T}(Y)$.

The dual space $H_f^*$ is mapped by the fiber functor $\phi$ to the dual
of $\bigotimes_f \left(\otimes \phi(x_j)^{(*)}\right)$. Each space
$\phi(x_j)$ appears here in a dual pair, for the two faces bounded by its
edge. A choice of Ising action $\theta\in \bigoplus_i
\phi(x_i)\otimes\phi(x_i)^*$ defines,\footnote{To avoid dependence on our
choice of edge orientations, the action $\theta$ must be symmetric under the
involution $x\leftrightarrow x^\vee$.} by contraction, a functional on the
last space and hence on $H_f^*$. Dualizing it gives a vector in
$\VAA(Y;\Lambda)$. Projected to $\cF_\mathscr{T}(Y)$, this is the
Ising partition function; see~\S\ref{subsec:6.2}.

\csubsection{Non-abelian Kramers-Wannier duality.}\par  In \S\ref{sec:8}, we
complete the above constructions to a lattice boundary theory for the $3$D
Turaev-Viro theory $\cR_G$, equipped with order and disorder operators at the
ends of the Wilson and 't~Hooft lines. The most general story pertains to
finite-dimensional semi-simple Hopf algebras (\S\ref{hopf} below), but let us
summarize it now for a general finite group $G$.  

   \begin{theorem} 
 Theorem \ref{maintheorem} generalizes to a duality between the gauge theory
of a finite group $G$ and the Turaev-Viro theory based on $\mathcal{R}ep(G)$,
and a Kramers-Wannier duality of their lattice boundary theories. There is an
interchange of Wilson and 't~Hooft defects in the bulk theories, and of order
and disorder operators for the boundary, and the Ising partition functions
agree.\footnote {Possibly after normalization by an overall constant.}
  \end{theorem}

\noindent 
 We provide a proof in~\S\ref{subsec:8.1}.

\csubsection{Lattice theory from a bicolored TQFT}\par\label{polarized}
 Underlying our non-abelian generalization is a reformulation of the gauged
Ising model in the language of extended field theory, boundaries and defects;
we describe it in~\S\ref{subsec:8.2a}.  Abstractly, we start with a bulk $3$D
topological field theory $\cF$ and two boundary theories $\rB, \rB'$; the
$G$-valued Ising story uses $\cF =\cG_G$ with its Dirichlet and Neumann
boundary conditions.  Algebraically, $\cG_G$ is generated by the tensor
category $\mathscr{V}ect[G]$, while the boundary theories are defined by the
left module category $\mathscr{V}ect[G]$ and the fiber functor of global
sections to $\mathscr{V}ect$. These theories satisfy the condition that
$\mathrm{Hom}_\cF(\rB,\rB')$ is the trivial $2$-dimensional theory
($G$-bundles on an interval which are based at one end and free at the other
are canonically trivialized), which determines a canonical defect $\rD$
between the boundary theories $\rB$ and $\rB'$.

\begin{remark}[Polarization of $\cF$]
 Our gauge theory quadruple satisfies the extra condition that each of
$\rB,\rB'$ is a \emph{generating} boundary theory for $\cF$ (see
Remark~\ref{gen} below for explanation.)  Generation and complementarity make
the pair $(\rB,\rB')$ akin to a \emph{polarization} of $\cF$, analogous with
the chiral and anti-chiral boundary conditions of Chern-Simons theory of the
WZW model, or the more general respective picture in rational conformal field
theory.  Our construction converts lattice models into a discrete,
topological analogue of that famous construction.
 \end{remark}  

\begin{remark}\label{gen}
 When $\cF=\cF_\mathscr{T}$ is generated by a \emph {multi-fusion
category}\footnote{A finite, rigid, semi-simple category; assuming simplicity
of the unit object makes it \emph{fusion}.}  $\mathscr{T}$, any boundary
theory comes from a finite module category $\mathscr{M}$, and the generation
condition is that $\mathscr{M}$ induces a Morita equivalence of $\mathscr{T}$
with its centralizer $\mathscr{E}=\mathrm{End}_\mathscr{T}(\mathscr{M})$.
(This is equivalent to the \emph{faithfulness} of \cite{EGNO}. For fusion
categories, all non-zero module categories are faithful.) The regular
$\mathscr{T}$-module $\mathscr{T}$ always generates; for $\cG_G$, the
generating property of the Neumann condition is the Morita equivalence
$\mathscr{V}ect[G]\simeq\mathscr{R}ep(G)$.
 \end{remark}

We now re-interpret our relative Ising model on a latticed surface
$(Y,\Lambda)$ as a multi-layered topological picture. From the lattice
$\Lambda$, build a self-indexing Morse function $f$ on $Y$ with minima at the
vertices, saddle points at the edge centers and maxima centered in the
faces. Color the set $0\le f <1$ by the Dirichlet boundary condition $\rB$
and the set $1< f \le 2$ by the Neumann condition $\rB'$. The level set $f=1$
is colored by the defect $\rD$.

This is not quite a valid TQFT picture --- the defect lines cross at saddle
points --- so we make one final change: we erase all color inside a small
disk $D$ around each saddle point.  The boundary of $D$ is now subdivided
into four arcs, alternately colored $\rB,\rB'$, separated by defect
points. (See Figure~\ref{sinsin} in \S\ref{sec:8} below.) Our quadruple
$\cQ:=(\cF,\rB,\rB',\rD)$ assigns to each $D$ a vector space $H_1=\cQ(D)$,
which can be identified\footnote{The identification is canonical up to the
antipodal involution.} with the space of functions on $G$
(\S\ref{computeH1}, Example~\ref{thm:44}), home of the Ising action
$\theta$. 

\begin{remark}\label{frobhopf}
 The space $H_1$ is a \emph{Frobenius Hopf algebra} \cite{Sw, LR,
EGNO}. Figures~\ref{mult}, \ref{quadratics} in \S\ref{sec:8} give a pictorial
construction of the operations.  In certain settings, this reconstructs most
information in $\cQ$ (cf.~\S\ref{hopf}).
 \end{remark}  
 
The relative theory formalism reads our final colored surface as a linear map
from $H_1^{\otimes E}$ (one tensor factor for each edge) to the vector space
$\cF(Y)$. Applying this to the vector $\Theta=\bigotimes_E \theta\in
H_1^{\otimes E}$ gives a vector $\mathbf{Z}(\cF,\theta,
\Lambda)\in\cF(Y)$. For $\cF=\cG_G$, it is the Ising partition function
obtained earlier from the lattice definition.

General correlators are incorporated using line and point defects. Line
defects in $3$D are classified by objects in the category associated to the
circle. Any boundary theory $\rB$ supplies an object $W_\rB\in\cF(S^1)$,
produced by the cylinder colored by $\rB$ at one end (\S\ref{sec:9},
Figure~\ref{fig:13}). This linearly generates a subcategory of line
defects. When $\cF$ stems from a tensor category and $\rB$ comes from $\cF$
as a module over itself, the resulting defects are the \emph{Wilson
defects}. When $\rB'$ is a fiber functor for $\cF$, we call them
\emph{'t~Hooft defects}. When Wilson and 't~Hooft defect lines terminate at
points of matching color on the surface $Y$, they may be promoted to
order/disorder operators. (We postpone the abstract construction to
Definition~\ref{abstractorderdef}.) Therewith, the quadruple $\cQ$ defines a
vector $\mathbf{Z}(\cF,\theta,\Lambda,W,t)$ in the defective $\cF$-space for
$Y$, capturing the full lattice theory in TQFT language.

\csubsection{Electromagnetic duality for Hopf algebras.}\par\label{hopf} The
simplest generalization of gauge theory assumes that $\cF$ is generated by a
tensor category $\mathscr{T}$ which also is $2$-dualizable as a module over
itself, thus defining the regular (Dirichlet) boundary condition $\rB$. If
the category $\mathscr{T}$ is abelian, this forces it to be multi-fusion
\cite{FT2}. A complementary boundary condition $\rB'$ must then be a
tensor functor to $\mathscr{V}ect$; the generating condition
(Remark~\ref{gen}) confirms it to be a \emph{fiber functor} (and also forces
$\mathscr{T}$ to be a \emph{fusion} category). The reconstruction theorem of
\cite{Ha, Os, EGNO} assures us that $\sA$ is the tensor category of finite
modules and co-modules, respectively, of a finite-dimensional semi-simple
Hopf algebra. This is nothing but our friend $H_1$ (see ~\S\ref{sec:8}).

The dual electro-magnetic side $\cF'$ is based on the centralizer category
$\mathrm{End}_{\mathscr{T}} (\mathscr{V}ect)$. The same reconstruction
theorem identifies the latter with the tensor category of $H_1$-comodules.
Duality interchanges the categories of modules and co-modules of $H_1$ and
$H_1^\vee$, which label their Wilson and 't~Hooft defects,
respectively. Interchange of the order/disorder operators relies on their
categorical definitions in \S\ref{sec:8}, and it is now clear for formal
reasons that the duality $\cF\leftrightarrow\cF'$ interchanges the lattice
Ising models of the two theories.
 
If $H_1$ is neither commutative nor co-commutative, then neither theory
$\cF,\cF'$ has a classical field theory interpretation. This gives a
\emph{quantum version} of the Ising model in $2$ dimensions, with its gauge
coupling.

\csubsection{More general bi-colored theories.}\par\label{abstractkoszul} When
the polarization $(\rB,\rB')$ is defined by modules
$\mathscr{M},\mathscr{M}'$, the complementarity
$\mathrm{Hom}_\mathscr{T}(\mathscr{M},\mathscr{M}') = \mathbf{1}$ and the
generating condition allow us to Morita convert the quadruple
$\cQ=(\cF_\mathscr{T},\rB,\rB',\rD)$ into the two triples $(\cE,
\mathbf{1},\cE)$ and $(\cE', \cE', \mathbf{1})$ built from the centralizer
categories $\mathscr{E},\mathscr{E}'$ of $\mathscr{T}$ in
$\mathscr{M},\mathscr{M}'$; see Remark~\ref{remark8.6}.  The fourth members of each quartet, the relevant
defects, come from the obvious identifications
$\mathrm{Hom}_{\cE}(\cE,\mathbf{1}) =\mathbf{1}$ (and primed).  We recognize
the earlier electromagnetic duality as the equivalence $\cE\simeq\cE'$, which
interchanges the fiber functor $\mathbf{1}$ with the regular boundary
conditions. Specifically, as in the Hopf situation, \[ \mathscr{E}'
=\mathrm{End}_{\mathscr{E}}(\mathbf{1}), \qquad \mathscr{E}
=\mathrm{End}_{\mathscr{E}} (\mathbf{1}).  \]

\csubsection{Higher-dimensional duality.}\par \label{spectra} Electromagnetic
duality for finite abelian groups generalizes to \emph{higher gauge theories}
in higher dimension. Just as gauge theory with finite group $A$ theory counts
flat bundles, which are maps, up to homotopy, to the classifying space $BA$,
higher gauge theories count classes of maps into the higher Eilenberg-MacLane
spaces $B^rA = K(A;r)$. These are more familiar as the $r$th cohomology
classes with values in $A$. In space-time dimension $(d+1)$, $d+1\ge 2$, the
duality identifies the theories with targets $B^rA$ and $B^{d-r}A^\vee$, for
any $r\le d$. These spaces are in a generalized Pontrjagin duality, induced
from the pairing $A\times A^\vee \to \bC^\times$. A categorified Fourier
transform gives an equivalence between the respective field theories
\cite{FHLT}.  We spell this out in~\S\ref{sec:7}.

When $0< r< d$, we can replicate our constructions above to place lattice
boundary theories for these in a higher Kramers-Wannier duality: this will
relate theories valued in $B^{r-1}A$ and $B^{d-r-1}A$, in dimension
$d$. Ising has $d=2, r=1$. These theories involve ordinary $A$-valued
homology and cohomology groups, and there is a natural lattice formulation of
this higher duality: see for instance \cite[\S6.1.4]{ID}.  However, one
illustration of our abstract formulation in \S\ref{polarized} gives a
\emph{homotopical} version of these dualities, in which the target fields are
valued into any spectrum $\boldsymbol{T}$ with finite homotopy groups, rather
than in an Eilenberg-MacLane $K(A;r)$. Such a spectrum defines a generalized
homology theory $X\mapsto H_\bullet(X;\boldsymbol{T})$, and has a
\emph{Pontrjagin dual spectrum}, which represents the generalized cohomology
theory given by the (shifted) Pontrjagin dual groups $X\mapsto H^\bullet
(X;\boldsymbol{T}^\vee) := H_{d-\bullet}(X;\boldsymbol{T})^\vee$.
Generalized (co)homologies do not admit a chain/cochain formulation, and are
difficult to express explicitly in lattice format; the conversion in
\S\ref{polarized} to a handle decomposition of the manifold offer a
substitute TQFT method for their construction.

\begin{remark}  
 It should be clear now how to extend the higher-dimensional construction in
\S\ref{polarized} to higher dimension, but one detail stands out. The
quartered disk $D$ is really a $1$-handle, with the attaching faces colored
by $\rB$ and complementary faces $\rB'$. There is also a $0$-handle $D_0$ and
a $2$-handle $D_2$, with monochrome boundaries $\rB, \rB'$, respectively. The
gauge theory spaces $\cQ(D_{0,2})$ are $1$-dimensional, and inserting an
action there would only change Ising theory by an overall scale
factor. However, in higher dimension one must consider all handles.
 \end{remark}

   \section{Review of field theory concepts}\label{sec:2}
% lastsubsec@000

Relativistic field theories are formulated on Minkowski spacetime.  A key
property---positivity of energy---enables Wick rotation to Euclidean space.
An energy-momentum tensor gives deformations away from Euclidean space; a
strong from is an extension of the theory to arbitrary Riemannian manifolds.
The resulting structure has been axiomatized, originally by Segal~\cite{Se}
in the case of 2-dimensional conformal field theories, and these axioms have
been elaborated and extended in many directions, particularly for
\emph{topological} field theories.  In its most basic form an
$n$-dimensional\footnote{$n$~is the dimension of spacetime} topological field
theory is a map
  \begin{equation}\label{eq:1}
     \rF\:\tbord\longrightarrow \Vect_{\CC} 
  \end{equation}
with (i)~domain the bordism category whose objects are closed
$(n-1)$-manifolds and morphisms are (compact) bordisms between them,
(ii)~codomain the linear category $\Vect=\Vect_{\CC}$ of finite dimensional
complex vector spaces and linear maps, and (iii)~$\rF$~a symmetric monoidal
functor which maps disjoint unions to tensor products.  See~\cite{At} for an
early exposition.  The map~$\rF$ encodes the state spaces, point operators,
correlation functions, and partition functions of a field theory.  However,
it does not capture extended operators---line operators, surface operators,
etc.---nor does it capture the full locality of field theory.  An
\emph{extended field theory} is a map
  \begin{equation}\label{eq:2}
     \rF\:\bord n\longrightarrow \sC 
  \end{equation}
from an $n$-category\footnote{or $(\infty ,n)$-category} of bordisms to some
target $n$-category.  The \emph{cobordism hypothesis}~\cite{BD,L} tells that
an extended topological field theory is determined by its value on a point,
which is an $n$-dualizable object ~$C\in \sC$.  For theories of unoriented
manifolds the object~$C$ has $O_n$-invariance data for the canonical action
of~$O_n$ on the $\infty $-groupoid of $n$-dualizable objects.  We refer the
reader to \cite{L,T1,F1,F2} for motivation, elaboration, and examples.

  \begin{example}[$\sC=\TensCat$]\label{thm:54}
 In this paper we mostly consider $n=3$~dimensional extended topological
field theories with codomain the 3-category $\tcat$ of complex linear tensor
categories (enriched over~$\cVect$).  The paper~\cite{DSPS} develops the
theory of $\tcat$ (over arbitrary ground fields); see~\cite{EGNO} for
background on tensor categories.  An object of~$\tcat$ is a tensor
category~$\sA$, a 1-morphism $\sA\to\sA'$ is an $(\sA',\sA)$-bimodule
category, a 2-morphism is a linear functor commuting with the bimodule
actions, and a 3-morphism is a natural transformation of functors.  A tensor
category~$\sA$ is \emph{fusion} if it is semisimple, satisfies strong
finiteness properties, and the vector space~$\EA 1\cong \CC\cdot \id_1$,
where $1\in \sA$ is the unit object.  A fusion category is 3-dualizable.  A
(pivotal structure) on~$\sA$ is a tensor isomorphism~$\rho $ from the
identity functor~$\id_{\sA}$ to the double dual functor.  The
\emph{dimension} of an object~$x\in \sA$ is then the composition 
  \begin{equation}\label{eq:110}
     \dim(x)\:1\xrightarrow{\;\;\coev_x\;\;} x\otimes
     x\dual\xrightarrow{\;\;\rho _x\otimes \id\;\;} x^{\vee\vee}\otimes x\dual
     \xrightarrow{\;\;\ev_{x\dual}\;\;} 1 
  \end{equation}
in~$\EA 1$ of coevaluation, pivotal structure, and evaluation.  A pivotal
structure provides $SO_2$-invariance data on~$\sA$.  A pivotal structure is
\emph{spherical}~\cite{BW1} if $\dim(x)=\dim(x\dual)$ for all~$s\in \sA$, in
which case $\sA$~is $SO_3$-invariant.  
  \end{example}

We introduce \emph{boundary theories} \cite[Example~4.3.22]{L} and
\emph{domain walls} \cite[Example~4.3.23]{L}.

  \begin{definition}[]\label{thm:1}
 Let $\sC$~be an $n$-category, $C\in \sC$ an $n$-dualizable object equipped
with\footnote{We can and will use other tangential structures, such as
orientation.} $O_n$-invariance data, and $\rF_C\:\bord n\to\sC$ the
corresponding topological field theory.

 \begin{enumerate}[label=\textnormal{(\roman*)}]

 \item \emph{Topological boundary data} for~$\rF_C$ is an $(n-1)$-dualizable
morphism $B \:1\to C$ in~$\sC$ equipped with $O_{n-1}$-invariance data.

 \item Suppose $B,B'\:1\to C$ are topological boundary data.  Then
\emph{domain wall data} from~$B$ to~$B'$ is an $(n-2)$-dualizable
2-morphism $D\:B\to B'$ in~$C$ equipped with $O_{n-2}$-invariance data.
  \end{enumerate}
  \end{definition}

\noindent
 An extension of the cobordism hypothesis~\cite[\S4.3]{L} gives from~(i) the
associated \emph{topological boundary theory}, which is a natural
transformation of functors $\bord{n-1}\to\sC$, denoted
  \begin{equation}\label{eq:3}
     \rB _B \:1\longrightarrow \tau \mstrut _{\le n-1}\rF_C, 
  \end{equation}
where the truncation~ $\tau \mstrut _{\le n-1}\rF_C$ is the composition
$\bord{n-1}\longrightarrow \bord n\xrightarrow{\;\rF_C\;}\sC$, and $1$~is the
trivial theory.  Then $B $~is the value of~$\rB _B $ on a point.  For
example, if $X$~is a compact manifold with boundary, and $X$~as a bordism has
$\partial X$~incoming, then the pair~$(\rF_C,\rB _B )$ evaluates on~$X$ to
  \begin{equation}\label{eq:4}
     \CC\xrightarrow{\;\;\rB _B
     (\bX)\;\;}\rF_C(\bX)\xrightarrow{\;\;\rF_C(X)\;\;} \CC. 
  \end{equation}
This is the usual picture in physics of a boundary theory.  For more about
the categorical aspects of this definition, see~\cite{JFS}.  From~(ii) we
obtain a \emph{domain wall theory} from~$B$ to~$B'$: 
  \begin{equation}\label{eq:111}
     \rD_D\:\rB\longrightarrow \rB'. 
  \end{equation}
Such a theory can be evaluated on manifolds with codimension two corners, as
we discuss in~\S\ref{sec:9}.

  \begin{remark}[]\label{thm:2}
 Let~$n=3$, $\sC=\tcat$, and suppose $\sA$~is a 3-dualizable tensor category,
for example a fusion category.  Then a map $1\to\sA$ in~$\tcat$ is a linear
category~$\sL$ equipped with a left $\sA$-module structure.  There are
dualizability constraints on the left $\sA$-module.  There is a canonical
example, namely~$\sL=\sA$: the tensor structure on~$\sA$ induces a module
structure on~$\sL$.  We explore this ``regular boundary theory''
in~\S\ref{sec:9}.

  \end{remark}

The map~$\rB _B $ defines a \emph{relative} $(n-1)$-dimensional
theory~\cite{FT1}.  For example, its value on a closed $(n-1)$-manifold~$Y$ is
a linear map
  \begin{equation}\label{eq:5}
     \rB _B (Y)\:\CC\longrightarrow \rF_C(Y), 
  \end{equation}
or simply an element of the vector space~$\rF_C(Y)$.  (Here we assume the
looping~$\Omega ^{n-1}\sC$ is equivalent to~$\Vect_{\CC}$.)  The relative
theory on~$Y$ is the value of the pair~$(\rF_C,\rB _B )$ on $[0,1]\times Y$,
where we take $\{0\}\times Y$ as incoming and equipped with the topological
boundary theory~$\rB _B $; by contrast, $\{1\}\times Y$ is outgoing and is
free---no boundary theory.  This relation between the relative theory and
boundary theory works for any morphism~$M$ in~$\bord{n-1}$ in place of~$Y$.
For the relative theory we only use the truncation~$\tau \mstrut _{\le
n-1}\rF_C$, but to evaluate the pair~$(\rF_C,\rB _B )$ on arbitrary
$n$-dimensional bordisms, as in~\eqref{eq:4}, we use the full theory~$\rF_C$.

   \section{Three-dimensional finite gauge theories and electromagnetic duality}\label{sec:3}
% lastsubsec@  5

  \subsection{Finite gauge theory and topological boundary conditions}\label{subsec:3.1}

Let $G$~be a finite group.  We construct finite gauge theory with gauge
group~$G$ as an extended field theory 
  \begin{equation}\label{eq:6}
     \FG\:\bord3\longrightarrow \tcat. 
  \end{equation}
It is a theory of unoriented manifolds.  It can be defined using the
cobordism hypothesis by declaring
  \begin{equation}\label{eq:7}
     \FG(\pt) = \cVect[G], 
  \end{equation}
where $\cVect[G]$~is the tensor category of finite rank complex vector
bundles over~$G$ with convolution product: if $W,W'\to G$ are vector bundles,
then
  \begin{equation}\label{eq:8}
     (W*W')_h=\bigoplus\limits_{gg'=h}W\mstrut _g\otimes W'_{g'},\qquad h\in
     G; 
  \end{equation}
the convolution product of morphisms is defined similarly.  One may
regard~$\cVect[G]$ as the ``group ring'' of~$G$ with coefficients
in~$\cVect$.  Write $O_3\cong \pmo\times SO_3$.  Then $\pmo$-equivariance
data on $\cVect[G]$ is the equivalence $\cVect[G]\xrightarrow{\;\cong \;}
\cVect[G]^{\textnormal{op}}$ obtained by pullback along inversion $x\mapsto
x\inv $ on~$G$.  The dual of $W\to G$ in~$\cVect[G]$ can be identified with
the bundle whose fiber at~$x\in G$ is~$W^*_{x\inv }$, so there is a natural
map from~$W$ to its double dual.  This defines a pivotal structure.  The
identity object $\bone\to G$ in~$\cVect[G]$ has 
  \begin{equation}\label{eq:112}
     \bone_g=\begin{cases} \CC,&g=e;\\0,&g\neq e.\end{cases} 
  \end{equation}
The dimension of~$W\to G$ is $\sum_{g\in G}\dim(W_g)$.  The pivotal structure
is spherical.

The theory~$\FG$ can also be constructed from a classical model using a
finite path integral~\cite{F3,FHLT}.  The partition function of a closed
3-manifold~$X$ counts the isomorphism classes of principal $G$-bundles $P\to
X$.  For a manifold~$M$ define~$\bung M$ as the groupoid whose objects are
principal $G$-bundles $P\to M$ and morphisms are isomorphisms
covering~$\id_M$.  Then
  \begin{equation}\label{eq:9}
     \FG(X)= \sum\limits_{[P]\in \pi _0\bung X}\frac{1}{\#\Aut P}. 
  \end{equation}
If $Y$~is a closed 2-manifold, then 
  \begin{equation}\label{eq:10}
     \FG(Y)=\Fun\bigl(\bung Y \bigr) 
  \end{equation}
is the vector space of complex functions on (isomorphism classes of)
principal $G$-bundles over~$Y$.  If $S$~is a closed 1-manifold, then 
  \begin{equation}\label{eq:45}
     \FG(S) = \Vect\bigl(\bung S \bigr) 
  \end{equation}
is the linear category of complex vector bundles over the groupoid~$\bung S$.
It is more subtle to sum over the groupoid~$\bung \pt$ to obtain 
  \begin{equation}\label{eq:46}
     \FG(\pt)=\cVect[G],
  \end{equation}
the starting point of the construction with the cobordism hypothesis.
 
There is a topological boundary theory
  \begin{equation}\label{eq:11}
     \BH\:1\longrightarrow \FG 
  \end{equation}
attached to a subgroup~$H\subset G$, and it has a classical description.
Namely, for a manifold~$M$ let $\bunh M$ be the groupoid whose objects are a
principal $G$-bundle $P\to M$ together with a section~$\sigma $ of the fiber
bundle $P/H\to M$ with fiber~$G/H$;, equivalently, $\bunh M$~is the groupoid
of principal $H$-bundles over~$M$.  In physical terms $\BH$~ is a gauged
$\sigma $-model on the homogeneous space~$G/H$, but we do not sum over the
$G$-bundle.  That is, there is a forgetful\footnote{Forget the
section~$\sigma $.  Alternatively, \eqref{eq:12}~maps a principal $H$-bundle
to its associated principal $G$-bundle.} map
  \begin{equation}\label{eq:12}
     \bunh M\longrightarrow \bung M 
  \end{equation}
and the topological boundary theory~$\BH$ sums over the fibers
of~\eqref{eq:12}.  (The fibers are the relative fields of `relative field
theory': we work \emph{relative} to the base.)  A central extension of~$H$
by~$\TT$ determines a twisted version of~$\BH$, which is also a boundary
theory.

There are two extreme cases.  If $H=\{e\}$ then we sum over trivializations;
if $H=G$ then \eqref{eq:12}~is an isomorphism and we have the \emph{free}
boundary condition.  We call these \emph{Dirichlet} and \emph{Neumann}
boundary theories, respectively.  For~$M=\pt$ the sum over the fibers
of~\eqref{eq:12} produces the $\cVect[G]$-module $\cVect(G/H)$ of finite rank
complex vector bundles over~$G/H$.  If $W\to G$ and $V\to G/H$ are vector
bundles, then the module product is
  \begin{equation}\label{eq:13}
     (W*V)_{yH} = \bigoplus\limits_{x(x'H)=yH}W\mstrut _x\otimes V_{x'H}. 
  \end{equation}
The regular boundary theory described in Remark~\ref{thm:2} and explored
in~\S\ref{sec:9} corresponds to~$H=\{e\}$.

  \subsection{The Koszul dual to finite gauge theory}\label{subsec:3.5}

There is another extended 3-dimensional topological field theory 
  \begin{equation}\label{eq:14}
     \hFG\:\bord3\longrightarrow \tcat 
  \end{equation}
defined using the cobordism hypothesis by declaring that the value 
  \begin{equation}\label{eq:15}
     \hFG(\pt) = \cRep(G) 
  \end{equation}
on a point is the tensor category of finite dimensional complex
representations of~$G$.  The $\pmo$-equivariance data maps a
representation~$\hW$ to its dual~$\hW^*$.  The $SO_3$-equivariance data is
the spherical structure defined by the usual map of a representation into its
double dual.  If $G$~is abelian the theory~$\hFG$ is the quantization of
gauge theory for the Pontrjagin dual to~$G$, see~\S\ref{subsec:3.2}; we do
not know a classical theory whose quantization is~$\hFG$ if $G$~is
nonabelian.

  \begin{proposition}[]\label{thm:4}
 There is a Morita equivalence $\cVect[G]\simeq  \cRep(G)$, i.e., an
equivalence of module 2-categories
  \begin{equation}\label{eq:16}
     \sF\:\textnormal{$\Vect[G]$-mod}\xrightarrow{\;\;\cong \;\;}
     \textnormal{$\Rep(G)$-mod}.  
  \end{equation}
For a subgroup $H\subset G$ the image of $\Vect(G/H)$ under~$\sF$ is the
category~$\Rep(H)$ of finite dimensional complex representations of~$H$.
  \end{proposition}

\noindent
 Let $W$~be a representation of~$G$ and $V$~a representation of ~$H$.  Then
the module product of~$W$ and~$V$ is the $H$-representation $i^*W\otimes V$,
where $i^*W$~is the restriction of~$W$ to a representation of~$H\subset G$.

  \begin{remark}[]\label{thm:23}
 Let $j\:\Bord_3(SO_3)\to\Bord_3$ be the forgetful map from the oriented
bordism category to the unoriented bordism category.  The Morita equivalence
is $SO_3$-invariant, so by the cobordism hypothesis defines an equivalence
$\FG\circ j\xrightarrow{\;\simeq \;}\hFG\circ j$ of \emph{oriented}
3-dimensional field theories.  The equivalence of unoriented theories
requires a twist by the orientation sheaf on one side or the other; see
Remark~\ref{thm:29} for the abelian case.
  \end{remark}

  \begin{remark}[]\label{thm:47}
 See~ \cite[Example 7.12.19]{EGNO} for another account (with a different
definition of Morita equivalence).
  \end{remark}

  \begin{proof}
 The Morita equivalence is implemented by the invertible
$\bigl(\Rep(G),\Vect[G] \bigr)$-bimodule category $\Vect_{G_L}(G)$ of finite
rank complex vector bundles over~$G$ equivariant for the left multiplication
action of~$G$ on~$G$; the inverse $\bigl(\Vect[G],\Rep(G) \bigr)$-bimodule
category $\Vect_{G_R}(G)$ uses right multiplication.  The transform
of~$\Rep(H)$ is the $\Vect[G]$-module
  \begin{equation}\label{eq:17}
     \Vect_{G_R}(G)\otimes \mstrut _{\Rep(G)}\Rep(H). 
  \end{equation}
Each category in~\eqref{eq:17} is the category of finite rank complex vector
bundles on a finite global quotient stack~$X\gpd K$ of a finite group~$K$
acting on a finite set~$X$, as summarized in the diagram 
  \begin{equation}\label{eq:18}
     \begin{gathered} \xymatrix{G\gpd G_R\ar[dr]&& \pt\gpd H\ar[dl]\\
     &\pt\gpd G} 
     \end{gathered} 
  \end{equation}
The monoidal structure on $\Rep(G)=\Vect(\pt\gpd G)$ is tensor product.  In
this situation \eqref{eq:17}~ is the category of finite rank complex vector
bundles on the fiber product of~\eqref{eq:18}, which is the set~$G/H$; see
\cite[Theorem 1.2]{BFN} for a much more general statement.  
  \end{proof}

  \begin{remark}[]\label{thm:24}
 The Morita equivalence~\eqref{eq:16} also fits into a general picture.
Suppose $\pi \:X\to Y$ is an essentially surjective\footnote{Every object~$y$
of~$Y$ is equivalent to the image of an object of~$X$ under~$\pi $} map of
finite groupoids.  The fiber product $X\times _YX$ is the set of arrows in a
groupoid~$\sG$ with set of objects~$X$, and $\sG$~is equivalent to~$Y$.
Pullback $\pi ^*\:\Vect(Y)\to\Vect(\sG)$ is an equivalence of categories; the
inverse is descent using the equivariance data.  (See~\cite[\S A.3]{FHT}, for
example.)  In other words, there is a Morita equivalence of the commutative
algebra~$\Fun(Y)$ of functions on~$Y$ under pointwise multiplication with the
(convolution) groupoid algebra~$\CC[X\times _YX]$ of~$\sG$.  For example,
$Y=\pt$ and $X$~a finite set reduces to the Morita triviality of a matrix
algebra.  Functions on~$X$ form an invertible bimodule which exhibits the
Morita equivalence.  The Morita equivalence~\eqref{eq:16} is the once
categorified version, applied to the surjective map $\pi \:\pt\to\pt\gpd G$
of finite \emph{stacks}.  See~\cite{BG} for non-discrete generalizations.
  \end{remark}

  \subsection{Abelian duality as Pontrjagin duality}\label{subsec:3.2}

To begin we recall that associated to a finite abelian group~$\pA$ is its
\emph{Pontrjagin dual} group\footnote{$\TT$~is the group of unit norm complex
numbers.}  $\pA\dual=\Hom(\pA,\TT)$ of characters.  The double dual of~$M$ is
canonically isomorphic to~$M$.  Furthermore, the \emph{Fourier transform}
  \begin{equation}\label{eq:19}
     \sF\:\Fun(\pA)\longrightarrow \Fun(\pA\dual) 
  \end{equation}
is an isomorphism of the vector spaces of functions on~$\pA$ and~$\pA\dual$.
It is defined by convolution with the complex conjugate of the universal character
  \begin{equation}\label{eq:20}
     \chi \:\pA\times \pA\dual\longrightarrow \CC, 
  \end{equation}
up to a numerical factor:
  \begin{equation}\label{eq:21}
  \begin{aligned}
   \sF\:\Fun(\pA)&\longrightarrow \Fun(\pA\dual)\\
     f&\longmapsto \Bigl(a\dual\mapsto
     \frac{1}{\sqrt{\#\pA}}\,\sum\limits_{a\in \pA}\overline{\chi 
     (a,a\dual)}\,f(a)\Bigr).
  \end{aligned}
  \end{equation}
In terms of the correspondence diagram 
  \begin{equation}\label{eq:54}
     \begin{gathered}
     \xymatrix{&(\pA\times \pA\dual,\bar\chi )\ar[dl]_p\ar[dr]^q\\\pA&&\pA\dual}
  \end{gathered}
  \end{equation}
the map~\eqref{eq:21} is, up to a factor, the composition $q_*\circ \bar\chi
\circ p^*$ acting on functions; the inverse uses ~$\chi $ as integral kernel.

Returning to finite gauge theory, if the gauge group $G=A$~is finite abelian,
then there is a natural equivalence $\Rep(A)\simeq \Vect[A\dual]$; tensor
product maps to convolution.  If $B\subset A$~is a subgroup, then
$\Rep(B)\simeq \Vect[B\dual]$ and we identify $B\dual\cong A\dual/B^\perp$,
where $B^\perp\subset A\dual$ is the annihilator of~$B$.  Therefore, for
abelian groups the Morita equivalence~\eqref{eq:16} reduces to the duality
map $A\longleftrightarrow A\dual$ on abelian groups and the annihilator map
$B\longleftrightarrow B^\perp$ on subgroups.
 
The Morita equivalence Proposition~\ref{thm:4} in the abelian case is an
instance of electromagnetic duality.  One expression of the latter is a
field-theoretic Fourier transform \cite[Lecture 8]{W1}, which we adapt to
3-dimensional finite gauge theories via a correspondence diagram
  \begin{equation}\label{eq:22}
     \begin{gathered}
     \xymatrix{&\tG_{A,A\dual\!,\,BA}\ar[dl]_p\ar[dr]^q\\\rG_A&&\rG_{A\dual}}
     \end{gathered} 
  \end{equation}
of equivalences of extended 3-dimensional topological field theories.  Each
is a finite path integral.  As in~\eqref{eq:9} the theories~$\rG_A,\rG_{A\dual}$
count principal bundles.  The theory~$\tG_{A,A\dual\!,\,BA}$ has three
classical fields on a manifold~$M$: a principal $A\dual$-bundle $P\dual\to
M$, an $A$-gerbe $\sG\to M$, and a trivialization $P\to M$ of $\sG\to M$.
Whereas $\rG_A,\rG_{A\dual}$~are theories of unoriented manifolds,
$\tG_{A,A\dual\!,\,BA}$~is a theory of oriented manifolds.  The exponentiated
action on a closed oriented 3-manifold~$X$ is ostensibly
  \begin{equation}\label{eq:23}
     e^{iS_X}(P,P\dual,\sG) = \langle [P\dual]\cup[\sG]\,,\,[X] \rangle, 
  \end{equation}
where $[P\dual]\in H^1(X;A\dual)$, $[\sG]\in H^2(X;A)$, and $[X]\in H_3(X)$ is the
fundamental class.  But, in fact, the action is trivial since the existence
of~$P$ forces $[\sG]=0$.  The equivalence~$p$ is obtained by summing over~$P\dual$
and then over~$\sG$; these sums give canceling factors\footnote{The first
ratio is the number of isomorphism classes of $A\dual$-bundles divided by the
number of automorphisms of each.  The second is the reciprocal of the number
of automorphisms of an $A$-gerbe, accounting for automorphisms of
automorphisms.}  
  \begin{equation}\label{eq:24}
     \frac{\#H^1(X;A\dual)}{\#H^0(X;A\dual)}\;\frac{\#H^0(X;A)}{\#H^1(X;A)}
     =1. 
  \end{equation}
We can take~$\sG$ trivial and so are reduced to~$\rG_A$.  The equivalence~$q$
is obtained by summing over~$P$ and then over~$\sG$; the sums contribute 
  \begin{equation}\label{eq:25}
      \frac{\#H^1(X;A)}{\#H^0(X;A)}\;
      \frac{\#H^0(X;A)}{\#H^1(X;A)} = 1. 
  \end{equation}

  \begin{remark}[]\label{thm:5}
 If we perform a similar duality in even dimensions, then the factors do not
cancel and we pick up~$(\#A)^{\chi(X)}$, where $\chi(X)$~is the Euler number
of~$X$.  In other words, the duality involves tensoring with an invertible
Euler theory.  The Euler factor also occurs in electromagnetic duality with
continuous abelian gauge groups, for example in~\cite{W2}.
  \end{remark}

Another picture: The Morita equivalence is an invertible 2-dimensional
\emph{domain wall} between~$\rG_A$ and~$\rG_{A\dual}$.  In the abelian case it is
most easily expressed in the language of homotopy theory, as we
explain in~\S\ref{sec:7}.  Briefly, consider the correspondence diagram of
pointed spaces and cocycles
  \begin{equation}\label{eq:55}
     \begin{gathered} \xymatrix{&(BA\times BA\dual,c
     )\ar[dl]_(.55)p\ar[dr]^(.55)q\\BA&&BA\dual} \end{gathered} 
  \end{equation}
where $c\in Z^2(BA\times BA\dual;\TT)$ represents the cohomology class of the
canonical Heisenberg extension of~$A\times A\dual$.  The 3-dimensional
theory~$\rG_A$ is constructed by summing over homotopy classes of maps to~$BA$,
the 3-dimensional theory~$\rG_{A\dual}$ by summing over homotopy classes of
maps to~$BA\dual$, and the Morita isomorphism via the correspondence
diagram.  For example, if $Y$~is a closed \emph{oriented} surface,  then
  \begin{equation}\label{eq:36}
     \rG_A(Y) = \Fun\bigl(\Bun_A(Y) \bigr) \cong \Fun\bigl(H^1(Y;A) \bigr) 
  \end{equation}
and the correspondence diagram~\eqref{eq:55} induces\footnote{Take homotopy
classes of maps of~$Y$ into~\eqref{eq:55} to form the correspondence
diagram~\eqref{eq:54} with $\pA=H^1(Y;A)$.}  an isomorphism
  \begin{equation}\label{eq:37}
     \sF\:\rG_A(Y)\longrightarrow \rG_{A\dual}(Y). 
  \end{equation}
There is a Pontrjagin-Poincar\'e duality pairing between $H^1(Y;A)$ and
$H^1(Y;A\dual)$ using cup product, the Pontrjagin duality pairing
$A\dual\times A\to\TT$, and the fundamental class.  Up to a factor,
\eqref{eq:37}~is the Fourier transform~\eqref{eq:21}.

See \cite[\S6.1.4]{ID} for an alternative account of abelian duality using chain
complexes. 

  \begin{remark}[]\label{thm:29}
 If $Y$~is a closed manifold which is not necessarily oriented, then there is
a Pontrjagin-Poincar\'e duality pairing between~$H^1(Y;A)$
and~$H^1(Y;\widetilde{A\dual})$, where $\widetilde{A\dual}$~is the local
system~$\Hom(A,\widetilde{\TT})$ and $\widetilde{\TT}\to Y$ is associated to
the orientation double cover.  This is part of an equivalence of unoriented
theories, the latter a gauge theory of orientation-twisted principal
$A\dual$-bundles.
  \end{remark}

  \subsection{Loop operators}\label{subsec:3.3}

Let $G$~be a finite group and consider the finite gauge theory~$\FG$
of~\S\ref{subsec:3.1}.  Recall that there is a classical model with fields
the groupoid~$\Bun_G$ of principal $G$-bundles; $\FG$~is computed by summing
over~$\Bun_G$, as in~\eqref{eq:9}.  In general, (finite) path integrals can
be modified by \emph{operator insertions}, and in this 3-dimensional finite
gauge theory there are two distinguished classes of \emph{loop operators}
associated to 1-dimensional submanifolds of a closed 3-manifold~$X$.  The
\emph{Wilson operator} is an insertion into the sum, whereas the \emph{'t
Hooft operator} alters the groupoid of $G$-bundles.
 
The Wilson operator is defined for $S\subset X$  an oriented connected
1-dimensional submanifold and $\chi \:G\to\TT$ a character of~$G$.  Then the
function
  \begin{equation}\label{eq:29}
     h_{S,\chi }\:\bung X\longrightarrow \CC 
  \end{equation}
maps a principal $G$-bundle $P\to X$ to $\chi $~applied to the
holonomy.\footnote{The holonomy is determined up to conjugacy and $\chi $~is a
class function.}  The finite path integral~\eqref{eq:9} with Wilson operator
~$(S,\chi )_W$ inserted is 
  \begin{equation}\label{eq:30}
     \FG\bigl(X;(S,\chi )_W \bigr) = \sum\limits_{[P]\in \pi _0\bung
     X}\frac{h_{S,\chi }(P)}{\#\Aut P} . 
  \end{equation}

  \begin{remark}[]\label{thm:9}
 Since $\FG$~is a theory of unoriented manifolds, we should not need an
orientation on~$S$ to define the loop operator.  Indeed, we can drop the
orientation and replace~$\chi $ by a function from orientations of~$S$ to
characters of~$G$ which inverts the character when the orientation is reversed. 
  \end{remark}

The 't Hooft operator is defined for $S\subset X$ a co-oriented connected
1-dimensional submanifold and $\gamma \subset G$~a conjugacy class.  Define
the groupoid $\Bun_G\bigl(X;(S,\gamma ) \bigr)$ whose objects are principal
$G$-bundles $P\to X\setminus S$ with holonomy~$\gamma $ around an oriented
linking curve to~$S$.  The finite path integral~\eqref{eq:9} with 't~Hooft
operator~$(S,\gamma )_H$ inserted is
  \begin{equation}\label{eq:31}
     \FG\bigl(X;(S,\gamma )_H \bigr) = \sum\limits_{[P]\in \pi
     _0\Bun_G(X;(S,\gamma ))}\frac{1}{\#\Aut P} .  
  \end{equation}
As in Remark~\ref{thm:9} we can drop the co-orientation and replace~$\gamma $
by a function from co-orientations to conjugacy classes which inverts under
co-orientation reversal.
 
The abstract line operators in an extended $n$-dimensional field theory~$\rF$
are objects in the category~$\rF(S^{n-2})$.  For the 3-dimensional finite gauge
theory 
  \begin{equation}\label{eq:32}
     \FG(\cir)\simeq \Vect_G(G), 
  \end{equation}
where $\Vect_G(G)$~is the category of vector bundles on~$G$ equivariant for
the conjugation action; this is ~\eqref{eq:45} for~$S=\cir$.  The
category~$\Vect_G(G)$ is the Hochschild homology, or in this case also the
Drinfeld center, of $\Vect[G]$; see~\cite[Example~8.5.4]{EGNO}.  The Wilson
loop operators form the full subcategory of equivariant vector bundles
supported at the identity~$e\in G$, which is equivalent to the
category~$\Rep(G)$.  The 't~Hooft operators form the full subcategory of
equivariant vector bundles on which the centralizer~$Z_x(G)$ of each~$x\in G$
acts trivially on the fiber at~$x$.  The general loop operator is an amalgam
of these two extremes.

  \begin{remark}[]\label{thm:60}
 Let~$\rF$ be an $n$-dimensional extended topological field theory and
$S\subset X$ a connected 1-dimensional submanifold of an $n$-manifold~$X$.
The link of~$S$ at each point is diffeomorphic to~$S^{n-2}$, but there is no
preferred diffeomorphism.  Furthermore, the group of diffeomorphisms
of~$S^{n-2}$ may act nontrivially on~$\rF(S^{n-2})$.  Therefore, to specify a
loop operator on~$S$ it is not sufficient\footnote{For example, typically in
3-dimensional Chern-Simons theory~\cite{W3} one imposes a normal framing
of~$S$ to rigidify the $SO_2$-action (Dehn twist, ribbon structure)
on~$\rF(S^1)$.} to give an object of~$\rF(S^{n-2})$.  For the finite gauge
theories~$\rF$ considered in this paper the objects in~$\rF(\cir)$
corresponding to Wilson and 't Hooft operators are $SO_2$-invariant, so no
normal framing is required.  See~\S\ref{subsec:8.1a} for further discussion. 
  \end{remark}

Now suppose $G=A$ is a finite abelian group.  Then the conjugation action
of~$A$ on~$A$ is trivial, and we identify 
  \begin{equation}\label{eq:35}
     \Vect_A(A)\simeq \Vect(A\times A\dual) 
  \end{equation}
by decomposing the representation of~$A$ on each fiber.  Wilson operators are
labeled by vector bundles pulled back under the projection $A\times A\dual\to
A\dual$; 't~Hooft operators by vector bundles pulled back under the
projection $A\times A\dual\to A$.  Electromagnetic duality exchanges~$A$
and~$A\dual$, so exchanges Wilson and 't~Hooft operators.

There are also ``loop'' operators on a compact 3-manifold~$X$ with nonempty
boundary for compact 1-dimensional submanifolds whose boundary is contained
in~$\partial X$ and which intersect~$\partial X$ transversely; such
submanifolds are termed `neat'.  First, let $Y$~be a closed 2-manifold and
fix distinct points $y_1,\dots ,y_k\in Y$.  Excise a small open disk about
each~$y_i$ to form a compact 2-manifold~$Y'$ with $\partial Y'$~diffeomorphic
to the disjoint union of $k$~circles.  Fix a diffeomorphism $(\cir)^{\amalg
k}\to \partial Y'$.  Then viewing~$\partial Y'$ as incoming, the extended
field theory~$\FG$ assigns to~$Y'$ a functor
  \begin{equation}\label{eq:86}
     \FG(Y')\:\VGG\times \cdots\times \VGG\longrightarrow \Vect 
  \end{equation}
Therefore, if each~$y_i$ is labeled by an object of~$\VGG$, then we obtain a
vector space.  

Let $S\subset X$ be a connected oriented normally framed neat 1-dimensional
submanifold, labeled by $W\in \VGG$.  We interpret the result of
applying~$\FG$ to this situation by constructing a 2-morphism in the bordism
category.  Excise a tubular neighborhood~$\nu _S$ of~$S$---a solid
cylinder---to obtain a 3-manifold~$X'$ with corners.  Then $\partial
(X')=Y'\cup_{\cir\amalg \cir}\partial _0\nu _S$, where $Y'$~ is $\partial X$~
with open disks about~$\partial S$ excised and $\partial _0\nu _S$~is a
cylinder---the boundary of~$\nu _S$ with the open disks removed.  Fix a
diffeomorphism $(\cir)^{\amalg 2}\to\partial Y'$.  Then in the bordism
category we obtain the diagram of morphisms
  \begin{equation}\label{eq:87}
     \xymatrix@C+22pt{\cir\amalg \cir\rtwocell<5>^{\partial _0\nu
     _S}_{Y'}{_{\;\;X'}}& \emptyset ^1}  
  \end{equation}
in which $\emptyset ^1$ is the empty 1-dimensional manifold.  Apply~$\FG$ to
obtain
  \begin{equation}\label{eq:88}
     \xymatrix@C+22pt{\VGG\times \VGG\quad \quad\;
     \rtwocell<8>^{\FG(\partial _0\nu 
     _S)}_{\FG(Y')}{_{\qquad \;\FG(X')}}& \Vect_\CC} 
  \end{equation}
For $W\in \VGG$ evaluate~\eqref{eq:88} on~$(W,W)\in \VGG\times \VGG$ to
define ~$\FG$ on~$(X,S,W)$:
  \begin{equation}\label{eq:89}
     \FG(X')(W,W)\in \Hom\bigl(\FG(Y')(W,W),\FG(\partial _0\nu _S)(W,W) \bigr). 
  \end{equation}
Now $\partial _0\nu _S$ is a cylinder with the entire boundary incoming,
which is the ``evaluation morphism'' in the bordism category, hence
$\FG(\partial _0\nu _S)$~is the evaluation morphism
  \begin{equation}\label{eq:90}
     \begin{aligned} \VGG^{\textnormal{op}}\times \VGG&\longrightarrow \Vect
      \\ (W_1,W_2)&\longmapsto \Hom_{\VGG}(W_1,W_2)\end{aligned} 
  \end{equation}
We evaluate~$\FG(Y')$ using the classical gauge theory.  Restriction to the
boundary determines a map of groupoids 
  \begin{equation}\label{eq:91}
     \pi \:\bung{Y'}\longrightarrow \bung{\cir\amalg \cir}\approx G\gpd
     G\times G\gpd G, 
  \end{equation}
where $G$~acts on itself by conjugation.  For $W_1,W_2\in \VGG$ the value of
$\FG(Y')(W_1,W_2)$ is the vector space of global sections of $\pi
^*(W_1^*\boxtimes W\mstrut _2)\to \bung{Y'}$.

  \subsection{Topological boundary conditions; symmetry breaking}\label{subsec:3.4}

Let $G$~be a finite group and $H\subset G$ a subgroup.  Recall the
topological boundary theory $\BH\:1\to \FG$; see~\eqref{eq:11}.  It has
a classical description in which the boundary field is a reduction of a
$G$-bundle to an $H$-bundle.  If $Q\to Y$ is a principal $G$-bundle, then a
reduction to~$H$ is equivalently a section of $Q/H\to Y$.  The boundary
theory counts these sections: $\BH(Y)$~is the function on~$\bung Y$
whose value at~$Q$ is the number of sections.  (There are no automorphisms of
sections, so no weighting in the sum.)  We interpret the topological boundary
data~$\BH$ as \emph{symmetry breaking} from~$G$ to the subgroup~$H$.

Identify~$\FG(\cir)\simeq \Vect_G(G)$ as in~\eqref{eq:32}.   Evaluate
$\BH(\cir)\:\Vect\to\Vect_G(G)$ on~$\CC\in \Vect$ to obtain a
$G$-equivariant vector bundle $V_H\to G$.  We compute it by summing over the
fibers of~\eqref{eq:12}, which for~$M=\cir$ is the map of groupoids
  \begin{equation}\label{eq:26}
     \pi \:H\gpd H\longrightarrow G\gpd G. 
  \end{equation}
Thus $V_H=\pi _*\triv{\CC}$, where $\triv{\CC}\to H\gpd H$ is the trivial
rank one complex vector bundle.  The fiber of~$\pi $ over~$g\in G$ is the
groupoid~$F_g$ whose objects are pairs $(h,\tg)\in H\times G$ such
that~$\tg h\tg\inv =g$.  A morphism $k\:(h,\tg)\to(h',\tg')$ is given by~$k\in H$ such
that $khk\inv =h'$ and~$\tg k\inv =\tg'$.  In particular, if~$g\in H$ then 
  \begin{equation}\label{eq:28}
     \pi _0F_g\cong Z_g(G)/Z_g(H), 
  \end{equation}
where $Z_g(G)$~is the centralizer of~$g$ in~$G$ and $Z_g(H)$~the centralizer
of~$g$ in~$H$.  In general,
  \begin{equation}\label{eq:27}
     (V_H)_g=\Fun(\pi _0F_g). 
  \end{equation}

  \begin{example}[]\label{thm:6}
 If $G=A=\bmut$ then there are two subgroups.  If~$H=\bmut$ then
$V_{\bmut}=\triv{\CC}$ is the trivial bundle with stabilizers acting
trivially.  If~$H=1$ is the trivial group, then the fiber~$(V_{\bmut})_1$ at
$1\gpd A\subset A\gpd A$ is the 2-dimensional regular representation of~$A$,
and the fiber~$(V_{\bmut})_{-1}$ at ${-1}\gpd A$ is the zero vector space.
Using~\eqref{eq:35} in each case $V\to A\times A\dual$~has two rank one
fibers and two zero fibers.  The two cases are exchanged under
electromagnetic duality, which exchanges $A\longleftrightarrow A\dual$.
  \end{example}

  \begin{remark}[]\label{thm:8}
 A central extension $1\longrightarrow\TT\longrightarrow
\widetilde{H}\longrightarrow H\longrightarrow 1$ also gives a topological
boundary condition for~$\FG$, and in fact every indecomposable topological
boundary condition has this form; see~\cite{Os} and \cite[Corollary
7.12.20]{EGNO}.  The boundary theory is a \emph{weighted} counting of
reductions of a principal $G$-bundle to a principal $H$-bundle.  The weight
is derived from the cohomology class in $H^2(B\widetilde{H};\TT)$ of the
central extension.  We do not encounter any nontrivial central extensions in
the application to lattice models, nor do we see a mechanism whereby symmetry
breaking would lead to one: the weights~$\theta $ we use
(Definition~\ref{thm:18} below) are nonnegative, and there are no central
extensions of~$H$ with center~$\RR^{>0}$.  This reasoning leads to
Conjecture~\ref{thm:33}.
  \end{remark}

   \section{Lattice theories on the boundary}\label{sec:4}
% lastsubsec@  5

Now we introduce the boundary lattice theories.  There is a parameter, a
function on the group used to weight the interactions on each edge of the
lattice, and we introduce the appropriate function space
in~\S\ref{subsec:4.1}.  In~\S\ref{subsec:4.5} we defined precisely our notion
of lattices in 1-~and 2-manifolds.  The boundary theories are defined as
non-extended theories in~\S\ref{subsec:4.2}, and the order and disorder
operators introduced in~\S\ref{subsec:4.3}.  We conclude by unifying
electromagnetic and Kramers-Wannier dualities in~\S\ref{subsec:4.4}.

  \subsection{Weighting functions}\label{subsec:4.1}

  \begin{definition}[]\label{thm:11}
 Let $A$~be a finite abelian group.  A function $\theta \:A\to\RR$ is
\emph{admissible} if (i)~$\theta (a)\ge0$ for all~$a\in A$; (ii)~$\theta
(-a)=\theta (a)$ for all~$a\in A$; and (iii)~$\theta \dual(a\dual)\ge0$ for
all~$a\dual\in A\dual$, where $\theta \dual=\sF\theta \:A\dual\to\RR$ is the
Fourier transform of~$\theta $. 
  \end{definition}

\noindent
 See~\eqref{eq:21} for the definition of the Fourier transform.  Observe that
$\theta \dual$~is admissible if and only if $\theta $~is.  It
follows\footnote{Proof: Write $\theta \dual=\phi ^2$ for a positive
function~$\phi $, and then for any~$a\in A$ we have by Cauchy-Schwarz
  \begin{equation}\label{eq:56}
     \theta (a) = \frac{1}{\sqrt{\#A}}\sum\limits_{a\dual\in
     A\dual}\left[\overline{\chi (a,a\dual)}\phi (a\dual)\right]\phi (a\dual) \le
     \frac{1}{\sqrt{\#A}}\sum\limits_{a\dual\in A\dual} \phi (a\dual)^2 = \theta
     (0). 
  \end{equation}
} from these conditions that $\theta $~achieves its maximal value at~$a=0$,
which means that it models a \emph{ferromagnetic interaction}.  The set of
admissible functions on~$A$ is a convex subset of the vector space of all
real-valued functions on~$A$.

  \begin{example}[$A=\bmu 5$]\label{thm:12}
 A nonnegative even real-valued function~$\theta $ on~$A=\bmu 5$ is
determined by 3~nonnegative numbers 
  \begin{equation}\label{eq:51}
     \begin{aligned} a&=\theta (1) \\ b&=\theta (\lambda ) = \theta (\lambda
     ^4) \\ c&=\theta 
      (\lambda ^2)=\theta (\lambda ^3),\end{aligned} 
  \end{equation}
where $\lambda =e^{2\pi i/5}$.
The Fourier transform~$\theta \dual$ on~$A\dual=\zmod5$ takes values 
  \begin{equation}\label{eq:52}
     \begin{aligned} \theta \dual(0) &= (a + 2b + 2c)/\sqrt5 \\ \theta
      \dual(1 )=\theta \dual(4) &= (a + pb + qc)/\sqrt5 \\
      \theta \dual(2)=\theta \dual(3) &= (a + qb +
      pc)/\sqrt5,\end{aligned} 
  \end{equation}
where $p=2\cos(2\pi /5)$ and $q=2\cos(4\pi /5)$.  If $a=0$, then positivity
of~$\theta \dual$ forces $b=c=0$ as well, so we may assume~$a\neq 0$ and
multiplicatively normalize~$a=1$.  The region in the $(b,c)$-plane in which
the six numbers in~\eqref{eq:51} and~\eqref{eq:52} are nonnegative is the
convex hull of its 4~extreme points
  \begin{equation}\label{eq:53}
     (b,c) = (0,0),\;(1,1),\;(\frac p2,\frac q2),\;(\frac q2,\frac p2). 
  \end{equation}
For the first $\theta $~is the characteristic function of the trivial
subgroup $1\subseteq\bmu 5$, for the second $\theta $~is the characteristic
function of the full subgroup $\bmu 5\subseteq \bmu 5$; the Fourier transform
exchanges them.  The other two extreme points do not correspond to
subgroups of~$\bmu 5$. 
  \end{example}

  \begin{example}[$A=\bmu 4$]\label{thm:21}
 After dividing by positive multiplicative scaling, the space of
admissible~$\theta $ is a convex planar set with 4~extreme points: three are
characteristic functions of subgroups, and the fourth takes values
$a,a/2,0,a/2$ for some positive real number~$a$.
  \end{example}

If $G$~is a possibly nonabelian finite group, then there is a generalization
of the Fourier transform~\eqref{eq:21}.  Namely, if $\theta \:G\to\CC$ and
$\rho \:G\to\Aut(W)$ is a finite dimensional complex representation of~$G$,
then define 
  \begin{equation}\label{eq:47}
     \theta \dual(\rho )=\frac{1}{\sqrt{\#G}}\sum\limits_{g\in G}\theta
     (g)\overline{\rho (g)}\;\in \End(W). 
  \end{equation}
It suffices to evaluate~$\theta \dual$ on a representative set of irreducible
representations.

  \begin{definition}[]\label{thm:18}
  Let $G$~be a finite group.  A function $\theta \:G\to\RR$ is
\emph{admissible} if (i)~$\theta (g)\ge0$ for all~$g\in G$; (ii)~$\theta
(g\inv )=\theta (g)$ for all~$g\in G$; and (iii)~$\theta \dual(\rho )$ is a
nonnegative operator for each irreducible unitary representation~$\rho
\:G\to\Aut(W)$. 
  \end{definition}

\noindent
 Observe that the evenness condition~(ii) implies that $\theta \dual(\rho
)$~is self-adjoint.

  \subsection{Latticed 1- and 2-manifolds}\label{subsec:4.5}

The ``lattices'' we use in this paper are embedded in compact 1-~and
2-manifolds.  As a preliminary we define a model \emph{solid $n$-gon}
for~$n\in \ZZ^{\ge2}$.  If~$n\ge3$ then a solid $n$-gon is, say, the convex
hull of the $n^{\textnormal{th}}$~roots of unity in~$\CC$.  A solid 2-gon is,
say, the set
  \begin{equation}\label{eq:38}
     \left\{(x,y)\in \AA^2: \frac{x^2-1}{2}\le y\le\frac{1-x^2}{2}\right\}. 
  \end{equation}

  \begin{definition}[]\label{thm:13}
 \ 
 \begin{enumerate}[label=\textnormal{(\roman*)}]

 \item A \emph{latticed 1-manifold}~$(S,\nu )$ is a closed 1-manifold~$S$
equipped with a finite subset~$\nu \subset S$ which intersects each component
of~$S$ in a set of cardinality~$\ge2$.

 \item A \emph{latticed 2-manifold}~$(Y,\Lambda )$ is a compact
2-manifold~$Y$ equipped with a smoothly embedded finite graph~$\Lambda
\subset Y$ which intersects each component of~$Y$ nontrivially.  The closure
of each component of~$Y\setminus \Lambda $ is a smoothly embedded solid
$n$-gon with~$n\ge 2$.  The intersection $\Lambda \cap\partial Y$ must
determine a latticed 1-manifold.  Furthermore, if $e$~is an edge of~$\Lambda
$, then either (a)~$e\cap\partial Y=\emptyset $, (b)~$e\cap\partial Y$~is a
single boundary vertex of~$e$, or (c)~$e\subset \partial Y$ .

 \end{enumerate}
  \end{definition}

\noindent
 A component of~$Y\setminus \Lambda $ is called a \emph{face}.  We use
`$\Vertices(\Lambda )$', `$\Edges(\Lambda )$' to denote the sets of vertices and
edges of the lattice~$\Lambda $.  It is understood that an embedding of a
solid $n$-gon takes vertices to vertices and edges to edges.  There is no
choice of embedding in the data of a latticed 2-manifold, only a condition
that an embedding exists.  Our definition rules out loops in~$\Lambda $ but
allows faces which share more than a single edge.  Up to cyclic symmetry a
connected latticed 1-manifold is homeomorphic to a connected finite graph
whose vertices have valence two: a polygon.  We use the notation~$(S,\Pi )$
for a latticed 1-manifold; $\Pi \subset S$~is an embedded graph, each
component of which is an embedding of an $n$-gon, $n\in \ZZ^{\ge2}$.

  \begin{definition}[]\label{thm:14}
 Let~$(Y,\Lambda )$ be a latticed 2-manifold.  A \emph{dual latticed
2-manifold}~$(Y,\Lambda \dual)$ is characterized by bijections $\pi
_0(Y\setminus \Lambda )\to\Vertices(\Lambda \dual)$ and $\Edges(\Lambda
)\to\Edges(\Lambda \dual)$ such that (i)~the vertex of~$\Lambda \dual$
corresponding to a face~$f$ is contained in the interior of~$f$ and
(ii)~corresponding edges~$e\subset \Lambda $ and~$e\dual\subset \Lambda \dual$
intersect transversely in a single point.
  \end{definition}

  \begin{proposition}[]\label{thm:15}
 Let~$(Y,\Lambda )$ be a latticed 2-manifold.  Then a dual lattice $\Lambda
\dual$ exists.
  \end{proposition}

\noindent
 In fact, the space of dual lattices~$\Lambda \dual$ is contractible, though
we do not give a formal proof here.

  \begin{proof}
 Choose a point~$p_e$ in the interior of each edge of~$\Lambda $ and fix an
embedding~$\psi _f$ of a solid $n$-gon onto each closed face~$f$
of~$(Y,\Lambda )$.  The vertices of~$\Lambda \dual$ are the images~$\psi _f(c)$
of the centers of the model solid $n$-gons.  The edges are constructed from the
image under~$\psi _f$ of line segments joining~$c$ to~$\psi _f\inv (p_e)$ for
each edge~$e$ in the boundary of~$f$; then straighten the resulting angle at
each~$p_e$.
  \end{proof}

  \subsection{Lattice models as boundary theories}\label{subsec:4.2}

Let $G$~be a finite group and $\FG$~the 3-dimensional finite gauge theory
discussed in~\S\ref{sec:3}.  Fix an admissible function $\theta \:G\to\RR$
(Definition~\ref{thm:18}).  We now define a 2-dimensional boundary
theory~$\rI_{(G,\theta )}$ for~$\FG$, but on a bordism category of 1-~and
2-manifolds; we do not ``extend down to points''.  (We will extend down to
points in~\S\ref{sec:8}.) More precisely, the boundary theory
$\rI=\rI_{(G,\theta )}$ is defined on the bordism category of latticed
1-manifolds and latticed 2-dimensional bordisms between them. 
 
Let $(M,\Lambda )$~be a latticed 1-~or 2-manifold and $\pi \:Q\to M$ a
principal $G$-bundle.  The boundary theory is a finite $\sigma $-model whose
fields comprise the finite set 
  \begin{equation}\label{eq:39}
     \conf Q\Lambda =\left\{ \textnormal{sections of }Q \Res{\pi \inv
     \bigl(\Vertices(\Lambda ) \bigr)}\to\Vertices(\Lambda ) \right\} . 
  \end{equation}
Suppose $(Y,\Lambda )\:(S_0,\Pi _0)\to(S_1,\Pi _1)$ is a 2-dimensional
latticed bordism between latticed 1-manifolds, and $\pi \:Q\to Y$ a principal
$G$-bundle, then restriction to the boundaries defines a correspondence
diagram of finite sets 
  \begin{equation}\label{eq:40}
     \begin{gathered} \xymatrix{&\conf Q\Lambda \ar[dl]_{r_0}
     \ar[dr]^{r_1} \\ \conf{R_0}{\Pi _0}&& \conf{R_1}{\Pi _1}}
     \end{gathered} 
  \end{equation}
in which $R_0\to S_0$ and $R_1\to S_1$ are the restrictions of $Q\to Y$ to
the incoming and outgoing boundaries, respectively.  Define a function 
  \begin{equation}\label{eq:41}
     K\:\conf Q\Lambda \longrightarrow \RR, 
  \end{equation}
a sort of integral kernel, as follows.  If~$e\subset \Lambda $ is an edge,
then parallel transport along~$e$ identifies the fibers of $Q\to Y$ over the
boundary points of~$e$.  The values of a section $s\in \conf Q\Lambda $
over these two vertices are related by an element $g(s;e)\in G$, defined up
to inversion depending on the order of the boundary points (orientation
of~$e$).  Since the function $\theta \:G\to\RR$ is even, the number
  \begin{equation}\label{eq:42}
     K(s)=\prod\limits_{e\in \Edges(\Lambda )\setminus \Edges(\Lambda \cap
     S_1)}\theta \bigl(g(s;e) \bigr)  
  \end{equation}
is independent of edge orientations.  In~\eqref{eq:42} we multiply the
weighting factor over incoming and interior edges of~$\Lambda $, but not over
outgoing edges.  Under composition of bordisms the correspondence
diagrams~\eqref{eq:40} compose by fiber product and the integral
kernels~\eqref{eq:42} multiply.

  \begin{definition}[]\label{thm:16}
 \

 \begin{enumerate}[label=\textnormal{(\roman*)}]

 \item For a latticed 1-manifold~$(S,\Pi )$ and principal $G$-bundle $R\to S$
set 
  \begin{equation}\label{eq:43}
     \rI(S,\Pi )[R]=\Fun\bigl(\conf R\Pi  \bigr). 
  \end{equation}

 \item For a latticed bordism $(Y,\Lambda )\:(S_0,\Pi _0)\to(S_1,\Pi _1)$ set 
  \begin{equation}\label{eq:44}
     \rI(Y,\Lambda )[Q] = (r_1)_*\circ K\circ (r_0)^*\:\rI(S_0,\Pi
     _0)[R_0]\longrightarrow \rI(S_1,\Pi _1)[R_1];
  \end{equation}
see~\eqref{eq:40} and~\eqref{eq:41} for notation.

 \end{enumerate}
  \end{definition}

\noindent
 In~\eqref{eq:44} `$K$'~is multiplication by the function~$K$.  It is
straightforward to extend~\eqref{eq:43} to a functor $\bung S\to\Vect$, i.e.,
to an equivariant vector bundle over~$\bung S$, or equivalently---according
to~\eqref{eq:45}---an object in the category~$\FG(S)$.  Then
\eqref{eq:44}~defines an equivariant map between the equivariant vector
bundles $\rI(S_0,\Pi _0)$ and~$\rI(S_1,\Pi _1)$.  This is precisely what a
boundary theory $\rI\:1\to \FG$ must do.

If $(Y,\Lambda )$~is a \emph{closed} latticed surface, then
\eqref{eq:44}~reduces to the function on~$\bung Y$ whose value at~$Q$ is the
partition function 
  \begin{equation}\label{eq:57}
     \rI(Y,\Lambda )[Q] = \sum\limits_{s\in \conf Q\Lambda }\;
     \prod\limits_{e\in \Edges(\Lambda )}\theta \bigl(g(s;e) \bigr) 
  \end{equation}
of the Ising model. 

  \begin{remark}[]\label{thm:17}
 Let $(S,\Pi )$ be a latticed 1-manifold.  For each principal $G$-bundle
$R\to S$ the boundary theory produces a vector space~$\rI(S,\Pi )[R]$.  If
$R\to S$ is the trivial $G$-bundle $S\times G\to S$, then this is the usual
state space in the quantum mechanical interpretation of the Ising model.  For
nontrivial $R\to S$ it is the state space of a ``twisted sector''.  Now form
the bordism $\zo\times (S,\Pi )\:(S,\Pi )\to(S,\Pi )$ in which $Y=\zo\times
S$ and $\Lambda =\zo\times \Pi $.  For $R\to S$ set $P=\zo\times R\to Y$.
Then the linear endomorphism~$\rI(Y,\Lambda )[P]$ of~$\rI(S,\Pi )[R]$ is the
``transfer matrix'' in the sector defined by $R\to Y$.  It may be interpreted
as~$\exp(-H_R)$---Wick rotated discrete time evolution over a single unit of
time---where $H_R$~is the Hamiltonian operator in that sector. 
  \end{remark}

  \begin{remark}[]\label{thm:22}
 Suppose $\theta $~is a multiple of the characteristic function of a
subgroup~$H\subset G$.  Then $K(s)$ in~\eqref{eq:42} vanishes unless all
parallel transports~$g(s;e)$ lie in~$H$.  If so, then $s$~determines a
section of the fiber bundle $Q/H\to \Lambda $, and it extends to a section
over~$Y$, so a reduction of $Q\to Y$ to the subgroup~$H$.  In this case the
boundary theory~$\rI_{(G,\theta )}$ is\footnote{up to a multiplicative factor
$C^v$, where $v$~is the number of vertices in a lattice.  A similar factor
occurs in Proposition~\ref{thm:62}.  To correct for these factors the
integral kernel~\eqref{eq:42} should include products over vertices and faces
of constants~$C$ and~$C'$, so the boundary lattice theory is parametrized
by~$(C,\theta ,C')$; see Definition~\ref{isingpart}.} the topological
boundary theory~\eqref{eq:11}.
  \end{remark}

  \subsection{Order and disorder operators}\label{subsec:4.3}

We continue with a finite group~$G$ and an admissible function $\theta \:G\to
\RR$.   
 
Suppose $X$~is a compact 3-manifold with latticed boundary~$(\bX,\Lambda )$.
As in~\eqref{eq:4} we can evaluate the pair $(\FG,\rI_{(G,\theta )})=(\FG,\rI)$
on~$X$ to compute a number: $\rI(\bX,\Lambda )$~is a function on~$\bung{\bX}$
and $\FG$~evaluated on~ $X$ as a bordism with $\bX$~incoming ~is a linear
functional on $\Fun\bigl(\bung{\bX} \bigr)$.  The explicit formula
combines~\eqref{eq:9} and~\eqref{eq:44}: 
  \begin{equation}\label{eq:48}
     (\FG,\rI)(X,\Lambda ) = \sum\limits_{[P]\in \pi _0\bung X}\frac{1}{\#\Aut
     P}\sum\limits_{s\in \conf{(\partial P)}\Lambda } \;\prod\limits_{e\in
     \Edges(\Lambda )}\theta \bigl(g(s;e) \bigr) . 
  \end{equation}

  \begin{figure}[ht] 
  \centering
  \includegraphics[scale=1]{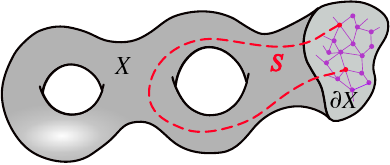}\qquad \;
  \includegraphics[scale=1]{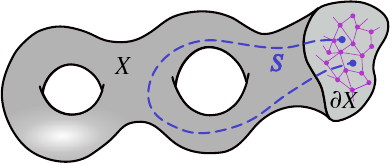}
  \caption{Wilson loop/order operator (left); 't Hooft loop/disorder operator
  (right)}\label{fig:2}
  \end{figure}

Recall from~\S\ref{subsec:3.3} that $\FG$~admits two distinguished kinds of
loop operators associated to an embedded oriented\footnote{See
Remark~\ref{thm:9} to drop the orientation.}  circle: Wilson and 't Hooft
operators.  We also discussed operators on neatly embedded closed
intervals~$S$, and these carry over to operators in the theory~$(\FG,\rI)$ on a
3-manifold with latticed boundary.  For Wilson operators we require the
boundary of~$S$ to lie in $\Vertices(\Lambda )\subset \Lambda \subset \bX$.
Fix a character $\chi \:G\to\TT$.  Then for a principal $G$-bundle $P\to X$
and a section $s\in \conf{(\partial P)}\Lambda $ over $\Vertices(\Lambda
)$ we define
  \begin{equation}\label{eq:49}
     h_{S,\chi }(P,s) = \chi \bigl(g(s;S) \bigr), 
  \end{equation}
where as in~\eqref{eq:42} the group element~$g(s,S)\in G$ sends~$s(\partial
_-S)$ to~$s(\partial _+S)$, after parallel transport along~$S$.  The
partition function with this ``Wilson path operator'' inserted is 
  \begin{equation}\label{eq:50}
     (\FG,\rI)(X,\Lambda ) = \sum\limits_{[P]\in \pi _0\bung X}\frac{1}{\#\Aut
     P}\sum\limits_{s\in \conf{(\partial P)}\Lambda } h_{S,\chi }(P,s)
     \prod\limits_{e\in \Edges(\Lambda )}\theta \bigl(g(s;e) \bigr) . 
  \end{equation}

  \begin{remark}[]\label{thm:19}
 In the language of spin systems $h_{S,\chi }(P,s)$~is called an \emph{order
operator}.  It is usually defined for a single vertex of~$\Lambda $ rather than
a pair of vertices connected by a path (which may be contained in~$\bX$ or
can wander into the interior of~$X$).  Indeed, if we restrict to the
untwisted sector in which $P\to X$ is the trivial $G$-bundle then we can
define the order operator at a single vertex, but for general $P\to X$ we
need the path~$S$.  We discuss order operators without paths below.
  \end{remark}

We turn now to the 't Hooft operator, for which we require the boundary
of~$S$ to lie in~$\bX\setminus \Lambda $.  Fix a conjugacy class~$\gamma
\subset G$.  As in~\eqref{eq:31} we let $\Bun_G\bigl(X;(S,\gamma ) \bigr)$ be
the groupoid whose objects are principal $G$-bundles $P\to X\setminus S$ with
holonomy~$\gamma $ about~$S$.  A bundle $P\to X\setminus S$ is defined
on~$\Lambda \subset \bX$, hence $\conf P\Lambda $ is still defined.  The
partition function with this ``'t Hooft path operator'' inserted is given
by~\eqref{eq:48} with~$\bung X$ replaced by $\Bun_G\bigl(X;(S,\gamma )
\bigr)$.  In the language of spin systems this is a \emph{disorder operator}.

There are more general point operators in the boundary theory which do not
require the point to lie at the end of a 1-manifold~$S$.  First, we consider
general defects in the topological gauge theory~$\FG$ on a closed
2-manifold~$Y$.  Suppose $y_1,\dots ,y_N\in Y$ is a finite set of points.
Let $Y'$ denote $Y$~with disjoint open disks~$D_{y_\alpha }$ about
the~$y_\alpha $ removed, and read~$Y'$ as a bordism with $\partial
Y'$~incoming.  Fix diffeomorphisms $\cir\xrightarrow{\;\approx \;}\partial
D_{y_\alpha }$.  Then, as in~\eqref{eq:91}, there is a fiber bundle of
groupoids
  \begin{equation}\label{eq:132}
     \pi \:\bung{Y'}\longrightarrow (G\gpd G)^{\times N}. 
  \end{equation}
Fix vector bundles $W_\alpha \to G\gpd G$, so $G$-equivariant vector bundles
$W_\alpha \to G$, and use them as inputs into~$\FG(Y')$.  Then
$\FG(Y')(W_1,\dots ,W_N)$ is the vector space of sections of 
  \begin{equation}\label{eq:133}
     \pi ^*(W_1\boxtimes\cdots\boxtimes W_N)\longrightarrow \bung{Y'}. 
  \end{equation}
Observe that because of nontrivial automorphisms of $G$-bundles, this may be
the zero vector space. 
 
We now describe point operators in the boundary theory $\rI=\rI_{G,\theta }$
on a latticed surface~$(Y,\Lambda )$.  There are two types of such point
defects: order and disorder operators.  For an order operator at~$y_i$ we
have $y_i\in \Vertices(\Lambda )$; the vector bundle $W_i\to G\gpd G$ is
supported at~$e\gpd G$, so is a representation of~$G$; and we fix a
vector~$\xi _i\in W_i$.  (Identify~$W_i$ with its fiber at~$e$.)  For a
disorder operator at~$y_j$ we have $y_j\in Y\setminus \Lambda $; the vector
bundle $W_j\to G\gpd G$ has trivial action of centralizers, so is a vector
space~$W_{j,g/G}$ for each conjugacy class~$g/G$; and we fix a vector $\xi
_{j,g/G}\in W_{j,g/G}$.  An irreducible $W_j\to G\gpd G$ is supported on a
single conjugacy class; the vector~$\xi $ vanishes at other conjugacy
classes.  Suppose there is order operator data at vertices~$v_1,\dots ,v_L$
and disorder operator data at faces $f_{1},\dots ,f_M$.  Then the
generalization of~\eqref{eq:57} with these point defects is
  \begin{multline}\label{eq:134}
     \quad\rI(Y,\Lambda )(W_1,\dots ,W_{L+M})[Q] \\= \sum\limits_{s\in \conf
     Q\Lambda }\left\{
     \bigotimes\limits_{i=1}^L \bigl(s(v_i)\times \xi _i \bigr) \;\otimes
     \bigotimes\limits_{j=K+1}^M \xi _{j,\hol_j(Q)}\right\}
     \prod\limits_{e\in \Edges(\Lambda )}\!\!\theta \bigl(g(s;e) \bigr) .
  \end{multline}
Here $Q\to Y'$ is a $G$-bundle; $\hol_j(Q)$ is the holonomy about~$f_j$, a
conjugacy class in~$G$; and $s(v_i)\times \xi _i\in Q_{v_i}\times \mstrut _G
W_i$ is a vector in the vector space defined by mixing the $G$-torsor
$Q_{v_i}$ with the $G$-representation~$W_i$.

  \subsection{Kramers-Wannier duality as electromagnetic duality}\label{subsec:4.4}

We investigate the effect of electromagnetic duality in finite abelian gauge
theory~(\S\ref{subsec:3.2}) on the Ising boundary theories.  The following is
stated earlier as Theorem~\ref{maintheorem}.

  \begin{theorem}[]\label{thm:25}
 Electromagnetic duality for 3-dimensional finite abelian gauge theory
extends to the Ising boundary conditions of lattice theories, whereon it
becomes Kramers-Wannier duality. It interchanges the action~$\theta $ with
its Fourier transform~$\theta \dual$.  Order operators of the Ising model are
boundary points of Wilson loops, disorder operators boundary points of 't
Hooft loops, and they are interchanged under duality. 
  \end{theorem}

\noindent
 The remainder of this section is devoted to an explicit proof on a closed
latticed surface.  The isomorphism of full boundary lattice theories is
proved in a more general context in~\S\ref{subsec:8.1}.
 
As a preliminary we state without proof properties of the finite Fourier
transform~\eqref{eq:21}.

  \begin{lemma}[]\label{thm:27}
 Let $\phi \:\pA\to\pB$ be a homomorphism of finite abelian groups, $\phi
\dual\:\pB\dual\to\pA\dual$ the Pontrjagin dual homomorphism, and
$\FA\:\Fun(\pA)\to\Fun(\pA\dual)$ the Fourier transform.

 \begin{enumerate}[label=\textnormal{(\roman*)}]

 \item If $f\in \Fun(\pA)$, then 
  \begin{equation}\label{eq:58}
     \FB(\phi _*f) = \sqrt{\frac{\#\pA}{\#\pB}}\; (\phi
     \dual)^*(\FA f). 
  \end{equation}

 \item If $g\in \Fun(\pB)$, then
  \begin{equation}\label{eq:59}
     \FA(\phi ^*g) = \sqrt{\frac{\#\pA}{\#\pB}}\; (\phi
     \dual)_*(\FB g). 
  \end{equation}

 \item Let 
  \begin{equation}\label{eq:74}
     1\longrightarrow \TT\xrightarrow{\;\;\lambda \;\;}
     \tpL\longrightarrow \pL\longrightarrow 0 
  \end{equation}
be a central extension and
  \begin{equation}\label{eq:75}
     0\longrightarrow \pL\dual\longrightarrow
     \tpL\dual\xrightarrow{\;\;\lambda \dual\;\;}\ZZ\longrightarrow 0 
  \end{equation}
its Pontrjagin dual.  Then the Fourier transform
$\Fun(\tpL)\to\Fun(\tpL\dual)$ restricts to an isomorphism of the vector
space of sections of the line bundle over~$L$ associated to~\eqref{eq:74}
with functions on the $\pL\dual$-torsor $(\lambda \dual)\inv (1)\subset
\tpL\dual$.
 
 \end{enumerate} 
  \end{lemma}

Suppose now the homomorphisms 
  \begin{equation}\label{eq:145}
     \begin{gathered} \xymatrix{M'\;\ar@{^{(}->}[r]^i & M\ar@{->>}[r]^(.4)\pi
     &M''} 
     \end{gathered} 
  \end{equation}
of finite abelian groups satisfy~$\pi \circ i=0$.  Let $K=\Ker\pi $ and
$H=\Ker\pi /\Image i$.  The Pontrjagin dual to~\eqref{eq:145} is 
  \begin{equation}\label{eq:146}
     \begin{gathered} \xymatrix{{M'}\dual & M\dual \ar@{->>}[l]_(.4){i\dual} &
     {M''}\dual \ar@{_{(}->}[l]_(.45){\pi \dual} } \end{gathered} 
  \end{equation}
Set $\tK=\Ker i\dual$ and $\tH=\Ker i \dual/\Image \pi \dual$.  ($\tH$~and
$H$~are in Pontrjagin duality.)  Consider the diagrams
  \begin{equation}\label{eq:147}
     \begin{gathered} \xymatrix{K\ar@{^{(}->}[r]^j \ar@{->>}[d]_(.43)p &
     M\\H}\qquad \qquad \qquad \xymatrix{\tK\ar@{^{(}->}[r]^{\tj}
     \ar@{->>}[d]_(.4){\tpp} & M\dual\\\tH} \end{gathered} 
  \end{equation}
and suppose $\Theta \in \Fun(M)$. 

  \begin{lemma}[]\label{thm:61}
 We have the following equality of Fourier transforms:
  \begin{equation}\label{eq:148}
     \sF_H(p_*\,j^*\Theta ) = \frac{\#K}{\sqrt{\#M\cdot \#H}}\;
     \tpp_*\,\tj^*\,\sF_M(\Theta ). 
  \end{equation}
  \end{lemma}

  \begin{proof}
 Apply Lemma~\ref{thm:27}(i),(ii) and use the pullback square 
  \begin{equation}\label{eq:149}
     \begin{gathered} \xymatrix{\tK\ar[r]^{\tj} \ar[d]_(.45){\tpp} &
     M\dual\ar[d]^(.45){j\dual} \\ \tH\ar[r]^(.37){p\dual} & M\dual/\Image\pi
     \dual} 
     \end{gathered} 
  \end{equation}
  \vskip-2.8em
  \end{proof}

Fix $\eta \in M''$ and $\omega \in {M'}\dual$.  Then the $K$-torsor $K_\eta
=\pi \inv (\eta )$ and $\tK$-torsor $\tK_\omega =(i\dual)\inv (\omega )$ fit
into the diagrams 
  \begin{equation}\label{eq:150}
     \begin{gathered} \xymatrix{K_\eta \ar@{^{(}->}[r]^{j_\eta }
     \ar@{->>}[d]_(.43)p & 
     M\\H_\eta }\qquad \qquad \qquad \xymatrix{\tK_\omega
     \ar@{^{(}->}[r]^{\tj_\omega } 
     \ar@{->>}[d]_(.4){\tpp} & M\dual\\\tH_\omega } \end{gathered} 
  \end{equation}
for the $H$-torsor $H_\eta =K_\eta /\Image i$ and $\tH$-torsor $\tH_\omega
=\tK_\omega /\Image\pi \dual$.  Use the character~$\omega \:M'\to\CC$ to
define twisted descent $p^\omega _*\:\Fun(K_\eta )\to \Fun_\omega (H_\eta )$
from $\Fun(K_\eta )$y to the vector space
  \begin{equation}\label{eq:151}
     p^*\!\Fun_\omega (H_\eta ) = \left\{ g\:K_\eta \to\CC : g\bigl(k+i(m') \bigr)
     = \omega (m')\inv g(k) \textnormal{ for all $m'\in M'$}\right\} 
  \end{equation}
of sections of the line bundle over~$H_\eta $ determined by~$\omega $: the
value of~$p_*^\omega $ on $f\in \Fun(K_\eta )$ is determined by the formula
  \begin{equation}\label{eq:154}
     p^*p^\omega _*f(k) = \sum\limits_{m'\in M'} \omega
     (m')f\bigl(k+i(m')\bigr),\qquad k\in K_\eta . 
  \end{equation}
There is a similar map~$\tpp_*^\eta $ which interchanges the roles of~$\eta $
and~$\omega $.  Lemma~\ref{thm:27}(iii) tells that the Fourier
transform~$\sF_M$ maps the subspace~$(j_\eta )_*p^*\Fun_\omega (H_\eta )$
of~$\Fun(M)$ to the corresponding subspace $(\tj_\omega )_*\tpp^*\Fun_\eta
(\tH_\omega )$ of~$\Fun(M\dual)$.  There is no Fourier transform defined
\emph{a priori} on~$\Fun_\omega (H_\eta )$, so we define it in terms
of~$\sF_M$, inserting the appropriate factors in Lemma~\ref{thm:61}, and
conclude 
  \begin{equation}\label{eq:152}
     \sF_{H_\eta }(p_*^\omega \,j^*_\eta \Theta ) = \frac{\#K}{\sqrt{\#M\cdot \#H}}\;
     \tpp_*^\eta \,\tj^*_\omega \,\sF_M(\Theta ). 
  \end{equation}

 Let $(Y,\Lambda )$ be a closed latticed surface.  Let $A$~be a finite
abelian group.  As in~\eqref{eq:134} we fix data for order and disorder
operators, which we take to be \emph{irreducible}:\footnote{Assume the $v_i$
are distinct and the~$f_j$ lie in distinct faces.}
  \begin{equation}\label{eq:136}
     \begin{aligned} v_1,\dots ,v_L&\in \Vertices(\Lambda ),\qquad &&\omega
      _1,\dots ,\omega _L\in A\dual; \\ f_1,\dots ,f_M&\in Y\setminus \Lambda
      ,\qquad &&\eta _1,\dots ,\eta _M\in A.\end{aligned} 
  \end{equation}
Let $\Lambda \dual$ be a dual lattice (Definition~\ref{thm:14}), chosen so
that $f_1,\dots ,f_M$ are vertices of~$\Lambda \dual$.  The cochain complexes 
  \begin{equation}\label{eq:137}
     \begin{gathered} \xymatrix@R-1.3em{C^0(\Lambda ;A)\ar[r]^{d^0}& C^1(\Lambda
     ;A)\ar[r]^{d^1}& C^2(\Lambda ;A) \\ C^2(\Lambda \dual;A\dual)&
     C^1(\Lambda \dual;A\dual)\ar[l]_(.45){\partial ^1} & C^0(\Lambda
     \dual;A\dual)\ar[l]_(.45){\partial ^0}} \end{gathered} 
  \end{equation}
are Pontrjagin dual.  The order data determines a character $\omega
\:C^0(\Lambda ;A)\to \CC$ and the disorder data an element $\eta \in
C^2(\Lambda ;A)$.   
 
The stack~$\Bun_A(\Lambda )$ has a small model, the action groupoid of
$C^0(\Lambda ;A)$ acting on~$C^1\GA$ by translation via~$d^0$.  (An element
of~$C^1\GA$ defines $Q\to \Lambda $ trivialized over~$\Vertices(\Lambda )$.)
The disorder operator acts by restriction to the subgroupoid ~$C^0\GA$ acting
on the $Z^1\GA$-torsor $(d^1)\inv (\delta _\eta )$,
where as usual $Z^1\GA=\ker d^1$.  Fix an admissible $\theta \in \Fun(A)$ and
let $\Theta \in \Fun\bigl(C^1\GA \bigr)$ be 
  \begin{equation}\label{eq:139}
     \Theta (c)= \prod\limits_{e\in \Edges(\Lambda )}\theta
     \bigl(c(e) \bigr),\qquad c\in C^1\GA. 
  \end{equation}
Then the Ising partition function~\eqref{eq:134} is 
  \begin{equation}\label{eq:140}
     \rI (Y,\Lambda )(z) = \sum\limits_{s\in C^0\GA}\omega (s)\cdot \Theta
     (z+d^0s),\qquad 
     z\in (d^1)\inv (\delta _\eta ).
  \end{equation}
We apply Lemma~\ref{thm:61} to compute its Fourier transform.  Take
$M=C^1\GA$, $M'=B^1\GA$, and\footnote{The disorder data~$\eta $ must be a
boundary.} $M''=B^2\GA$, where as usual $B^{\bullet }$~denotes the subgroup
of boundaries.  Then $K=Z^1\GA$ and $H=H^1(Y;A) \cong \pi _0\Bun_A(Y)$.
First, suppose $\eta =0$ and~$\omega =0$: no order or disorder operators.
Then \eqref{eq:140}~reduces to $\rI (Y,\Lambda )=\#H^0(Y;A)\cdot
p_*\,j^*\,\Theta $ and \eqref{eq:148} immediately implies the following,
which is part of Theorem~\ref{thm:25}.

  \begin{proposition}[]\label{thm:62}
 The Fourier transform $\Fun\bigl(\pi _0\Bun_A(Y)\bigr)\to \Fun\bigl(\pi
_0\Bun_{A\dual}(Y) \bigr)$ maps the Ising partition function $\rI _{(A,\theta
)}(Y,\Lambda )$ to the numerical factor
  \begin{equation}\label{eq:153}
     \frac{\# Z^1\GA}{\sqrt{\#C^1\GA\cdot \#H^1(Y;A)}} 
  \end{equation}
times the Ising partition function $\rI _{(A\dual,\theta \dual)}(Y,\Lambda
\dual)$.
  \end{proposition}

  \begin{remark}[]\label{thm:63}
 The numerical factor~\eqref{eq:153} equals 
  \begin{equation}\label{eq:155}
     \sqrt{\frac{\#C^0\GA}{\#C^2\GA}}. 
  \end{equation}
Hence if we divide $\rI_{(A,\theta )}(Y,\Lambda )$ by $\sqrt{\#C^0(\Lambda
;A)}$, which is a factor of~$(\#A)^{-1/2}$ for each vertex, then we achieve
precise agreement under Fourier transform.  (The extra factor is a special
case of a generalized Ising action; see Definition~\ref{isingpart}.)
  \end{remark}

If there are order and disorder operators, described by~$\eta $ and~$\omega
$, then the Ising partition function~\eqref{eq:140} is $\rI (Y,\Lambda
)=\#H^0(Y;A)\cdot p_*^\omega \,j^*_\eta \,\Theta $; see ~\eqref{eq:150}
and~\eqref{eq:154} for notation.  Therefore, \eqref{eq:152}~ tells the
equality with the Ising partition function of the dual model, with order
and disorder operators exchanged.

  \begin{remark}[]\label{thm:64}
 We do not prove that electromagnetic duality induces the Fourier transform
defined in \eqref{eq:152}.  Rather, we complete the proof of
Theorem~\ref{thm:25} in~\S\ref{subsec:8.1} in a more general setting. 
  \end{remark}

   \section{Low energy effective topological field theories}\label{sec:5}
% lastsubsec@000

In this section we consider qualitative aspects of the lattice theories.  Our
discussion is heuristic and conjectural. 
 
To begin, we recall some general principles about quantum systems.  First,
the low energy behavior of a quantum system is thought to be
well-approximated by a scale-independent relativistic quantum field theory.
In particular, this is applied to lattice systems, in which case the emergent
relativistic invariance is a strong assumption.  Furthermore, there is a
notion of a \emph{gapped} quantum system: the Hamiltonian has a spectral gap
above the lowest energy.  For lattice systems one assumes that this energy
gap is bounded below independent of the lattice.  Then for a gapped theory,
in many cases, the low energy effective field theory is thought to be
topological.\footnote{In general, it should be a topological theory tensored
with an invertible theory; see~\cite[\S5.4]{FH}.  For the gapped lattice
models in this paper we assume that the low energy effective theory is
topological.}  If we consider a moduli stack~$\sM$ of quantum theories with
fixed discrete parameters, then there is a locus~$\Delta \subset \sM$ of
\emph{phase transitions}.  Points in~$\sM\setminus \Delta $ may represent
gapped or ungapped systems; the points in~$\Delta $ labeling
\emph{first-order} phase transitions may also represent gapped systems
whereas those points in~$\Delta $ labeling \emph{higher-order} phase
transitions represent gapless systems.  Path components of~$(\sM\setminus
\Delta )_{\textnormal{gapped}}$ are called \emph{gapped phases}.
Furthermore, the low energy effective field theory associated to a point
of~$(\sM\setminus \Delta )_{\textnormal{gapped}}$ is thought to be a complete
invariant of the gapped phase.  In this section we use the full force of the
global symmetry---the presentation of lattice theories as boundary theories
for fully extended finite gauge theory---to deduce constraints on the low
energy field theory.

  \begin{remark}[]\label{thm:30}
 There is also dynamics, the \emph{renormalization group flow}
on~$(\sM_G\setminus \Delta _G)_{\textnormal{gapped}}$.  Its limit points
should be the possible low energy theories.
  \end{remark}

  \begin{remark}[]\label{thm:31}
 For a fixed finite group~$G$ we take~$\sM_G$ as the space of admissible
functions on~$G$ divided by positive multiplicative rescaling;
see~\S\ref{subsec:4.1}.  Since our lattice theories are a limited class of
nearest neighbor interactions (see~\eqref{eq:42}), we do not expect a naive
renormalization group flow on~$\sM_G$.  To construct a flow we would have to
project from a chimerical moduli space of all lattice theories with
symmetry~$G$ back onto this space.
  \end{remark}

Fix a finite group~$G$ and an admissible function~$\theta $.  Recall the
definition of the Hamiltonian in a lattice model from Remark~\ref{thm:17}.
There is a Hamiltonian associated to every latticed 1-manifold~$(S,\Pi )$
equipped with a principal $G$-bundle $R\to S$.  Choose $S=S^1$ the standard
circle and so $\Pi $~a polygon.  The construction (Definition~\ref{thm:16})
of the lattice model applied to the cylinder $[0,1]\times S^1$ with embedded
``prism'' $[0,1]\times \Pi $ defines an endomorphism of the state space of
functions on configurations on~$\Pi $.  This is the Wick-rotated propagation
through a unit of time, so is $\exp(-H_{(\Pi ,R)})$ for the Hamiltonian
operator\footnote{The Hamiltonian is undefined (infinite) on the kernel of
the evolution operator.}~$H_{(\Pi ,R)}$.  For the trivial $G$-bundle
$\cir\times G\to\cir$ the operator~$H_{\Pi ,R}$ is the usual Hamiltonian, but
there are twisted sectors and so ``twisted Hamiltonians'' for nontrivial
$R\to\cir$.  Full locality permits us to relate different sectors, at least
in principle, by modifying $R\to\cir$ over small intervals in~$\cir$.  Recall
that $\theta \:G\to\RR$ achieves its maximum at the identity element~$e\in
G$.  Assume that $e$~is the unique maximum point of~$\theta $.  It follows
that in the untwisted sector the constant $G$-valued function on vertices
of~$\Pi $ is a minimal energy configuration.  Furthermore, the minimal energy
configuration in a twisted sector has energy dependent on values of~$\theta $
on~$G\setminus \{e\}$, so has strictly larger energy than the minimal energy
in the untwisted sector.  That difference may or may not be bounded away from
zero as we vary~$\Pi $.

  \begin{remark}[]\label{thm:32}
 The strong assumption that the low energy effective field theory be
\emph{fully local} implies that `low energy' must also be fully local in the
following sense.  `Low'~is a global minimum---the minimal energy in the
untwisted sector---and so if a twisted sector has minimal energy uniformly
larger than the global minimum, its states are not present in the low energy
approximating field theory: the field theory has value the zero vector space
in that sector. 
  \end{remark}

Let~$\sM_G$ denote the space of admissible functions on~$G$, up to rescaling.
As above, we assume a locus~$\Delta _G\subset \sM_G$ of lattice systems at
which phase transitions occur, and we restrict to the subset of the
complement representing gapped theories.  Then associated to $\theta \in
(\sM_G\setminus \Delta _G)_{\textnormal{gapped}}$ we expect an effective low
energy \emph{topological} field theory.  Since the lattice
model~$\rI_{(G,\theta )}$ has a global symmetry group~$G$, encoded in strong
form by realizing~$\rI_{(G,\theta )}$ as the boundary theory of 3-dimensional
pure gauge theory~$\FG$, we expect the same for its low energy
approximation~$\rL_{(G,\theta )}$.  Now we are on solid mathematical ground:
$\rL_{(G,\theta )}$ ~is a \emph{topological} boundary theory for an extended
topological field theory.  In~\S\ref{subsec:3.4} we described how a
subgroup~$H\subset G$ gives rise to a topological boundary theory~$\BH$,
and in Remark~\ref{thm:8} we quoted a (straightforward) theorem
\cite[Corollary~7.12.20]{EGNO} in the theory of tensor categories to the
effect that \emph{every}\footnote{This result depends on choosing $\TensCat$
as the target 3-category of~$\FG$.  We have not investigated whether varying
this choice produces additional boundary theories.}  topological boundary
theory which is indecomposable is built in this way from a subgroup $H\subset
G$ and a central extension of~$H$.  We also argued in Remark~\ref{thm:8} that
we do not expect to see nontrivial central extensions.  Furthermore, the
topological theory corresponding to~$H$ (no central extension) is realized
as~$\rI_{(G,\theta _H)}$ for $\theta _H\in \sM_G$ the extreme point at which
$\theta _H$~is the characteristic function of~$H$; see Remark~\ref{thm:22}.
Given these facts, it seems reasonable to propose that neither nontrivial
central extensions nor decomposable boundary theories occur.

  \begin{conjecture}[]\label{thm:33}
 The low energy field theory~$\rL_{(G,\theta )}$, $\theta \in (\sM_G\setminus
\Delta _G)_{\textnormal{gapped}}$, is~$\BH$ for some subgroup~$H\subset G$. 
  \end{conjecture}

  \begin{remark}[]\label{thm:34}
 As explained in \S\ref{subsec:3.4} the boundary theory~$\BH$ implements
symmetry breaking from~$G$ to~$H$.  Therefore, Conjecture~\ref{thm:33}
asserts that gapped phases of Ising lattice models are distinguished by
symmetry breaking, which realizes Landau's vision\footnote{Landau proposed
classifying phases by patterns of symmetry breaking.  In many examples it is
now understood that symmetry breaking is not sufficient; see~\cite{We} for an
early articulation.} for this class of models.
  \end{remark}

  \begin{remark}[]
 Conjecture~\ref{thm:33} predicts that the renormalization group flow
(Remark~\ref{thm:30}) has limit points~$\theta _H\in \sM_G$ corresponding to
subgroups~$H\subset G$. 
  \end{remark}

In the remainder of this section we describe explicitly the special case of
the classical Ising model~$G=\bmut$.  An admissible function $\theta
\:\bmut\to\RR$ can be normalized so $\theta (1)=1$; then $\theta (-1)=a$ with
$0\le a\le1$.  Thus $\sM_{\bmut}\cong [0,1]$.  The extreme point~$a=0$
corresponds to the trivial subgroup~$1\subset \bmut$; the extreme point~$a=1$
to the full subgroup~$\bmut\subset \bmut$.  The usual parameter for the
classical Ising model is the inverse temperature~$\beta $, which is related
to~$a$ by $a=e^{-2\beta }$.  The locus~$\Delta _{\bmut}$ of phase transitions
consists of a single point $a_\Delta =\sqrt2-1$, or equivalently
$\sinh(2\beta )=1$.  The Fourier transform induces the involution
$a\leftrightarrow (1-a)/(1+a)$ on~$\sM_{\bmut}$, after uniquely identifying
the abelian groups $\bmut\dual\cong \bmut$; the fixed point locus is the
single point~$\Delta _{\bmut}=\{a_{\Delta }\}$.  The low temperature phase is
$0\le a < a_\Delta $.  The high temperature phase is $a_\Delta <a\le1$.  The
expectation is that $a=a_\Delta $ is a source for the renormalization group
flow, which limits to sinks at~$a=0$ and~$a=1$ in the low and high
temperature phases, respectively.  This phase diagram, depicted in
Figure~\ref{fig:1}, is standard in the physics literature,
e.g.,~\cite[Figure~3.7]{C}.

  \begin{figure}[ht]
  \centering
  \includegraphics[scale=1]{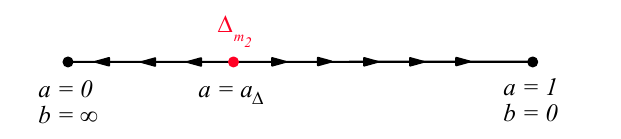}
  \caption{Flow on moduli space of one-dimensional Ising models}\label{fig:1}
  \end{figure}

Recall the topological boundary theories~\eqref{eq:11}.  The low energy
effective boundary field theory is~$\rB _0$ for the low temperature phase
and $\rB _{\bmut}$ for the high temperature phase.  We worked out the values
on~$\cir$ in Example~\ref{thm:6}.  Recall that $\rB (\cir)$~is an object in
the category~$\rG_{\bmut}(\cir)$ of $\bmut$-equivariant vector bundles over~$\bmut$,
so a pair~$V_1,V_{-1}$ of representations of~$\bmut$.  The underlying vector
spaces are the vacuum states of the Ising theory in the untwisted~($V_1$) and
twisted~($V_{-1}$) sectors; the $\bmut$-action is that of the global symmetry
which reverses all of the classical spins simultaneously.
 
For the low temperature phase the low energy field theory is~$\rB =\rB _0$.
The untwisted space ~$V_1$ is the regular representation of~$\bmut$, which is
realized as the space of functions on a set of two points which are exchanged
by the $\bmut$-action.  For each polygon~$\Pi \subset \cir$ the configuration
set of the Ising model is the set of functions $s \:\Vertices(\Pi) \to\bmut$;
the two-element vacuum set is the subset of constant functions, exchanged by
the global symmetry.  The two-dimensional vector space~$V_1$ is the lowest
eigenspace of the untwisted Hamiltonian.  The twisted space~$V_{-1}$ is the zero
vector space.  In the twisted sector there are many classical
configurations~$s$ of small energy---after trivializing the nontrivial double
cover $R\to\cir$ on~$\cir\setminus \{p\}$ for~$p\in \cir\setminus
\Vertices(\Pi )$ they each have a single vertex in~$\Pi $ at which $s$~takes
a value distinct from that at the other vertices.  The lowest energy quantum
state is the equally weighted superposition; the indicator function of this
subset of configurations.  Its energy is strictly larger than the minimal
energy in the untwisted sector.  At low temperature the energy differences
persist\footnote{We do not attempt a quantitative statement or proof here.}
independent of~$\Pi $, which leads to the vanishing of~$V_{-1}$.
 
By contrast, in the high temperature phase the energy differences do not
persist as the lattice~$\Pi $ is refined; note they are completely absent
at~$a=1$.  The low energy theory~$\rB =\rB _{\bmut}$ has $V_1,V_{-1}$~each the
trivial one-dimensional representation of~$\bmut$; the symmetry is unbroken in
each sector.  The vacuum state of the quantum theory on~$(\cir,\Pi )$ is an
equal mixture of all classical configurations $s\:\Vertices(\Pi )\to\bmut$, the
constant function on the configuration set.  This is true in both
sectors---at high temperature the global twisting of the double cover does
not affect the behavior of the Hamiltonian on~$\cir$.

   \section{The regular topological boundary theory}\label{sec:9}
% lastsubsec@  2

In this section we illustrate computations in an extended field theory with
boundary conditions and domain walls.  We use some of these results
in~\S\ref{sec:6}.  We begin in~\S\ref{subsec:9.1} with a review of oriented
2-dimensional extended field theories.  Then in~\S\ref{subsec:9.2} we
consider a 3-dimensional theory based on a spherical fusion category,
emphasizing the utility of the regular (Dirichlet) boundary theory
(Remark~\ref{thm:2}).

  \subsection{2-dimensional extended field theories}\label{subsec:9.1}

We summarize some aspects of~\cite{MS}; see also~\cite{La}
and~\cite[\S4.2]{L}.

Let $C$~be a semisimple linear category with finitely many simple objects.  A
\emph{Frobenius structure} (or \emph{Calabi-Yau structure}) on~$C$ is a
collection of nondegenerate\footnote{The associated bilinear pairing
$f_1,f_2\mapsto \tau _x(f_1f_2)$ on~$\EC x$ is nondegenerate.} traces
  \begin{equation}\label{eq:113}
     \tau _x\:\EC x\longrightarrow \CC,\qquad x\in C, 
  \end{equation}
which for every sequence 
  \begin{equation}\label{eq:114}
     x_0\xrightarrow{\;\;f_0\;\;} x_1\xrightarrow{\;\;f_1\;\;}
     x_2\longrightarrow \;\cdots\; \xrightarrow{\;\;f_{k-1}\;\;}
     x_k\xrightarrow{\;\;f_k\;\;} x_0 
  \end{equation}
of morphisms in~$C$ satisfy: $\tau _{x_i}(f_{i+k}\cdots f_{i+1}f_i)$ is
independent of $i=0,\dots ,k$.  (Set $f_{i+k+1}=f_i$ for $i=0,\dots ,k-1$.)
Equivalent data is a nondegenerate trace $\tau \:HH_0(C)\to\CC$ on the
Hochschild homology.  Let $x_1,\dots ,x_N$ be a representative set of simple
objects.  The data of~$\tau $ amounts to $N$~complex numbers
  \begin{equation}\label{eq:115}
     \lambda _i=\tau _{x_i}(\id_{x_i}),\qquad i=1,\dots ,N. 
  \end{equation}
The finite semisimple category~$C$ is 2-dualizable in the
2-category~$\Cat=\Cat_{\CC}$ of complex linear categories, and a Frobenius
structure is precisely $SO_2$-invariance data.  By the cobordism hypothesis
there is an associated 2-dimensional topological field theory 
  \begin{equation}\label{eq:116}
     \TC\:\Bord_2(SO_2)\longrightarrow \Cat 
  \end{equation}
of oriented manifolds with $\TC(\pt)=C$.  An object~$x\in C$ is the image of
$\CC\in \Vect$ under a morphism $\Vect\to C$ in~$\Cat$.  Apply the cobordism
hypothesis to construct a boundary theory\footnote{Since the group~$SO_1$ is
trivial we obtain a boundary theory of oriented manifolds.} $\rB_x\:1\to\TC$.
The pair~$(\TC,\rB_x)$ maps the bordism~(i) in Figure~\ref{fig:6} to the
object~$x\in C$.

  \begin{figure}[ht]
  \centering
  \includegraphics[scale=1]{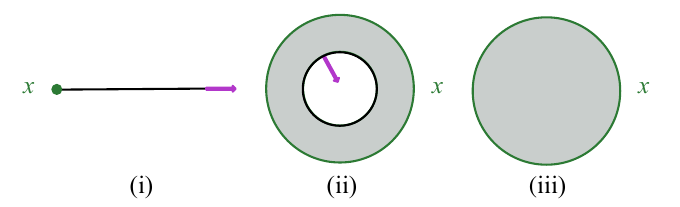}
  \caption{Three bordisms with boundary theory labeled by~$x\in C$.  The
  arrows indicate incoming vs.~outgoing boundary components.}\label{fig:6}
  \end{figure}

The vector space~$A=\TC(\cir)$ is a commutative Frobenius algebra.
Multiplication is $\TC$ ~applied to a pair of pants, the trace $\tau
_A\:A\to\CC$ is $\TC$~applied to a 2-disk~$D^2$ with $\partial D^2$~outgoing,
and the identity element~$\id_A$ is $\TC$~applied to~$D^2$ with $\partial
D^2$~incoming.  There is a vector space basis $e_1,\dots ,e_N$ of~$A$
consisting of idempotents; the order is not canonical.  The identity element
is $\id_A=e_1+\cdots +e_N$ and\footnote{Equation~\eqref{eq:117} follows from
(2.30), (2.31), and (2.36) in~\cite{MS}.}
  \begin{equation}\label{eq:117}
     \tau _A(e_i)=\lambda _i^2,\qquad i=1,\dots ,N. 
  \end{equation}
If $\OC$~is the finite set of isomorphism classes of simple objects in~$C$,
then we can identify $C=\Vect\bigl(\OC \bigr)$ as the category of vector
bundles on~$\OC$ and $A=\Fun\bigl(\OC \bigr)$ as the commutative algebra of
functions on~$\OC$.  Suppose the label~`$x$' in the last two bordisms of
Figure~\ref{fig:6} is a $\delta $-function at~$e_i\in A$.  Then bordism~(ii)
evaluates to an element~$\chi _i\in A$, which by~\cite[(2,36)]{MS} is
  \begin{equation}\label{eq:118}
     \chi _i=\frac{e_i}{\lambda _i}. 
  \end{equation}
Hence bordism~(iii) evaluates to 
  \begin{equation}\label{eq:119}
     \tau _A(\chi _i) = \lambda _i. 
  \end{equation}
The partition function of a closed oriented surface~$Y$ is computed by
Verlinde's formula 
  \begin{equation}\label{eq:120}
     \TC(Y) = \sum\limits_{i=1}^N\lambda _i^{\Euler(Y)} . 
  \end{equation}
In particular, 
  \begin{equation}\label{eq:121}
     \TC(S^2) = \sum\limits_{i=1}^N \lambda _i^2, 
  \end{equation}
which follows by decomposing the bordism $S^2\:\emptyset
^1\to\emptyset ^1$ as the composition $\emptyset
^1\xrightarrow{\;D^2\;}\cir\xrightarrow{\;D^2\;}\emptyset ^1$, from which we
conclude $\TC(S^2)=\tau _A(\id_A)$.

  \subsection{Computations in 3-dimensional extended field theories}\label{subsec:9.2}

Let $\sA$~be a spherical fusion category and 
  \begin{equation}\label{eq:122}
     \FAA\:\Bord_3(SO_3)\longrightarrow \tcat 
  \end{equation}
the associated 3-dimensional topological field theory of oriented manifolds
with $\FAA(\pt)=\sA$.  Recall (Remark~\ref{thm:2}) that a left
$\sA$-module~$\sL$ is a morphism $\sL\:1\to\sA$ in~$\tcat$.  If the
underlying category of~$\sL$ is semisimple with finitely many simple objects,
then $\sL$~is 2-dualizable.  To define a boundary theory $\rB\mstrut
_{\!\sL}\:1\to\FAA$ of \emph{oriented} manifolds we need $SO_2$-invariance
data on~$\sL$ as an $\sA$-module.  We call this a \emph{relative Frobenius
structure} on~$\sL$; it is relative to the pivotal structure on~$\sA$.  As in
\cite[\S7.12]{EGNO} let $\sE(\sL)=\End_{\sA}(\sL)$ be the tensor category
of $\sA$-module functors $\sL\to\sL$.  Then a relative Frobenius structure is
a nondegenerate trace on the vector space $\End_{\sE(\sL)}(1)$, where
$1\subset \sE(\sL)$ is the unit object. 
 
Specialize to the regular boundary theory with $\sL=\sA$ as a left
$\sA$-module.  Then $\sE(\sA)$~is identified with~$\sA$ (acting by right
multiplication on~$\sA$), and a relative Frobenius structure is a
nondegenerate trace on~$\EA1$.  Since we assume $\sA$~is fusion, $\EA1$~is a
1-dimensional vector space with canonical basis~$\id_1$.  Thus we
define the \emph{canonical relative Frobenius structure} on~$\sA$ to be the
trace with value~1 on~$\id_1\in \EA1$. 

  \begin{remark}[]\label{thm:55}
 For each object~$x\in \sA$ the pivotal structure~$\rho $ determines a linear
map $\EA x\to\EA1$ which sends $f\in \EA x$ to (compare~\eqref{eq:110})
  \begin{equation}\label{eq:123}
     1\xrightarrow{\;\;\coev_x\;\;} x\otimes x\dual\xrightarrow{\;\;f
     \otimes \id\;\;} x\otimes x\dual\xrightarrow{\;\;\rho _x\otimes \id\;\;}
     x^{\vee\vee}\otimes x\dual \xrightarrow{\;\;\ev_{x\dual}\;\;} 1 
  \end{equation}
Therefore, a relative Frobenius structure determines a Frobenius structure on
the category underlying~$\sA$.  In the sequel we use the canonical choice. 
  \end{remark}

  \begin{figure}[ht] 
  \centering
  \includegraphics[scale=1]{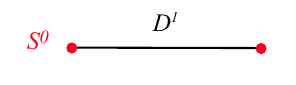}
  \caption{$D^1$ with boundary colored red by the regular boundary
  theory}\label{fig:7}
  \end{figure}

It is convenient to make pictorial computations for the pair~$(\FAA,\BA)$.
For example, consider the oriented interval~$D^1$ with boundary $\partial
D^1=S^0$ ``colored'' by the boundary theory~$\BA$, as rendered in
Figure~\ref{fig:7}.  To interpret it we first ignore the boundary theory,
read~$D^1$ as a bordism $\pt\amalg \pt\to \emptyset ^0$, and apply~$\FAA$ to
obtain the functor\footnote{More precisely, it is the bimodule in~$\tcat$
which represents the functor~\eqref{eq:124}.}
  \begin{equation}\label{eq:124}
     \begin{aligned} \Amod\times \Amod&\longrightarrow \Cat \\
      \sL_1\quad ,\;\quad \sL_2\;\;\;&\longmapsto
     \Hom_{\sA}(\sL_1,\sL_2)\end{aligned}  
  \end{equation}
Imposing the boundary condition amounts to setting $\sL_1=\sL_2=\sA$, thus
we obtain the linear category underlying~$\sA$. 

  \begin{remark}[]\label{thm:56}
 In $\Bord_3(SO_3)$ we must orient the tangent bundle of every bordism~$M$,
stabilized to have rank~3.  At an incoming boundary component~$(\partial
M)_0$, part of the data is an orientation-preserving isomorphism
  \begin{equation}\label{eq:125}
     \triv{\RR}\oplus \triv{\RR^k}\oplus T(\partial
     M)_0\xrightarrow{\;\;\cong \;\;} \triv{\RR^k}\oplus TM 
  \end{equation}
of vector bundles over~$(\partial M)_0$ which takes $1\in \triv{\RR}$ to an
inward pointing vector.  (Here $\dim M=3-k$.)  This applies to~$M=D^1$ with
both boundary points incoming, and it shows why the two boundary points can
be chosen to be isomorphic objects in $\Bord_3(SO_3)$.  Note that in a
1-dimensional field theory such an oriented bordism does not
exist.\footnote{Indeed, if $\rF\:\Bord_1(SO_1)\to\Vect$ assigns $\rF(\pt)=V$
for a finite dimensional vector space~$V$, a boundary theory $1\to\rF$ is
determined by~$\xi \in V$ and this data does not allow us to attach a number
to~$D^1$.}
  \end{remark}

  \begin{figure}[ht]
  \centering
  \includegraphics[scale=1.05]{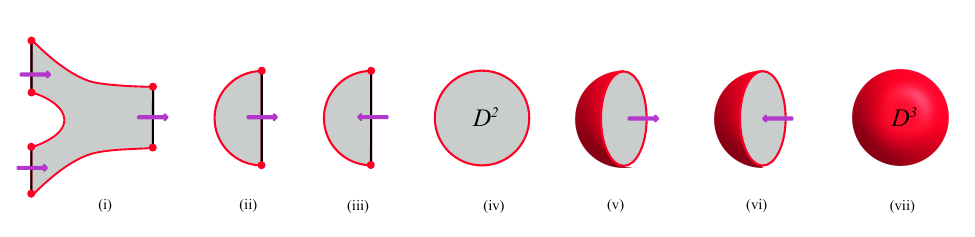}
  \caption{Seven bordisms to evaluate under~$(\FAA,\BA)$}\label{fig:8}
  \end{figure}

Figure~\ref{fig:8} shows seven morphisms in\footnote{In fact, there is a
variant bordism category which accounts for the colored boundaries;
see~\cite[\S4.3]{L}.} $\Bord_3(SO_3)$, with parts of boundaries colored
with~$\BA$.  The arrows indicate incoming vs.~outgoing boundary components.
We evaluate $(\FAA,\BA)$ on each in turn.  Bordism~(i) evaluates to the
tensor structure $\sA\otimes \sA\to\sA$.  Glue bordism~(ii) to an input of
bordism~(i) to deduce that it evaluates to the tensor unit~$1\in \sA$.
Morphism~(iii) evaluates to the linear functor
  \begin{equation}\label{eq:126}
     \begin{aligned} \sA&\longrightarrow \Vect \\ y&\longmapsto
      \Hom_{\sA}(1,y)\end{aligned} 
  \end{equation}

  \begin{figure}[ht]
  \centering
  \includegraphics[scale=1]{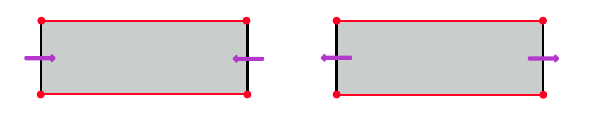}
  \caption{Evaluation and coevaluation}\label{fig:9}
  \end{figure}

\bigskip\noindent 
 This follows from the evaluation and coevaluation in
Figure~\ref{fig:9}---duality data for the underlying category~$\sA$---and
this can be done in the dimensional reduction~$\TA$ below.  Continuing, we
see bordism~(iv) in Figure~\ref{fig:8} evaluates to the vector space~$\EA1$;
bordism~(v) to $\id_1\in \EA1$; bordism~(vi) to the (canonical) relative
Frobenius trace on~$\EA1$; and bordism~(vii) to the number~1, by the
normalization in the canonical relative Frobenius structure. 

  \begin{figure}[ht]
  \centering
  \includegraphics[scale=1]{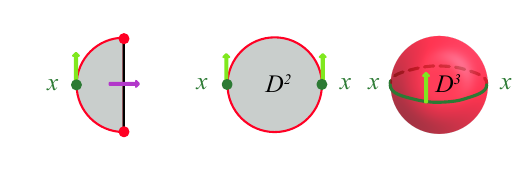}
  \caption{Three bordisms with domain wall labeled by~$x\in \sA$}\label{fig:10}
  \end{figure}

Tensoring on the right by an object~$x\in \sA$ defines a left $\sA$-module
map $\sA\to\sA$.  By the cobordism hypothesis this determines a domain
wall\footnote{Since the group~$SO_1$ is trivial there is no extra structure
necessary for~$\Dx$ in the \emph{oriented} theory~$\FAA$.} $\Dx\:\BA\to\BA$.
Consider the three bordisms in Figure~\ref{fig:10}.  We co-orient the
codimension two submanifold labeled~$x$ inside the codimension one boundary;
the co-orientation indicates the direction of the domain wall.  By cutting as
in Figure~\ref{fig:8} we evaluate bordism~(i) to the object~$x\in \sA$,
bordism~(ii) to the vector space~$\EA x$, and bordism~(iii) to the
number~$\dim(x)$.  For the latter we see first that bordism~(iii) is the
canonical trace of~$\EA x$ applied to~$\id_x$.  Comparing~\eqref{eq:123}
with~\eqref{eq:110} we compute it to be~$\dim(x)$.
 
Dimensional reduction of~$\FAA$ along the bordism~$b$ in Figure~\ref{fig:7}
is a 2-dimensional oriented theory 
  \begin{equation}\label{eq:127}
     \TA\:\Bord_2(SO_2)\longrightarrow \Cat 
  \end{equation}
with $\TA(\pt)=\FAA(b)=\sA$. 

  \begin{proposition}[]\label{thm:57}
 $\TA$~ induces the Frobenius structure~$\tau $ on the category
~$\sA$ which satisfies
  \begin{equation}\label{eq:128}
     \tau _x(\id_x)=\dim(x) 
  \end{equation}
for each object~$x\in \sA$ 
  \end{proposition}

\noindent
 Stated differently, $\tau $~is the canonical relative Frobenius structure
on~$\sA$.  

  \begin{figure}[ht]
  \centering
  \includegraphics[scale=.8]{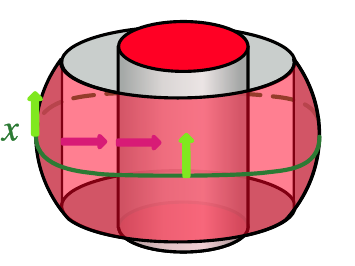}
  \caption{A decomposition of Figure~\ref{fig:10}(iii)}\label{fig:12}
  \end{figure}

  \begin{proof}
 The object~$x\in \sA$ determines the boundary condition~$\rB_{x}$ for~$\TA$
(Figure~\ref{fig:6}(i)) which corresponds to the domain wall~$\Dx$
in~$(\FAA,\BA)$ depicted in Figure~\ref{fig:10}(i).  Evaluate
Figure~\ref{fig:10}(iii) by decomposing as in Figure~\ref{fig:12}.  The outer
spherical annulus evaluates to the image of~$x$ in the commutative Frobenius
algebra~$\TA(\cir)$.  The inner solid cylinder evaluates to the trace
on~$\TA(\cir)$.  On the other hand, we already evaluated
Figure~\ref{fig:10}(iii) as~$\dim(x)$.  Now \eqref{eq:128}~follows
from~\eqref{eq:119}.
  \end{proof}

Recall that $\OA$~is the finite set of isomorphism classes of simple objects
in~$\sA$.   

  \begin{proposition}[]\label{thm:58}
 
  \begin{equation}\label{eq:129}
     \FAA(S^3) = \frac{1}{\sum\limits_{x\in
     \OA}\!\!\!\!{\dim(x)^2}^{\phantom{W}}}. 
  \end{equation}
  \end{proposition}

\noindent 
 A corollary is that the denominator---the \emph{categorical dimension} 
  \begin{equation}\label{eq:131}
     d(\sA) = \sum\limits_{x\in \OA}\!\!\dim(x)^2 
  \end{equation}
of~$\sA$---is nonzero, a known result~\cite[ Theorem~7.21.12]{EGNO}.

 \bigskip
  \begin{figure}[ht]
  \centering
  \includegraphics[scale=.8]{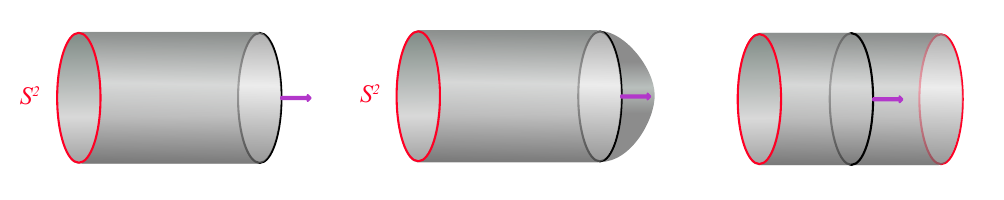}
  \caption{Three more bordisms with boundary theory}\label{fig:11}
  \end{figure}

  \begin{proof}
 First, $\FAA(S^2)$~is a commutative Frobenius algebra.  Let $\sigma
\:\FAA(S^2)\to\CC$ denote its trace.  There is an algebra isomorphism
$\FAA(S^2)\cong \End\mstrut _{\sZ(\sA)}(1)$, where $\sZ(\sA)=\FAA(\cir)$~is
the Drinfeld center of~$\sA$.  The latter is 1-dimensional, since $\EA1$~is,
and has canonical basis the unit~$u\in \FAA(S^2)$.  Then $\FAA(S^3)=\sigma
(u)$.  Bordism~(i) in Figure~\ref{fig:11} evaluates to a vector~$du$ for
some~$d\in \CC$.  Use bordism~(ii) to prove $\sigma (u)=d\inv $, since
bordism~(vii) in Figure~\ref{fig:8} evaluates to the number~1.  Bordism~(iii)
in Figure~\ref{fig:11} evaluates to $d^2\sigma (u)=d$, but as it is the
product of the bordism~$b$ of Figure~\ref{fig:7} with~$S^2$ it also evaluates
to
  \begin{equation}\label{eq:130}
     \TA(S^2) = \sum\limits_{x\in \OA}\dim(x)^2 = d(\sA), 
  \end{equation}
according to~\eqref{eq:121}. 
  \end{proof}

\bigskip
  \begin{figure}[ht]
  \centering
  \includegraphics[scale=1]{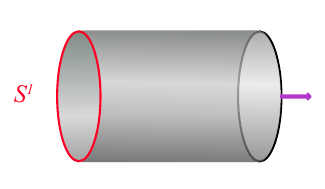}
  \caption{A special element of~$\sZ(\sA)$}\label{fig:13}
  \end{figure}

Finally, we observe that $(\FAA,\BA)$ maps the bordism in Figure~\ref{fig:13}
to a distinguished object in $\FAA(\cir)=\sZ(\sA)$.  Its simple constituents
generate a subcategory of ``abstract Wilson loop operators''.

   \section{Nonabelian lattice models and duality}\label{sec:6}
% lastsubsec@  2

Fix a spherical fusion category~$\sA$.  The cobordism hypothesis associates
to~$\sA$ an extended 3-dimensional topological field theory~$\FAA$ of
oriented manifolds.  This is the \emph{Turaev-Viro theory}~\cite{TV,BW1,BW2}
as generalized by Barrett-Westbury.  It has a state sum formulation.  In
principle the state sum can be derived from the extended field theory
structure, so ultimately from the cobordism hypothesis; see~\cite{Da} for a
careful treatment of state sums in 2-dimensional extended oriented field
theories.  In~\S\ref{subsec:6.1} we state the relevant formulas for
Turaev-Viro theory, mostly following~\cite{BK}.  Then in~\S\ref{subsec:6.2}
we construct a boundary theory for latticed manifolds, at least on closed
oriented latticed 2-manifolds.  For $\sA=\Vect[G]$ we recover the finite
gauge theory of~\S\ref{sec:4}.  Applied to $\sA=\Rep(G)$ we obtain a
(Kramers-Wannier) dual\footnote{Recall the discussion in~\S\ref{subsec:3.5}.}
to the $G$-Ising model for all finite groups~$G$, including $G$~nonabelian.

  \subsection{Turaev-Viro theory}\label{subsec:6.1}

Let $\sA$~be a spherical fusion category.  For objects $x_1,\dots ,x_n\in
\sA$ define the vector space
  \begin{equation}\label{eq:93}
     \langle x_1,\dots ,x_n \rangle = \HA(1,x_1\otimes \cdots\otimes
     x_n), 
  \end{equation}
where $1\in \sA$ is the tensor unit.  Then the pivotal structure produces an
isomorphism 
  \begin{equation}\label{eq:94}
     \langle x_1,\dots ,x_n \rangle\xrightarrow{\;\;\cong
     \;\;}\HA(x_1\dual,x_2\otimes \cdots\otimes
     x_n)\xrightarrow{\;\;\cong \;\;}\langle x_2,\dots ,x_n,x_1 \rangle 
  \end{equation}
whose $n^{\textnormal{th}}$ power is the identity.  Therefore, $\langle
x_1,\dots ,x_n  \rangle$ depends only on the set of cyclically ordered
objects $x_1,\dots ,x_n$.  Furthermore, for $x_i,y_i,z\in \sA$ there is a
pairing 
  \begin{equation}\label{eq:95}
     \langle x_1,\dots ,x_n,z\dual \rangle\otimes \langle z,y_1,\dots ,y_m
     \rangle\longrightarrow \langle x_1,\dots ,x_n,y_1,\dots ,y_m \rangle 
  \end{equation}
defined by tensoring $\HA$~spaces and contracting via evaluation
$z\dual\otimes z\to 1$.  For any $x_1,\dots ,x_n\in \sA$ the pairing
  \begin{equation}\label{eq:99}
     \langle x_n\dual,\dots ,x_1\dual \rangle\otimes \langle x_1,\dots ,x_n
     \rangle\longrightarrow \EA1\longrightarrow  \CC 
  \end{equation}
is nondegenerate.  The second arrow is the canonical relative Frobenius
structure defined in~\S\ref{subsec:9.2}.  Henceforth, we restrict to
\emph{simple} objects of~$\sA$.
 
Suppose $(Y,\Lambda )$~is a closed oriented latticed 2-manifold.  Let
$\OEdges(\Lambda )$ be the set of pairs $e=(\be,\mathfrak{o})$ in which
$\be\in \Edges(\Lambda )$ and $\mathfrak{o}$~is an orientation of~$\be$.  Let
$-e$~denote the oppositely oriented edge to~$e$.  Define a \emph{labeling} as
a function 
  \begin{equation}\label{eq:96}
     \ell \:\OEdges(\Lambda )\longrightarrow \Obj\sA 
  \end{equation}
which satisfies: $\ell (e)$~is a simple object for all~$e\in \OEdges(\Lambda
)$, and $\ell (-e)=\ell (e)\dual$.  We say labelings $\ell \sim \ell '$ are
equivalent if $\ell (e)\cong \ell '(e)$ for all $e\in \OEdges(\Lambda )$.
Define the vector space \cite[(3.2.2)]{BK}
  \begin{equation}\label{eq:97}
     \VAA(Y;\Lambda )\cong \bigoplus\limits_{\ell }\bigotimes\limits_{f} \;\langle \ell
     (e_1),\dots ,\ell (e_n) \rangle, 
  \end{equation}
where the direct sum is over equivalence classes of labelings~$\ell $, the
tensor product over faces~$f$ of~$\Lambda $, and the boundary of a face~$f$
is an $n$-gon with edges $e_1,\dots ,e_n$ cyclically ordered and oriented as
the boundary of~$f$.  It has the following interpretation in the extended
field theory~$\FAA$.  Namely, as in~\S\ref{sec:9} let~$\bA$ be the regular
boundary theory defined by~$\sA$ as a left module over itself.  Excise open
disks about each vertex of~$\Lambda $ to obtain a surface~$Y'$; insert the
regular boundary theory~$\bA$ on~$\partial Y'$.  Then $\FAA(Y')$ is the
vector space~\eqref{eq:97}, as we briefly explain.  Compute~$\FAA(Y')$ by
cutting along each edge of~$\Lambda \cap Y'$ to decompose~$Y'$ as a union of
polygons with vertices blown-up to arcs colored with~$\bA$; see
Figure~\ref{fig:14}.  Each such polygon evaluates to a vector
space~\eqref{eq:93}, once labels are inserted on the edges.  Evaluation of
the thin rectangles attached to doubled edges leads to the sum and product
in~\eqref{eq:97}.

  \begin{figure}[ht]
  \centering
  \includegraphics[scale=1]{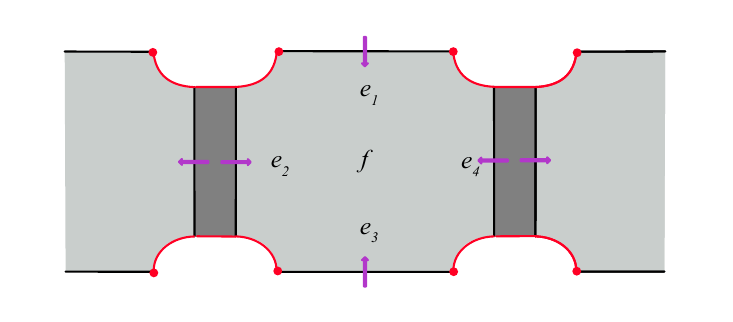}
  \caption{The blown-up lattice}\label{fig:14}
  \end{figure}

  \begin{example}[$\sA={\Vect[G]}$]\label{thm:48}
 A simple object is a ``skyscraper vector bundle'', a line supported at a
single element~$g\in G$.  Hence an equivalence class of labelings is a
function $g\:\OEdges(\Lambda )\to G$ such that $g(-e)=g(e)\inv $.  Let
$\CC_g$~be the skyscraper trivial line supported at~$g$ and $*$~the
convolution product~\eqref{eq:8}.  We easily compute the vector space
  \begin{equation}\label{eq:98}
     \langle g_1,\dots ,g_n \rangle = \Hom\mstrut
     _{\Vect[G]}(1,\CC_{g_1}*\cdots*\CC_{g_n}) = \begin{cases} \CC,&g_1\cdots
     g_n=e;\\0,&\textnormal{otherwise}.\end{cases} 
  \end{equation}
The set of labelings is equivalent to the groupoid of principal $G$-bundles
over~$\Lambda $ together with a section over $\Vertices(\Lambda )$; the
equivalence is parallel transport along edges.  (Recall that the section over
vertices is the canonical Dirichlet boundary theory; see the end
of~\S\ref{subsec:3.1}.)  According to~\eqref{eq:98} the vector space
in~\eqref{eq:97} for fixed~$\ell ,f$ is zero unless the holonomy around the
face~$f$ is trivial.  Therefore, we can identify~\eqref{eq:97} with the
vector space of functions on~$\bung {Y,\Vertices(\Lambda )}$, the stack of
$G$-bundles on~$Y$ trivialized at vertices of~$\Lambda $.
  \end{example}

  \begin{figure}[ht]
  \centering
  \includegraphics[scale=1]{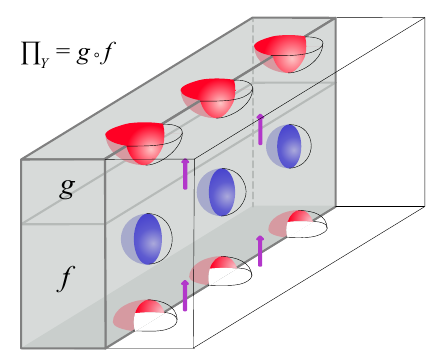}
  \caption{A schematic of the bordism $\PY=g\circ f$ with blue antispheres}\label{fig:15}
  \end{figure}

To compute the vector space~$\FAA(Y)$ we must glue back in the excised disks
about vertices.  We cannot literally do that, due to the boundary theory, so
we take a different route.  First, recall from the proof of
Proposition~\ref{thm:58} that $\FAA(S^2)$ is 1-dimensional with canonical
basis element~$u$.  Furthermore, $S^2$~colored with the boundary theory~$\BA$
defines the vector $d(\sA)u\in \FAA(S^2)$---this is the evaluation of the
relative theory, which amounts to evaluating~$(\FAA,\BA)$ on bordism~(i) in
Figure~\ref{fig:11}---for the nonzero constant~$d(\sA)$ defined
in~\eqref{eq:131}.  Define the ``antisphere'' to be a colored~$S^2$
postulated to have value $d(\sA)\inv u\in \FAA(S^2)$.  Now define a bordism
$\PY\:Y'\to Y'$ as the manifold~$[0,1]\times Y$ with, for each~$v\in
\Vertices(\Lambda )$; a ``hemiball'' removed about each vertex $\{0\}\times
v$, $\{1\}\times v$ at the boundary; and a small ball removed
at~$\{1/4\}\times v$ with boundary an antisphere.  Put the boundary
theory~$\bA$ on the hemisphere boundary of each excised hemiball.
Figure~\ref{fig:15} is a 2-dimensional slice through~$\PY$.  The
theory~$\FAA$ evaluates on each end of~$\PY$ to the vector space
$\FAA(Y')=\VAA(Y;\Lambda )$.

  \begin{lemma}[]\label{thm:49}
 $\FAA(\PY)\:\VAA(Y;\Lambda )\longrightarrow \VAA(Y;\Lambda )$ is a projection with image
isomorphic to~$\FAA(Y)$.
  \end{lemma}

  \begin{proof}
 The composition $\PY\circ \PY$ in the bordism category is diffeomorphic
to~$\PY$ with (i)~an interior ball excised for each $v\in \Vertices(\Lambda
)$ with boundary colored by~ $\bA$, and (ii)~two antispheres for each vertex.
 Use the multiplication in~$\FAA(S^2)$ to cancel the $\BA$-sphere against one
of the antispheres.  To complete the proof, refer to Figure~\ref{fig:15} and
observe that $f\circ g\approx \id_Y$, after canceling the $\bA$-spheres
against the antispheres, which implies that $\FAA(g)$~is an isomorphism
of~$\FAA(Y)$ onto the image of~$\FAA(\PY)$.
  \end{proof}

  \begin{figure}[ht]
  \centering
  \includegraphics[scale = .7]{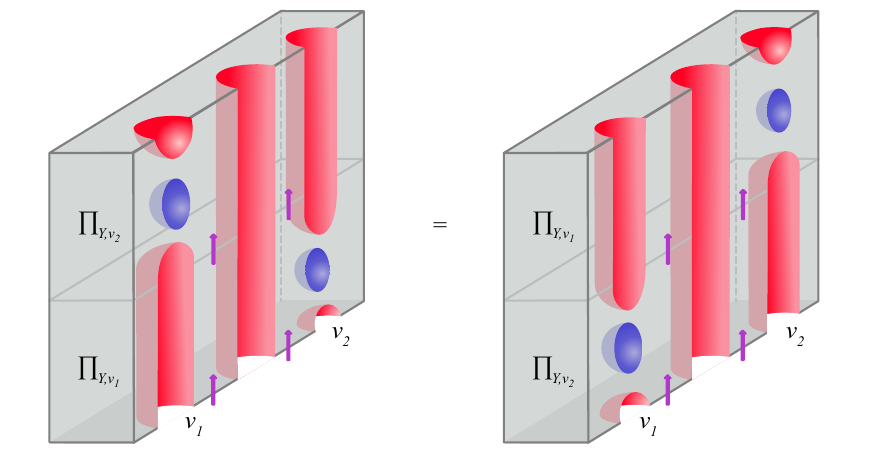}
  \caption{Commutativity of local projections}\label{fig:5}
  \end{figure}
 \bigskip

Decompose~$\FAA(\PY)$ as a composition of \emph{commuting} projections
$\FAA(\PYv)\:\VAA(Y;\Lambda )\to\VAA(Y;\Lambda )$, $v\in \Vertices(\Lambda )$, by setting
  \begin{equation}\label{eq:100}
     \PYv=[0,1]\times Y'\cup \mstrut _{[\frac1{16},\frac{15}{16}]\times \partial
     _vY'}\,\left[\frac 1{16},\frac{15}{16}\right]\times D^2 \;\;\setminus
     \;\;\left[\frac 
     18,\frac 38\right]\times D^2,  
  \end{equation}
where $\partial _vy'\approx \cir$ is the boundary of the excised ball
about~$v$ and the boundary of $[\frac 18,\frac 38]\times D^2$ is the
antisphere.  There is a standard ``picture proof'' (Figure~\ref{fig:5}) that
for vertices~$v_1,v_2$ we have
  \begin{equation}\label{eq:104}
     \Pi \mstrut _{Y,v_1}\circ \Pi \mstrut _{Y,v_2} = \Pi \mstrut
     _{Y,v_2}\circ \Pi \mstrut _{Y,v_1}  
  \end{equation}
in the bordism category.

  \begin{figure}[ht]
  \centering
  \includegraphics[scale=.9]{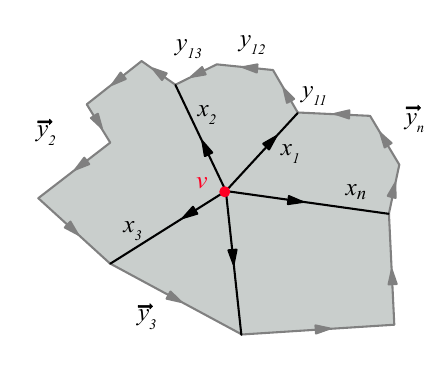}
  \caption{The labeled closed star of~$v$}\label{fig:4}
  \end{figure}
 \bigskip

We conclude with a formula for~$\FAA(\PYv)$ in terms of the state
sum~\eqref{eq:97}.  Figure~\ref{fig:4} depicts the ``closed star'' of the
vertex~$v$ and defines a labeling of its edges.  For any labeling~$\ell $
in~\eqref{eq:97} with these labels, the desired projection is the identity on
the vector space~\eqref{eq:93} associated to faces which do not contain~$v$,
tensored with the linear map
  \begin{equation}\label{eq:101}
     \begin{aligned} \bigotimes\limits_{i=1}^n\;\langle x_i,\vy_i,x_{i+1}\dual
      \rangle&\xrightarrow{\;\;\phantom{1/d(\sA)}\;\;}\;
      \bigoplus\limits_{\tx_i,z}\,\bigotimes\limits_{i=1}^n \;\langle
      x_i,\vy_i,x_{i+1}\dual \rangle\otimes \langle x_i, \tx_i\dual,z \rangle
     \otimes \langle 
      z\dual,\tx_i,x_i\dual \rangle \\ &\quad \;\; =\quad\;\;
      \bigoplus\limits_{\tx_i,z}\,\bigotimes\limits_{i=1}^n \;\langle
      x_i,\vy_i,x_{i+1}\dual \rangle\otimes \langle x_{i+1}, \tx_{i+1}\dual,z
     \rangle \otimes 
      \langle z\dual,\tx_i,x_i\dual \rangle \\
     &\xrightarrow{\;\;\phantom{1/d(\sA)}\;\;}\; 
      \bigoplus\limits_{\tx_i}\bigotimes\limits_{i=1}^n\;\langle
      \tx_i,\vy_i,\tx_{i+1}\dual\rangle\\&\xrightarrow{\;\;1/d(\sA)\;\;}\;
      \bigoplus\limits_{\tx_i}\bigotimes\limits_{i=1}^n\;\langle
      \tx_i,\vy_i,\tx_{i+1}\dual\rangle,\end{aligned} 
  \end{equation}
where we set $x_{n+1}=x_1$ and $\tx_{n+1}=\tx_1$.  The first map is
constructed from the dual of the duality pairing~\eqref{eq:99}.  The second
map is defined using the cyclic invariance~\eqref{eq:94} and the
pairing~\eqref{eq:95}:
  \begin{equation}\label{eq:102}
     \begin{aligned} \langle x_i,\vy_i,x_{i+1}\dual \rangle\otimes \langle
     x_{i+1}, &
       \tx_{i+1}\dual,z \rangle \otimes \langle z\dual,\tx_i,x_i\dual \rangle
       \\ &\longrightarrow \langle \tx_{i+1}\dual,z,x_{i+1} \rangle\otimes
       \langle x_{i+1}\dual,z\dual,\tx_i,\vy_i \rangle \\&\longrightarrow
     \langle \tx_i,\vy_i,x_{i+1}\dual,x_{i+1},\tx_{i+1}\dual  \rangle 
       \\&\longrightarrow
       \langle \tx_i,\vy_i,\tx_{i+1}\dual \rangle\end{aligned} 
  \end{equation}
The last map is the indicated multiplication operator. 

  \begin{figure}[ht]
  \centering
  \includegraphics[scale=.8]{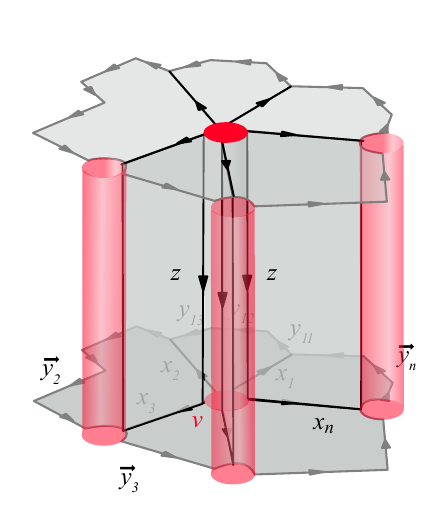}
  \caption{A small portion of~$\PYv$}\label{fig:17}
  \end{figure}

  \begin{proposition}[]\label{thm:59}
 The operator on~$\VAA(Y;\Lambda )$ defined in~\eqref{eq:101} is~$\FAA(\PYv)$. 
  \end{proposition}

  \begin{proof}
 We sketch a proof based on Figure~\ref{fig:17} and our computations
in~\S\ref{subsec:9.2}.  The left drawing depicts a decomposition of a portion
of~$\PYv$.  A solid cylinder has been bored from the north pole of the bottom
hemisphere at~$v$ to the south pole of the top hemisphere.  Its boundary is
co-oriented to point into the faces containing~$v$.  Edges of~$\Lambda $ are
doubled, as in Figure~\ref{fig:14}.  Each face of~$ $appears in the
decomposition of~$\PYv$ bounded by vertical polygons as in the right drawing;
the pair~$(\FAA,\BA)$ maps that particular polygon to the vector
space~$\langle x\mstrut _i,\tx\dual_i,z \rangle$.  There is a thin vertical
``slab'' for each edge.  The vertical faces of the slabs are outgoing.  The
bottom face of~$\PYv$ is incoming and the top face outgoing, as in
Figure~\ref{fig:15}.  The value of~$(\FAA,\BA)$ on the thin vertical slabs
are the coevaluation maps in the first line of~\eqref{eq:101}.  The
contractions in~\eqref{eq:102} occur when evaluating the face along the two
shown vertical polygons, bottom face, and vertical cylinder, each of which is
incoming.  The first contraction is gluing along~$x_i$, the second along~$z$,
and the third along~$x_{i+1}$.  The top face being outgoing there are no
contractions with the tilde variables.  The vertical solid cylinder can be
evaluated as a solid polygon in the reduced theory~$\TA$ of~\eqref{eq:127}.
This enforces that all vertical variables~$z$ be equal.  Finally, the
antisphere induces the multiplication operator at the last step
of~\eqref{eq:101}.  
  \end{proof}

  \begin{example}[$\sA={\Vect[G]}$]\label{thm:50}
 Resuming Example~\ref{thm:48}, we claim \eqref{eq:101} projects onto
functions on $\bung{Y,\Vertices(\Lambda )}$ which are invariant under change
of trivialization at~$v$.  To see this, recall that each label is a group
element, note that $z\in G$~functions as the change of trivialization at~$v$,
and $\tx_i$ is determined by~$x_i$ and~$z$.  Applying to all~$v$ we obtain
projection onto the image of 
  \begin{equation}\label{eq:103}
     \Fun\bigl(\bung Y \bigr)\longrightarrow
     \Fun\bigl(\bung{Y,\Vertices(\Lambda )} \bigr). 
  \end{equation}
  \end{example}

  \subsection{The lattice boundary theory}\label{subsec:6.2}

The boundary theory depends on a \emph{fiber functor}, i.e., a tensor functor 
  \begin{equation}\label{eq:105}
     \phi \:\sA\longrightarrow \Vect_\CC.
  \end{equation}
Furthermore, we need the generalization of the Ising weight function
in~\S\ref{subsec:4.1}.  First, duality is an involution
on~$\mathcal{O}(\sA)$.  Choose representative simple objects $x_1,\dots ,x_N$
in~$\sA$ and duality data for dual pairs of objects.  The desired weight
function is an element
  \begin{equation}\label{eq:106}
     \theta \in \bigoplus\limits_{i=1}^N \phi (x_i)\otimes \phi (x_i)^*
  \end{equation}
which satisfies the following ``evenness'' constraint: the duality data
induces an involution of the vector space (which permutes terms), and we
require that $\theta $~be invariant.  We assume $\phi $~and $\theta $~are
specified.

  \begin{example}[]\label{thm:51}
 For $\sA=\Vect[G]$ the fiber functor~$\phi $ maps a vector bundle $W\to G$
to the vector space $\bigoplus_{g\in G}W_g$.  If $x\in \Vect[G]$ is a
skyscraper line supported at~$g\in G$, then $\phi (x)\otimes \phi (x)^*$ is
the trivial line~$\CC$, and we can identify the vector space
in~\eqref{eq:106} with $\Fun(G)$.  The constraint forces $\theta $~ to be an
even function, as in Definition~\ref{thm:18}.
  \end{example}

  \begin{example}[]\label{thm:52}
 For $\sA=\Rep(G)$ the vector space in~\eqref{eq:106} is the home of the
``Fourier transform'' of a function on~$G$; see~\eqref{eq:47}. 
  \end{example}

Observe that for any object~$x\in \sA$ there is a canonical isomorphism 
  \begin{equation}\label{eq:107}
     \Hom\bigl(1,\Hom(1,x)^*\otimes x \bigr)\xrightarrow{\;\;\cong \;\;}
     \Hom(1,x)^*\otimes \Hom(1,x) 
  \end{equation}
and a canonical element in the latter.  Hence for a latticed
surface~$(Y,\Lambda )$ there is a canonical map
  \begin{equation}\label{eq:108}
  \begin{aligned}
     1\longrightarrow &\bigoplus\limits_{\ell }\bigotimes\limits_f\;\langle
     \ell (e_n)\dual,\dots ,\ell (e_1)\dual \rangle^*\otimes \ell
     (e_n)\dual\otimes \cdots\otimes \ell (e_1)\dual \\  
     &\!\!\!\!\!\!\!\cong \bigoplus\limits_{\ell }\bigotimes\limits_f
     \;\;\;\;\,\langle
     \ell (e_1),\dots ,\ell (e_n) \rangle\;\;\;\otimes \ell (e_n)\dual\otimes
     \cdots\otimes \ell (e_1)\dual ,
  \end{aligned}
  \end{equation}
where $e_1,\dots ,e_n$ are the edges of the $n$-gon~$f$ and we use the
duality pairing~\eqref{eq:99}.  Apply the fiber functor to obtain a linear
map of vector spaces
  \begin{equation}\label{eq:109}
     \CC\longrightarrow \bigoplus\limits_{\ell }\bigotimes\limits_f\;\langle
     \ell (e_1),\dots ,\ell (e_n) \rangle\otimes \phi \bigl(\ell
     (e_n)\bigr)^*\otimes \cdots\otimes \phi \bigl(\ell (e_1) \bigr)^*.
  \end{equation}
For a fixed labeling~$\ell $, each unoriented edge $\be\in \Edges(\Lambda )$
appears in~\eqref{eq:109} twice, once with each orientation.  Pair the image
of~$1\in \CC$ under~\eqref{eq:109} with $\prod_{e\in \Edges(\Lambda )}\theta
$ to obtain a vector in the vector space~$\VAA(Y;\Lambda )$ of~\eqref{eq:97}.
(The constraint on~$\theta $ is needed for this pairing to be well-defined.)
Finally, apply the projection~$\FAA(\PY)$ of Lemma~\ref{thm:49} to obtain the
partition function of the boundary lattice theory, a vector in~$\FAA(Y)$.

  \begin{example}[]\label{thm:53}
 For $\sA=\Vect[G]$ this prescription reproduces~\eqref{eq:57}, as we encourage
the reader to verify. 
  \end{example}

   \section{Bicolored boundary structures}\label{sec:8}
% lastsubsec@  1

Here, we give an abstract reformulation of our TQFT interpretation of the
Ising model and its correlators in the setting of \emph{fully extended}
topological field theories recalled in \S\ref{sec:2}: we will describe the
data of a $d$-dimensional lattice theory relative to a $(d+1)$-dimensional
topological field theory~ $\cF$ in the language of boundary and defect
structures. Unless otherwise specified, we will dwell in the world of
\emph{oriented} theories. 

As mentioned in the introduction, the formulation of Kramers-Wannier duality
in the context of a (nonextended) bicolored TQFT is anticipated in~\cite{S}.

\subsection{Bi-coloring via Morse functions.}\label{subsec:8.1a} 
 We assume at the outset that the $d$-manifold $Y$ ($d\ge 2$) carrying the
lattice theory is closed. Permitted cuts of $Y$ into cornered pieces will be
apparent after our reformulation: all faces of a cornered bordism must be
transverse to all boundary and defect structures.  Contact with the world of
extended TQFTs is made by converting a lattice on $Y$ into a handle
decomposition.  There is a standard way to do so, once we make the notion of
``lattice'' precise enough. We will take it to mean a \emph{piecewise linear}
decomposition $\Pi$ into convex polyhedra (not necessarily simplices),
compatible with the smooth structure. From this, one can build a Morse
function $f$ on $Y$ in a standard way, with one critical point on each
polyhedral face (including the $d$-dimensional ones). We will arrange for the
Morse function to be self-indexing, with vertices of degree $0$ and maxima of
degree $d$. The Morse function $d-f$ is then compatible, in the same way,
with the \emph{dual polyhedral decomposition}~ $\Pi^\vee$, which itself is
uniquely defined up to smooth isotopy.  These constructions are readily
carried out thanks to Whitehead's results \cite{Wh} on triangulations of
$C^1$ manifolds.

\label{colorcollar}
 The lattice theory on $Y$ being a boundary theory for $\cF$, the associated
TQFT picture is always based on a collar of $Y$ in some $(d+1)$-manifold; the
topological nature of $\cF$ confines all the information to a cylinder
$[0,1]\times Y$, with boundary $Y_0\amalg Y_1$, with $Y_0$ soon to be
decorated by boundary and defect structures. Its algebraic reading will
output a vector in the image $\cF(Y_1)$ of the smooth boundary, but the
incoming boundary $Y_0$ needs a preliminary change.

\begin{figure}[ht]
\centering
\includegraphics[scale=.9]{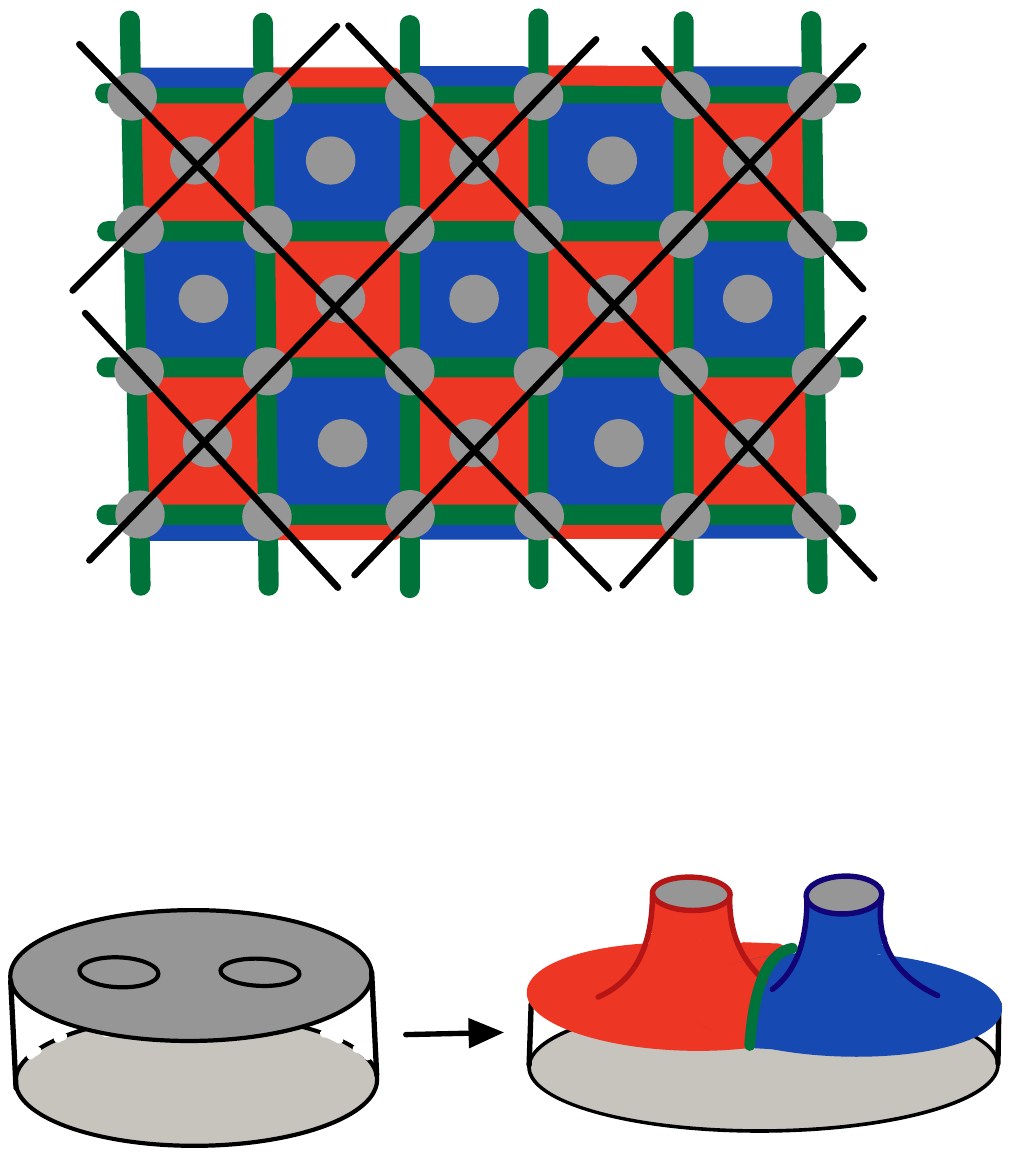}
\caption{Modifying $[0,1]\times Y$ into a cornered bordism}\label{corners}
\end{figure}

The Morse function $f$ gives a handle decomposition of $Y$, with a $p$-handle
$D^p\times D^q$ centered at each critical point $y$ of index $p$. Inside each
handle, mark out a smaller disk $D_y$, whose boundary is Hopf-split into two
solid tori $D^p\times S^{q-1}\cup S^{p-1}\times D^q$ glued along their
boundary torus $T_y=S^{p-1}\times S^{q-1}$, the level-set $f=p$ in $\partial
D_y$. We now deform the smooth structure on $[0,1]\times Y$ into a manifold
$X$ with corners of co-dimension $2$, converting each copy of $D_y\subset
Y_0$ into a face with corner $\partial D_y$; a complementary face is
$Y_0\setminus\coprod_yD_y=:Y_c$ (``$Y$ chromatic"), while the output $Y_1$
remains smooth and closed. The pair $(X, Y_c)$ is now a cornered bordism from
$\coprod_y(D_y,\partial D_y)$ to $(Y_1, \emptyset)$, as in
Figure~\ref{corners}.

We will use on $Y_c$ two colors (red and blue) to indicate boundary theories
$\rB,\rB'$ for $\cF$, as well as a collard green colored defect $\cD:\rB\to
\rB'$. Red and blue fill, alternately, the slices between critical level sets
$f=0,\dots ,d$, while green colors the level sets $f=1,\dots,d-1$.
Figure~\ref{sinsin} below gives an aerial view of the square lattice (in
grey) converted by means of the Morse function $\sin x\sin y$.

\begin{figure}[ht]
\centering
\includegraphics[scale=1]{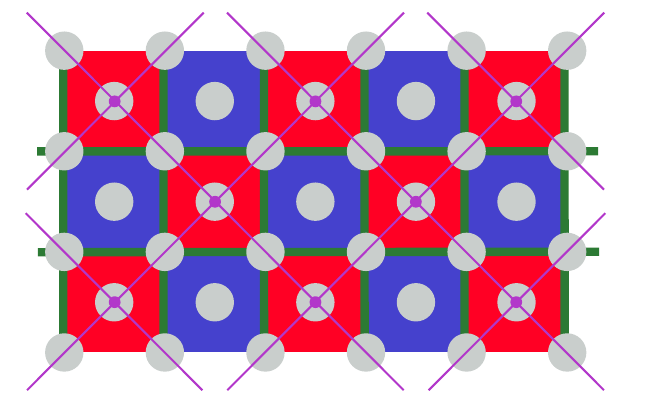}
\caption{Converting the square lattice to a bi-colored surface}\label{sinsin}
\end{figure}

\begin{remark}[Framings]
Recall that boundary theories, in our convention, are \emph{directed} from $1$ to $\cF$ \eqref{eq:3}. 
In the general setting of framed TQFTs, this convention picks the inward normal as a first vector of 
a $(d+1)$-framing on the red/blue parts, and a second vector tangent to the Morse flowlines on the 
green defect, pointing from red to green. We will specialize to oriented field theories, where we 
remember only orientations, with an orientation on the bulk now determines all the colored strata. 
Note, however, that when $d=2$, as for the Ising model, the Morse flow converts an orientation to a 
framing over $Y_c$, so in that dimension all the framing information of the theory is captured in 
these three handles.  
\end{remark}

The two solid tori in the boundary of each $p$-handle $D_y$ have opposite colors, assigned according to 
the parity of $p$, and separated by a green $T_y$. Thus colored, denote the boundary $\partial_cD_y$. 
Choose now a standard Euclidean model $D^{p,q}$ for this structured $D_y$; the 
quadruple $\cQ:=(\cF,\rB,\rB',\cD)$ associates to $(D^{p,q},\partial_cD^{p,q})$ a vector space $H_p$. 
This space, which plays a special role as a home of the Ising action below,  has a rich structure 
(already mentioned earlier, in \S\ref{frobhopf}); let us just mention its $E_p\times E_q$-bi-algebra 
structure, seen by writing $D^{p,q}=D^p\times D^q$ and multiplying in the $D^p$ and $D^q$ directions, 
respectively, and illustrated for $D^{1,1}$ in Figures~\ref{mult} and \ref{quadratics} below.

After coloring $Y_c$, $\cQ$ converts the bordism $X$ into a morphism
\begin{equation}\label{colory}
\cQ(X,Y_c): \bigotimes\nolimits_y \cQ(D_y,\partial_cD_y) \to \cF(Y).
\end{equation}
\begin{definition}\label{isingpart}
An \emph{Ising action} for $\cQ$ is a collection of vectors $\theta_p\in H_p$, each of which must 
be invariant under the oriented symmetry group $\mathrm{S}\left(\mathrm{O}(p)\times\mathrm{O}(q)\right)$ 
of $D^{p,q}$. The \emph{Ising partition function} in $\cF(Y)$ is the $\cQ(X,Y_c)$-image of their tensor 
product ranging over all handles. 
\end{definition}

\begin{remark} Since $H_p$ is a vector space and the action is topological,
only the components $\{\pm1\}$ of the symmetry group can act on it (and then
only if $0<p<d$). In a derived setting, there could be additional invariance
data from the continuous part of the group.  \end{remark}

Developing this, we introduce line defects of types $0\le p\le d$ for $\cF$,
allowed to end at point defects on $Y$. Locating the latter at critical
points of index $p$ allows us to promote them to ``order operators of type
$p$'' in the lattice theory as follows. A general line defect for $\cF$ is
labeled by an object in the category $\cF(S^{d-1})$. Such a defect requires a
normal framing; to collapse the framing information to $\mathrm{SO}(d)$, or
to another subgroup $S\subset\mathrm{O}(d)$, the label must be $S$-invariant.

\begin{remark} In the strict case, $S$ acts via its
$(\pi_0,\pi_1)$-truncation, represented on the automorphism $2$-group of the
category $\cF(S^{d-1})$. This involves a permutation action of $\pi_0S$,
along with a representation of $\pi_1 S$ by central automorphisms on each
object, with a natural compatibility. When $\pi_0=\{1\}$, an object is
invariant if $\pi_1$ acts trivially on it; otherwise, there is an additional
datum of trivializing the $\pi_0$-action on its orbit.  \end{remark}

For each $p$, a distinguished object $W_p$ is the $\cQ$-output at $1$ of the
cylinder $[0,1]\times\partial D^{p,q}$, with boundary at $0$ colored as
$\partial_cD^{p,q}$. Its construction ensures invariance under
$\mathrm{S}\left(\mathrm{O}(p)\times\mathrm{O}(q)\right)$.  An object $W_y$
is associated to each critical point $y$ of index $p$, identified with $W_p$
upon a standardization\footnote{Standardizations move under
$\mathrm{S}\left(\mathrm{O}(p)\times\mathrm{O}(q) \right)$; in the strict
setting, one is picked by orienting the descending Morse disk at $y$, the
lattice $p$-face.}  $D_y\cong D^{p,q}$.

\begin{definition}\label{abstractorderdef} A \emph{$p$-defect} is a line
defect whose label $\delta\in \cF(S^{d-1})$ lies in the subcategory linearly
generated by $W_p$, and which is invariant under
$\mathrm{S}\left(\mathrm{O}(p)\times\mathrm{O}(q)\right)$.  A \emph{$p$-order
operator at $y$} is a $p$-defect label $\delta$ plus a vector in
$\mathrm{Hom}(W_y,\delta)$.  \end{definition}

Thus defined, $p$-defects require $(p,q)$-splittings of the normal bundle;
other variants are possible. Note that
$H_p=\mathrm{Hom}(W_p,\mathbf{1})$. When $Y$ is marked with defects
$\delta_k$ at some critical points of matching index, we denote the space of
states $\cF(Y;\vec{\delta}\,)$.  A null-bordism of the defect within the
cylinder $X$ is converted by $\cF$ to a linear map $\cF(Y;\vec{\delta}\,)\to
\cF(Y)$.

\begin{proposition}[Correlators] With $p$-order operators placed at critical
points of matching index in $Y_c$, $\cQ$ determines a \emph{defective Ising
partition function} in $\cF(Y;\vec{\delta}\,)$. A null-bordism of the defect
maps it to the \emph{Ising correlator} in $\cF(Y)$.  \end{proposition}

We spell this out in the next subsection; for the sake of definiteness, we
restrict ourselves to the case $d=2$, with the electromagnetic Wilson and
't~Hooft defects, of indices $0$ and $2$, respectively. There is then no
symmetry ambiguity in standardizing the vector labels.

\subsection{The $d=2$ Ising model and beyond}\label{subsec:8.2a} 

\subsubsection{Field theories from tensor categories}
 In specializing to $d=2$, we make the simplifying assumption that the spaces
$H_0$ and $H_2$ are $1$-dimensional. While motivated by our interest in gauge
theory $\cG_G$ and the extreme, Dirichlet and Neumann boundary conditions
associated to the subgroups $H=\{1\}$ and $H=G$ (\S\ref{subsec:3.1}), this
property really is characteristic of theories $\cF$ defined by fusion
categories $\mathscr{T}$ with boundary theories defined by indecomposable
module categories.\footnote {Larger spaces appear in the case of multi-fusion
$\mathscr{T}$.}  Indeed, recall from \S\ref{subsec:6.2} that the $2$-disk in
the theory $\cF_\mathscr{T}$ with boundary colored by $\mathscr{M}$ yields
the space $\mathrm{End}_\mathscr{E}(\mathbf{1})$, in the category
$\mathscr{E}=\mathrm{End}_\mathscr{T}(\mathscr{M})$, and $\mathbf{1}$ is
simple \cite[\S7.12]{EGNO}.  The components $\theta_{0,2}$ of index $0$ and
$2$ of the Ising action are just scalars, affecting the Ising partition
function by an overall factor; we will set them to $1$ and concentrate on
$\theta:=\theta_1$ and on the defects at the other critical points, placed at
vertices and faces of the original lattice.

The gauge theory example features a \emph{complementarity} of the Dirichlet
and Neumann boundary structures $\rB,\rB'$: reducing $\cF$ along an interval
with blue and red marked endpoints give the trivial $2$-dimensional theory, a
fancy way to say that $G$-bundles on the interval trivialized at one end but
free at the other are canonically trivialized.  This is not used for much the
discussion that follows, up to the point of duality. For a theory
$\cF_\mathscr{T}$ generated by a fusion category $\mathscr{T}$ it allows us
to Morita transform our quadruple $\cQ$ in two different ways, to one
containing the regular boundary condition and to one containing a fiber
functor.  It is in these settings that we use the language of Wilson and
't~Hooft defects, interchanged under the electromagnetic duality defined by
the Morita equivalence.

 \begin{remark}\label{remark8.6}
One can ask at this stage to what extent one can abstractly replicate the generating condition of 
Remark~\ref{gen}. Recall \cite{FT1} that $\cF$ is a functor from a bordism $3$-category to some 
symmetric monoidal $3$-category $\mathscr{C}$, and $\rB, \rB': \mathbf{1}\to \cF$ are (truncated) 
natural transformations from the unit functor. ``Generating'' 
means that the algebra objects $\mathscr{E}=\mathrm{End}_{\cF(\ptt)}\left(\rB(\ptt)\right)\in 
\mathrm{End}_\mathscr{C}
(\mathbf{1})$ and its primed version should each generate a TQFTs isomorphic to $\cF$, with equivalences 
induced by the bi-module objects $\rB,\rB'$, respectively.\footnote{Pictorially, $\cE, \cE'$ are 
represented by the interval in theory $\cF$ capped with the boundary condition $\rB,\rB'$ at 
both ends, and multiplication defined geometrically by a fat Y graph.} This statement requires interpretation: 
for general $\mathscr{C}$, the algebra objects $\mathscr{E}, \mathscr{E}'\in\mathrm{End}_\mathscr{C}
(\mathbf{1})$ need not have the same nature as the generating object $\cF(\ptt) \in\mathscr{C}$, 
and the field theories may only be compared on $\mathrm{Bord}_{<1,2,3>}$. 
\end{remark}

\begin{figure}[ht]
\centering
\includegraphics[scale=1]{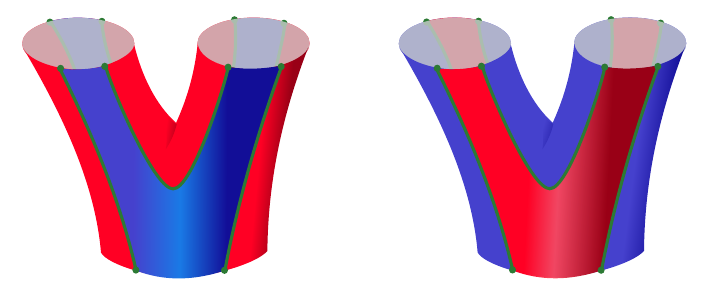}
\caption{The two multiplications (read down) and comultiplications (read up) on $H_1$.}\label{mult}
\end{figure}

\begin{figure}[ht]
\centering
\includegraphics[scale=1]{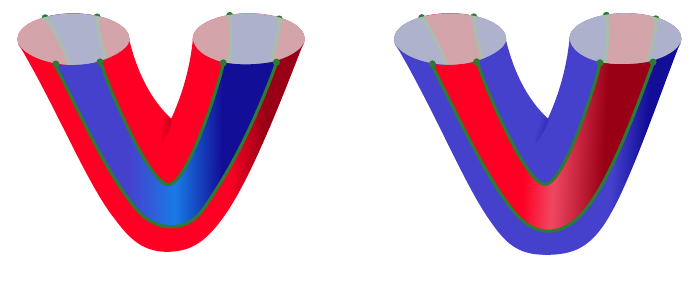}
\caption{The two Frobenius quadratic forms, differing by the antipode}\label{quadratics}
\end{figure}

\subsubsection{Relation to Hopf algebras} \label{computeH1} 
 We determine the gauge theory space $H_1 = \cQ(D_{1,1},\partial_cD^{1,1})$
from the classical model of $\cG_G$: the groupoid of bundles on $D^{1,1}$
trivialized on two arcs is equivalent to the set $G$ of monodromies, so $H_1$
is the space of functions on $G$.  (See also Example~\ref{thm:44}.)  This is
a Hopf algebra, with two Frobenius structures for the two operations, which
differ by the antipode. (The Frobenius form for the convolution algebra pairs
$g$ with $g^{-1}$, while multiplication is matched with the pointwise
pairing.) The Frobenius-Hopf operations have geometric interpretations, and
products can be interchanged with co-products by means of the Frobenius
forms, in a way that matches the switch from $H_1$ to its dual Hopf
algebra. The geometric construction of the operations is shown in
Figures~\ref{mult} and \ref{quadratics}. As the picture shows, the two
Frobenius quadratic forms on $H_1$ do not agree, but differ by a reflection
about the origin on (either) one of the factors: a quarter-rotation that
identifies the two sides of Figure~\ref{quadratics} takes opposite signs on
the two input boundaries. This reflection, the $\mathrm{O}(1)$-symmetry of
the handle, implements the antipode $S$.  For $\mathscr{V}ect[G]$ with
Dirichlet (red) and Neumann (blue) boundary structures, the left picture in
Fig.~\ref{mult} gives the group ring, while the right one defines the
commutative multiplication on $\bC[G]$.

This Frobenius-Hopf property is not an accident: every fusion category
$\mathscr{T}$ with a fiber functor $\phi:\mathscr{T}\to \mathscr{V}ect$ is
equivalent to the tensor category of modules over a finite semi-simple Hopf
algebra, by a Tannakian reconstruction theorem of Hayashi \cite{Ha} and
Ostrik \cite{Os}.  Moreover, the Koszul dual category
$\mathscr{E}:=\mathrm{End}_\mathscr{T}(\mathscr{V}ect)$ is the tensor
category of co-modules; equivalently (in light of finiteness) the category of
modules over the dual Hopf algebra. (The multi-fusion case is addressed by
the related notion of a weak Hopf algebra, see for instance
\cite[Ch.7]{EGNO}.) Thus, the most general setting of our story pertains to
finite, non-commutative gauge theory.

A full pictorial account of this TQFT-Hopf correspondence is planned for
\cite{FT2}; here, we merely check the requisite ``Peter-Weyl" decompositions
of $H_1$, which will ensure that we have the correct algebra in the Tannakian
reconstruction theorem: \begin{proposition} Let $x_i$ be a basis of simple
objects of $\mathscr{T}$ and $y_j$ one for $\mathscr{E}$. Then, as algebras
for the left and right multiplications in Fig.~\ref{mult}, \[ H_1\cong
\bigoplus\nolimits_i \phi(x_i)\otimes \phi(x_i)^\vee; \qquad H_1\cong
\bigoplus\nolimits_i \phi(y_i)\otimes \phi(y_i)^\vee.  \]
\end{proposition}
\begin{remark}
The two identifications are related by a $\pi/4$ rotation of $\partial_cD^{1,1}$, which implements the 
Fourier transform, interchanging $\theta \leftrightarrow \theta^\vee$ in the duality below. 
\end{remark}

\begin{figure}[ht]
\centering
\includegraphics[scale=1.2]{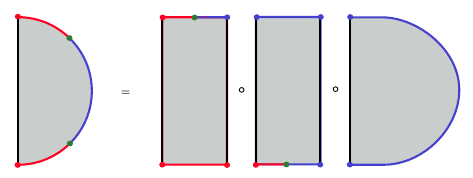}
\caption{The fiber functor}\label{factordisk}
\end{figure}

\begin{proof} The interval with two red endpoints computes, under $\cQ$, the
category $\mathscr{T}$ (cf.~Remark~\ref{thm:55}).  We claim that the
half-disk cut across the two red arcs represents the fiber functor; the
theorem then follows by pre-composing this with its adjoint. The claim
follows from the three-step factorization in Figure~\ref{factordisk}, which
splits it into the fiber functor $\phi$ landing in $\mathscr{V}ect\simeq
\mathrm{Hom}_\cF(\rB,\rB')$; its adjoint, the inclusion of
$\phi^\vee(\mathscr{V}ect)\subset\mathscr{E}$; and finally the application of
$\mathrm{Hom}_{\mathscr{E}}(\mathbf{1},\_\,)$.  Because $\phi^\vee(\bC) =
\bigoplus_j y_j\otimes \phi(y_j^\vee)$, the final two steps compose to the
fiber functor.
\end{proof}

\subsubsection{Order/disorder operators}
 Line defects in $3$D TQFT were described in \S\ref{subsec:3.3}, where
$\mathscr{V}ect[G]$ can be replaced by a general field theory $\cF$: the
cited discussion converts the ``line operator'' language into that of
bordisms with corners. We specialize to $0$- and $2$-defects.  In the theory
$\cF_\mathscr{T}$ with the regular boundary condition for $\rB$ and the fiber
functor for $\rB'$, these are the Wilson and 't~Hooft defects; for the Hopf
algebra $H_1$, these are the modules and co-modules, respectively. We now
explain how the promotion data in Definition~\ref{abstractorderdef} leads to
the defective Ising partition function (and correlators).

Repeat the cylinder construction of \S\ref{colorcollar}, with additional
holes bored through $X$ at each defect. This creates a white annulus on the
input face and a circular edge on the output $Y_1$. As we convened to insert
the unit at maxima and minima, the respective defect-free disks there may be
capped by colored hemispheres.  We read the picture with the white annulus
boundary as incoming, ready to absorb the defect label $\delta$. The other
annulus boundary is colored, and the result produces the vector space
$\mathrm{Hom}_\mathscr{T}(W_y,\delta)$ as a new tensor factor in the domain
of the map $\cQ(X,Y_c)$ in \eqref{colory}.  An \emph{order operator} (if
$p=0$) or \emph{disorder operator} ($p=2$) is a vector in the respective
$\mathrm{Hom}$-space at each defect point; contracted with $\cQ(X,Y_c)$, it
gives the Ising partition function in $\cF(Y_c,\vec{\delta}\,)$.  This is the
defective Ising partition function. A null-bordism of the defect maps it to
the correlator in $\cF(Y)$. The latter depends on the order/disorder
insertions as well as the null-bordism of the underlying defect.

\subsubsection{The Ising state space}
 Conversion to extended TQFT language indicates the way of factorizing Ising
theories (as relative theories) on manifolds with corners. The basic rule is
that any time-cut must be transversal to all structures.  For instance, we
can associate an object in the center category $Z(\mathscr{T})$ to any circle
cutting $Y_c$ through red and blue faces and across the green defect
lines. (We are unable to cut across the white disks, which have already
absorbed a vector that does not normally factor appropriately.)  When
translated to the lattice, this cutting seem counter-intuitive, as it
describes paths which zig-zag between vertices and face centers on the
lattice: neither of the obvious cutting methods --- along the edges or across
the edges --- is allowed, from the TQFT perspective.

The resulting circle comes subdivided into red and blue arcs. Pairing with an
object in $Z(\mathscr{T})$ yields an Ising space of states. Thus, in the
original Ising model, we thus get four spaces, one for each pair consisting
of a holonomy in $\bmuu_2$ and an irreducible $\bmuu_2$-representation. For
general $G$, we get one space for each conjugacy class and irreducible
representation of the stabilizer. These spaces also have natural inner
products, because of the reality and orientation-free nature of the theory.
This natural inner product seems to be an advantage of the TQFT cutting over
the obvious one (where correction by the square root of the action is
needed). The following relies on the identification
\cite[Proposition~7.14.6]{EGNO} of $Z(\mathscr{T})$ with the tensor category
of modules over the Drinfeld double algebra $D(H_1)$.

\begin{proposition} The Ising space of states for the circle with $N$
red/blue interval pairs is the free module $D(H_1)\otimes H_1^{\otimes
(n-1)}$ over the Drinfeld double, where $D(H_1)$~acts on the first factor
only.  \end{proposition} \begin{proof} The space, as a module over $D(H_1)$,
is the output of $[0,1]\times S^1$ with the circle at $0$ subdivided into the
$2N$ arcs. We compute it via the picture in Figure~\ref{isingspace}, by
factoring it successively via the categories
$\mathrm{Hom}(\rB,\rB')_\simeq\mathscr{V}ect$. The cutting lines for the
factorization are as indicated. The first cut yields the object $\bC$, which
is then tensored with the vector space $H_1$ at every step. The final
morphism $\mathrm{Hom}(\rB,\rB') \to Z(\mathscr{T})$ is the adjoint of the
fiber functor on $Z(\mathscr{T})$, and produces the double $D(H_1)$ as a
module over itself, times $H_1^{\otimes (N-1)}$.  \end{proof}

\begin{figure}[ht]
\centering
\includegraphics[scale=1.4]{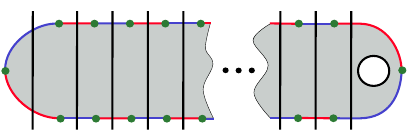}
\caption{Computation of the Ising space of states}\label{isingspace}
\end{figure}

\begin{remark} This presentation breaks the cyclic $\zmod N$ symmetry.  A
symmetric but less concrete presentation is given as follows. The center
$Z(\mathscr{T})$ can be identified with the $N$-fold cyclic tensor product \[
\mathscr{T}\otimes_\mathscr{T}\mathscr{T}\otimes_\mathscr{T}\dots
\otimes_\mathscr{T} \mathscr{T}\otimes_\mathscr{T}\rotatebox{90}{$\circlearrowleft$} \] and
receives as such an (additive) functor from $\mathscr{T}^{\otimes N}$. The
Ising space of states corresponds to the image of $\chi^{\otimes N}$, where
$\chi = \bigoplus_i x_i\otimes \phi(x_i^\vee)$ is the image of $\bC$ under
the adjoint of the fiber functor $\phi$.  \end{remark}

  \subsection{The duality theorem}\label{subsec:8.1}

The Morita equivalence between $\mathscr{T}$ and
$\mathscr{E}=\mathrm{End}_\mathscr{T}(\mathscr{V}ect)$ interchanges their
``red'' and ``blue'' boundary conditions: $\mathscr{T} \leftrightarrow
\mathscr{V}ect$ and $\mathscr{V}ect \leftrightarrow \mathscr{E}$. A matching
color switch on $Y_c$ is realized by a change to the Morse function $d-f$,
corresponding to the dual lattice. With the switch of boundary conditions
comes the switch in the type of defects and (dis)order operators. We have
seen above that for $\mathscr{T}= \mathscr{V}ect[G]$, the bi-colored
construction reproduces the gauge theory of \S\ref{sec:3} and \S\ref{sec:6}.
It follows that the color switch reproduces the theory $\cR_G$ of
\S\ref{sec:6}, proving, as advertised in the summary but now stated more
precisely:

\begin{theorem}\label{thm:8.13}
There is a natural duality, on oriented manifolds, identifying the gauge theory of a finite group $G$ 
with the Turaev-Viro theory based on $\mathscr{R}ep(G)$, and a matching Kramers-Wannier duality of their 
lattice boundary theories with Fourier dual Ising actions. There is an interchange of Wilson and 't~Hooft 
defects in the bulk theories, and 
of order and disorder operators for the boundary. Dual Ising partition match up, after adjusting an 
overall scale factor. \qed
\end{theorem}
\begin{remark}
The scale factor is concealed in the choice of $\mathrm{SO}(2)$-structures for the boundary theories.
The ``Fourier duality'' of the Ising actions is simply their transformation in the dual Hopf algebra 
using the Frobenius form: evenness allows us to use either of the two Frobenius pairings.  
Mind that both sides of the duality define \emph{unoriented} theories; however, as explained in the 
Abelian case, the duality does require an orientation, otherwise an orientation twist is introduced.
\end{remark}

The theorem, of course, applies to general fusion categories with a fiber
functor, corresponding to finite, semi-simple Hopf algebras $H$; the
categories of modules and co-modules are then interchanged.  In this
generalization, we leave the setting of theories quantized from a classical
model of fields, but we can find new examples of \emph{self-dual} theories
\cite{Ma}, which is not possible in non-abelian gauge theory.

   \section{Higher dimensions and finite path integrals}\label{sec:7}
% lastsubsec@  2

We generalize the two-dimensional lattice model for finite \emph{abelian}
groups to higher dimensions.  The background field encoding the higher global
symmetry is a higher abelian gauge field, and the abelian case admits a
natural generalization to \emph{spectra} (in the sense of stable homotopy
theory) which are finite homotopy types: only finitely many homotopy groups
are nonzero and each is a finite group.  These lattice models have arbitrary
``higher groups'' of this type as symmetry groups.  The construction is
phrased in terms of \emph{finite path integrals}; see~\cite{Q,F3,Tu,FHLT} for
various expositions.  All spaces and spectra in this section are finite
homotopy types.

  \subsection{Generalized finite electromagnetic duality}\label{subsec:7.1}

Suppose $S$~is a pointed space which is a finite homotopy type.  The example
relevant to~\S\ref{sec:3} is~$S=BG$ for $G$~a finite group.  Fix $n\in
\ZZ^{\ge0}$.  There is an $n$-dimensional extended topological field
theory~$\rF_S$ of unoriented manifolds which counts homotopy classes of maps
to~$S$.  Namely, if $X$~is a closed $n$-manifold, let $S^X=\Map(X,S)$ denote
the space of unpointed maps from~$X$ to~$S$.  The partition function is
  \begin{equation}\label{eq:76}
     \rF_S(X) = \sum\limits_{[\varphi ]\in \pi _0S^X} \frac{1}{\#\pi
      _1(S^X,\varphi )} \;\frac{\#\pi _2(S^X,\varphi )}{\#\pi
      _3(S^X,\varphi )} \cdots 
  \end{equation}
The sum is over homotopy classes of maps $X\to S$ and $\varphi $~is a
representative of a homotopy class.  For~$S=BG$ expression~\eqref{eq:76}
reduces to~\eqref{eq:9}.  The weighted sum accounts for automorphisms~$(\pi
_1)$, automorphisms of automorphisms~$(\pi _2)$, etc. 

  \begin{example}[]\label{thm:37}
 If $A$~is a finite \emph{abelian} group and $S=B^rA$ is the
$r^{\textnormal{th}}$~classifying space---the Eilenberg-MacLane space with
$\pi _rS=A$---then \eqref{eq:76}~reduces to 
  \begin{equation}\label{eq:77}
     \rF_{B^rA}(X) = \prod\limits_{i=0}^r\left( \#H^{r-i}(X;A) \right)
     ^{(-1)^i}. 
  \end{equation}
An $(n-1)$-dimensional theory relative to~$\rF_{B^rA}$ has a generalized
symmetry group~$B^{r-1}A$, sometimes called an ``$(r-1)$-form global
symmetry''~\cite{GKSW}.
  \end{example}

  \begin{remark}[]\label{thm:38}
 We can twist the sum~\eqref{eq:76} by a suitable cocycle\footnote{The
twisting cocycle can lie in a generalized cohomology theory, and it may also
depend on the topology of~$X$.  See~\cite{D} for an instance of the latter in
the low energy effective field theory attached to the double semion model.}
on~$S$.  For example, Dijkgraaf-Witten theory is the case $S=BG$ with cocycle
representing a class in $H^n(BG;\RR/2\pi i\ZZ)$.  An $(n-1)$-dimensional
theory relative to it is said to have an anomalous global symmetry
group~$G$.  Note that we can replace the mapping space $\Map(X_+,BG)$---or
rather its fundamental groupoid, which is all we use---by the equivalent
groupoid of principal $G$-bundles over~$X$.  With that understood the theory
of maps to~$BG$ is a gauge theory.  A similar remark holds for higher groups,
as in Example~\ref{thm:37}.
  \end{remark}

  \begin{example}[]\label{thm:43}
 Let $\bmu n\subset \TT$ be the group of $n^{\textnormal{th}}$~roots of
unity.  Then there is a pointed space~$S$ which fits into the (Postnikov)
principal fibration 
  \begin{equation}\label{eq:81}
     B^2\!\bmu n\longrightarrow S\longrightarrow B\!\bmu n 
  \end{equation}
and has $k$-invariant a generator of $H^3(B\!\bmu n;\bmu n)\cong \bmu n$.  We
can view $S$~as (the classifying space of) a 2-group; this 2-group acts as
global symmetries on any field theory relative to~$\rF_S$.  (We remark that
this is a finite analog of a 2-group which acts on certain 4-dimensional
gauge theories~\cite{CDI}.)  The $k$-invariant is unstable so this example
does not extend to a spectrum.  Still, it is possible that there is an
electromagnetic dual.
  \end{example}

 The finite path integral constructs an \emph{extended} $n$-dimensional
topological field theory: a vector space for a closed $(n-1)$-manifold, a
linear category or algebra for a closed $(n-2)$-manifold, etc.  See~\cite{F3}
for a heuristic treatment, \cite[\S3]{FHLT} for examples, and
\cite[\S8]{FHLT} for a general discussion.  Here we are content to tell the
result of the finite path integral in codimensions one and two. 
 
Let $Y$~be a closed $(n-1)$-manifold.  The vector space~$\rF_S(Y)$ is
simply~$\Fun(\pi _0S^Y)$, the space of functions on the set~$[Y_+,S]$ of
(free) homotopy classes of maps $Y\to S$.  If we twist by a cocycle
representing a class $\alpha \in H^n(S;\RR/2\pi i\ZZ)$, as in
Remark~\ref{thm:38}, then for $Y$~oriented we construct a line bundle
$p\:\sL\to \pi _{\le1}S^Y$ over the fundamental groupoid of $\Map(Y_+S)$.
(See \cite[Appendix~B]{FQ} for a special case.)  Then $\rF_S(Y)$~is the vector
space of sections of~$p$.  Observe that a loop in~$\pi _{\le1}S^Y$---an
automorphism of a map $\varphi \:Y\to S$---is a map $f\:\cir\times Y\to S$.
The holonomy of $\sL\to\pi _{\le1}S^Y$ around the loop is $\exp(\langle
f^*\alpha ,[\cir\times Y]  \rangle)\in \TT$.  If it is nontrivial for any
loop based at~$\varphi $, then all sections of~$p$ vanish at~$\varphi $. 
 
To a closed $(n-2)$-manifold~$Z$ we attach the linear category~$\rF_S(Z)$ of
complex vector bundles over~$\pi _{\le1}S^Z$, the fundamental groupoid of
the space of maps $Z\to S$.  This is the category of representations of the
path algebra of a finite groupoid equivalent to~$\pi _{\le1}S^Z$.  In the
presence of a nonzero cocycle we obtain a category of twisted vector bundles,
or equivalently the category of representations of a twisted path algebra;
see~\cite[\S8]{F3} for an example.

Now suppose that $S$~has an \emph{infinite loop space} structure, i.e., a
choice of spaces~$T_m$, $m\in \ZZ^{\ge0}$, and homotopy equivalences $S\to
\Omega ^mT_m$ which exhibit~$S$ as an $m$-fold loop space for every~$m\in
\ZZ^{\ge0}$.  Equivalently, a sequence $\{T_m\}_{m\in \ZZ^{\ge0}}$ of pointed
spaces and homotopy equivalences $T_m\to\Omega T_{m+1}$ defines a
\emph{spectrum}~$\sT$ and exhibits $S=T_0$ as the 0-space of the spectrum.
For example, $S=B^rA$ is the 0-space of the shifted Eilenberg-MacLane
spectrum~$\Sigma ^rHA$.  Spectra give homology and cohomology theories, and
vice versa; we write $H_{\bullet }(X;\sT)$, $H^{\bullet }(X;\sT)$ for the
homology, respectively cohomology, of a space~$X$ with coefficients in the
spectrum~$\sT$.  A spectrum~$\sT$ has a \emph{character dual}\footnote{If we
replace~$\TT$ with~$\QQ/\ZZ$ we obtain the \emph{Brown-Comenetz
dual}~\cite{BC}.  The character dual is usually defined with~$\Cx$ in place
of~$\TT$, but for finite homotopy types the circle group is sufficient.  Note
that $\TT, \QQ/\ZZ, \Cx$ all have the discrete topology in this context.}
spectrum~$\sT\dual$ defined as the cohomology theory
  \begin{equation}\label{eq:92}
     H^q(X;\sT\dual):=\Hom\bigl(H_q(X;\sT),\TT \bigr) = H_q(X;\sT)\dual. 
  \end{equation}
The dual to the sphere spectrum is denoted~$I\TT$, and then
$\sT\dual=\Map(\sT,I\TT)$.  Take $X=\pt$ in~\eqref{eq:92} to conclude that
the homotopy groups of~$\sT\dual$ are the Pontrjagin duals to the homotopy
groups of~$\sT$:
  \begin{equation}\label{eq:78}
     \pi _q\sT\dual\cong \left( \pi _{-q}\sT \right) \dual,\qquad q\in \ZZ. 
  \end{equation}
$\sT\dual$~has finite homotopy type if and only if $\sT$~does. 

  \begin{definition}[]\label{thm:40}
 Let $\rF_{\sT}\:\Bord_n\to\sC$ be the finite path integral theory based on the
spectrum~$\sT$.  The \emph{electromagnetic dual theory}
$(\rF_{\sT})\dual=\rF_{\Sigma ^{n-1}\sT\dual}$ is the finite path integral theory
based on the shifted character dual spectrum~$\Sigma ^{n-1}\sTd$.
  \end{definition}

The abelian duality of~\S\ref{subsec:3.2} is the case $n=3$, $\sT=\Sigma HA$,
for a finite abelian group~$A$.  The correspondence diagram~\eqref{eq:55} of
pointed spaces, which expresses the duality, generalizes to arbitrary spectra
as the commutative diagram of spectra 
  \begin{equation}\label{eq:79}
     \begin{gathered} \xymatrix{&(\sT\times \Sigma ^{n-1}\sTd,c
     )\ar[dl]_p\ar[dr]^q\\\sT&&\Sigma ^{n-1}\sTd} \end{gathered} 
  \end{equation}
in which 
  \begin{equation}\label{eq:80}
     c\:\sT\times \Sigma ^{n-1}\sTd\longrightarrow \Sigma ^{n-1}I\TT 
  \end{equation}
is the canonical duality map, which here represents an $(n-1)$-dimensional
domain wall.  Let $T_0$~be the 0-space of~$\sT$ and $T_{n-1}\dual$~the
0-space of~$\Sigma ^{n-1}\sTd$.  Then if $X=X'\cup_Y X''$ is an $n$-manifold
decomposed into the union of submanifolds~$X',X''$ with boundary intersecting
in the $(n-1)$-dimensional submanifold~$Y$, a field in the theory with domain
wall is a pair~$(\varphi ',\varphi '')$ of maps $\varphi '\:X'\to T_0$ and
$\varphi ''\:X''\to T_{n-1}\dual$.  On the restriction to~$Y$ we obtain a
representative of a class in $H^{n-1}(Y;I\TT)$.  As in~\S\ref{subsec:3.2} we
can use it to define a Fourier transform between the vector spaces associated
to~$Y$ in the theories~$\rF_{\sT}$ and~$(\rF_{\sT})\dual$.

  \begin{remark}[]\label{thm:41}
 At first glance to integrate this class over~$Y$ requires a framing on~$Y$,
since the cohomology theory~$I\TT$ is oriented for framed manifolds.  But the
duality map~\eqref{eq:80} may factor through a simpler cohomology theory with
a less stringent orientation condition.  For example, the duality map
in~\eqref{eq:55} takes values in~$H\TT$, so can be integrated over
(standardly) oriented manifolds. 
  \end{remark}

  \begin{remark}[]\label{thm:42}
 In the simplest case, $\sT=\Sigma ^rHA$ is the shifted Eilenberg-MacLane
spectrum of a finite abelian group~$A$.  Then $\varphi \:X\to T_0$ represents
a class in $H^r(X;A)$.  Electromagnetic duality is well-known for such higher
abelian gauge fields.
  \end{remark}

  \subsection{Boundary theories and domain walls}\label{subsec:7.2}
 
Let $S,B$~be pointed spaces, both finite homotopy types, and suppose $\pi
\:B\to S$ is a fibration.  From this data we construct a boundary theory~$\rB
_B$ for~$\rF_S$.  A field on a manifold~$X$ with boundary is $\varphi \:X\to
S$ as before, together with a lift $\psi \:\partial X\to B$ of the
restriction of~$\varphi $ to~$\partial X$.  So, for example, the value
of~$\rB _B$ on a closed $(n-1)$-manifold~$Y$ is the function of the
field~$\varphi \:Y\to S$ defined by counting lifts $\psi \:Y\to B$
of~$\varphi $ up to homotopy.  Let $\BYf$~denote the space of lifts.  Then we
have
  \begin{equation}\label{eq:82}
     (\rF_S,\rB _B)(Y)(\varphi ) = \sum\limits_{[\psi  ]\in \pi _0\BYf}
     \frac{1}{\#\pi 
      _1(\BYf,\psi )} \;\frac{\#\pi _2(\BYf,\psi )}{\#\pi
      _3(\BYf,\psi )} \cdots ,
  \end{equation}
where we sum over homotopy classes of lifts.  The result only depends on the
homotopy class of~$\varphi $.
 
For any~$S$ there are two canonical boundary theories.  The first is
associated to the inclusion $*\to S$ of the basepoint, made into a fibration
by replacing~$*$ with the contractible space~$P(S)$ of paths $\gamma
\:[0,1]\to S$ with $\gamma (0)=*$.  We call this the \emph{Dirichlet
boundary theory}.  The second is associated to the identity map $S\to S$ and
is termed the \emph{Neumann boundary theory}.  For finite gauge theory $S=BG$
the field~$\varphi $ may be replaced by a principal $G$-bundle
(Remark~\ref{thm:38}), in which case the Dirichlet boundary field is a
trivialization of the $G$-bundle and, as in general, the Neumann boundary
field is trivial.  These boundary theories are denoted~$\rB _e$ and~$\rB
_G$, respectively, in~\eqref{eq:11}.   
 
If pointed spaces~$S_1,S_2$ of finite homotopy type define theories
$\rF_{S_1},\rF_{S_2}$, then a correspondence diagram 
  \begin{equation}\label{eq:83}
     \begin{gathered} \xymatrix{&D\ar[dl]\ar[dr]\\S_1&&S_2} 
     \end{gathered}
  \end{equation}
of finite homotopy type pointed spaces defines a domain wall $\rD
_D\:\rF_{S_1}\to \rF_{S_2}$.  Similarly, a commutative diagram 
  \begin{equation}\label{eq:84}
     \begin{gathered} \xymatrix{&D\ar[dl]\ar[dr]\\B\ar[dr]&&B'\ar[dl]\\
     &S} \end{gathered} 
  \end{equation}
determines a domain wall between the boundary theories~$\rB _{B}\rB
_{B'}$.  There is a canonical domain wall between the Dirichlet and Neumann
boundary theories for any~$S$. 

  \begin{example}[]\label{thm:44}
 Fix~$n=3$ and~$S=BG$ for $G$~a finite group.  We use the canonical
Dirichlet (red) and Neumann (blue) boundary theories and the canonical
domain wall (green) between them.  Let~$Y$ be a closed 2-manifold.  Then
$\rF_S(Y)$~is the space of functions on~$\bung Y$, as in~\eqref{eq:10}.  If we
``color''~$Y$ red, so evaluate the pair~$(\rF_{BG},\rB
_{\textnormal{Dirichlet}})$ on~$Y$, we obtain a vector in~$\rF_S(Y)$---a
function on~$\bung Y$---by counting trivializations: the value of this
function on a principal $G$-bundle $Q\to Y$ is zero if $Q\to Y$ is not
trivializable and is $(\#G)^{\#\pi _0Y}$ if it is trivializable.  If instead
we color~$Y$ blue, then we deduce that $(\rF_{BG},\rB
_{\textnormal{Neumann}})(Y)$~is the constant function with value~1.

  \begin{figure}[ht]
  \centering
  \includegraphics[scale=1]{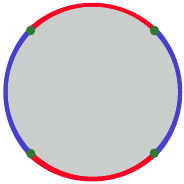}
  \caption{A 2-disk with colored boundary}\label{fig:3}
  \end{figure}

For more exercise we evaluate the quartet~$(\rF_{BG},\rB
_{\textnormal{Dirichlet}},\rB _{\textnormal{Neumann}},\rD
_{\textnormal{canonical}})$ on the 2-dimensional disk~$Y$ which is pictured
in Figure~\ref{fig:3}.  The boundary is divided into two sets of alternating
blue and red intervals connected by a total of 4~green points.  The quartet
is associated to the diagram 
  \begin{equation}\label{eq:85}
     \begin{gathered} \xymatrix{&\**\ar[dl]\ar[dr]\\\**\ar[dr]&&BG\ar[dl]\\
     &BG} \end{gathered} 
  \end{equation}
The value of the quartet on~$Y$ is a vector space, as it is for a closed
uncolored surface.  The finite path integral description tells that it is the
vector space of functions on path components of fields on~$Y$, which are
$G$-bundles trivialized over the red and green portions of the boundary.  The
result is isomorphic to the space of functions on~$G$.  This is used
in~\S\ref{subsec:8.2a}.
  \end{example}

For any finite homotopy type~$S$, the Dirichlet and Neumann boundary
theories, together with the canonical domain wall between them, can be input
to the constructions of~\S\ref{sec:8} to yield lattice theories.

%\appendix
\providecommand{\bysame}{\leavevmode\hbox to3em{\hrulefill}\thinspace}
\providecommand{\MR}{\relax\ifhmode\unskip\space\fi MR }
% \MRhref is called by the amsart/book/proc definition of \MR.
\providecommand{\MRhref}[2]{%
  \href{http://www.ams.org/mathscinet-getitem?mr=#1}{#2}
}
\providecommand{\href}[2]{#2}

  \end{document}